\documentclass[a4paper, 10pt, english, reqno]{amsart}
\usepackage[utf8]{inputenc}
\usepackage{fullpage}
\usepackage{graphicx}
\usepackage{amsmath,amssymb, amsfonts, amsthm, mathtools}
\usepackage{mathrsfs}
\usepackage{xcolor}
\usepackage{hyperref}
\usepackage{algorithm}
\usepackage{algpseudocode}
\usepackage{subcaption}
\usepackage{float}
\usepackage{thmtools}
\usepackage{empheq}
\usepackage{cleveref}
\usepackage{booktabs}

\DeclareMathOperator*{\Law}{\text{Law}}

\def\R{\mathbb R}

\newcommand{\tnorm}[1]{\left\vert\kern-0.25ex\left\vert\kern-0.25ex\left\vert #1
\right\vert\kern-0.25ex\right\vert\kern-0.25ex\right\vert}

\newtheorem{theorem}{Theorem}[section]
\newtheorem{lemma}[theorem]{Lemma}
\newtheorem{assumptions}[theorem]{Assumption}
\newtheorem{definition}[theorem]{Definition}
\newtheorem{proposition}[theorem]{Proposition}
\newtheorem{remark}[theorem]{Remark}

\newcounter{Acond}

\crefname{Acond}{Assumption}{Assumption}

\title{Mean-Field PhiBE: Continuous-Time Mean-Field Reinforcement Learning from Discrete-Time Data}

\author[E. Bayraktar]{Erhan Bayraktar\textsuperscript{1}}
\thanks{\textsuperscript{1} Department of Mathematics, University of Michigan, Ann Arbor, MI, USA}

\author[M. Hernandez]{Martin Hernandez\textsuperscript{2,*}}
\thanks{\textsuperscript{2} Department of Statistics and Data Science, University of California, Los Angeles, CA, USA}
\author[Q. Yan]{Qinxin Yan\textsuperscript{3}}
\thanks{\textsuperscript{3} Program in Applied and Computational Mathematics, Princeton University, Princeton, NJ, USA}

\author[Y. Zhu]{Yuhua Zhu\textsuperscript{2}}

\thanks{\textsuperscript{*} Corresponding author (\texttt{martinh@ucla.edu}).}

\date{}

\begin{document}

\begin{abstract}
This paper develops a model-free framework for continuous-time mean-field control when the population evolves according to unknown controlled McKean--Vlasov dynamics and only discrete-time transition data are available. Model-based mean-field control requires the continuous-time drift and diffusion coefficients, which are not directly observed from fixed-step transitions, while a direct reduction to a discrete-time Bellman equation loses the continuous-time generator structure. To bridge these two viewpoints, we introduce a Mean-Field-PhiBE (MF-PhiBE), which incorporates discrete-time transition information into a continuous-time PDE on the Wasserstein space. The MF-PhiBE replaces the unknown infinitesimal drift and covariance in the policy-evaluation equation by one-step estimators computed from data, while preserving the generator structure of the McKean-Vlasov HJB equation. We also derive a policy-gradient theorem for entropy-regularized randomized feedback policies, expressing the actor direction through an action-wise infinitesimal advantage and the score of the policy. Combining these two ingredients yields a model-free actor-critic method. We prove a first-order consistency estimate showing that the value induced by an optimal MF-PhiBE policy approximates the optimal continuous-time value as the observation time step vanishes. For entropy-regularized LQR, we establish first-order policy convergence and second-order value convergence; under suitable conditions, the population-averaged feedback means coincide exactly. Numerical experiments on an LQR benchmark and a crowd-aversion problem illustrate the proposed framework.
\end{abstract}

\maketitle


\smallskip


\medskip

\setcounter{tocdepth}{1}
\tableofcontents

\section{Introduction}
Many modern control problems arise from systems that evolve continuously in time, while data are available only at discrete observation times. This situation appears naturally in physical and engineered systems, where the state continues to evolve between measurements and interventions \cite{lin1998optimal,song2023reaching,pham2009continuous}. Continuous-time reinforcement learning (CTRL) addresses this setting by combining reinforcement learning with stochastic control models driven by differential equations; see, for instance, \cite{wang2020,jia2022,jia2023,jia_qlearning2023}. A direct reduction to a Markov decision process (MDP) is often computationally convenient, but it does not use the differential structure of the underlying dynamics. Consequently, even with arbitrarily many samples, the learned policy may retain a discretization error caused by the time-discrete formulation itself. This issue is one of the main motivations behind PDE-based approaches such as PhiBE and Optimal-PhiBE, which use discrete-time transition data while preserving the smoothness structure of the physical system \cite{zhu2024phibe,zhu2025optimalphibe}. This distinction is illustrated in \Cref{fig:value-functions-intro}, showing how the continuous-time value obtained from our framework remains much closer to the optimal value than the policy obtained after reducing the problem to a time-discrete Bellman formulation (MDP). 

The importance of continuous-time structure becomes even more pronounced in mean-field control. Mean-field control, also known as MKV control, describes the optimal control of a large population of cooperative agents through the evolution of its probability law; see \cite{carmona2018probabilistic1,carmona2018probabilistic2,benbook,Lacker}. The relevance of this structure becomes clear when one compares the model-based and the model-free settings. If the MKV coefficients are known, the value of a fixed policy can be evaluated through a stationary linear equation on $\mathcal P_2(\mathbb R^d)$, whose generator depends on the drift and diffusion coefficients and on Lions derivatives with respect to the population law \cite{CarmDel15,PhamWei}. Consequently, this characterization cannot be used directly in the model-free setting, where the MKV coefficients are unknown.

We consider an entropy-regularized infinite-horizon mean-field control problem with randomized feedback policies. Entropy regularization promotes exploration and is widely used in continuous-time reinforcement learning \cite{wang2020,jia2022,jia2023} and in continuous-time mean-field actor-critic methods \cite{frikhaActorCritic,Wei2023ContinuousTQ}. In the model-based setting, the corresponding actor-critic structure is determined by two objects: a stationary linear HJB equation on $\mathcal P_2(\mathbb R^d)$ for policy evaluation and a policy-gradient representation for policy improvement. Our focus is instead on the model-free regime, in which the drift and diffusion coefficients of the MKV dynamics are unknown and the learner has access only to discrete-time samples of states, population laws, actions, subsequent states, and rewards. The central question is how to approximate the continuous-time policy-evaluation and policy-improvement mechanisms from these observations without replacing the original control problem by a time-discrete MDP. The following contributions address this question.
\smallbreak

\paragraph{Contributions}

\noindent

\textbf{1. Policy gradient for continuous-time McKean--Vlasov control.} We extend the policy-gradient framework developed in our previous work for the representative-state/population formulation \cite{firstpaper} to the law-based social-planner formulation considered here. The resulting policy gradient takes an advantage-times-score form analogous to the classical policy-gradient formula in reinforcement learning, with the advantage characterized by an action-wise infinitesimal advantage function. This formulation provides a policy-improvement mechanism tailored to McKean--Vlasov dynamics.
\smallbreak

\textbf{2. MF-PhiBE from discrete-time transition dynamics.}
When only discrete-time transition dynamics are available, we introduce MF-PhiBE, which embeds this information into a continuous-time mean-field control formulation while preserving the McKean--Vlasov structure and avoiding the generally unidentifiable continuous-time drift and diffusion terms. We establish a first-order error bound for the value achieved by the optimal MF-PhiBE policy. For entropy-regularized LQR, we establish first-order policy convergence and second-order value convergence.  In particular, under suitable conditions, the population-averaged feedback means of the MF-PhiBE and true optimal policies coincide exactly.
\smallbreak


\textbf{3. Model-free algorithm from discrete-time data.}
When only discrete-time transition samples are available, we develop a model-free actor--critic algorithm based on the MF-PhiBE approximation. Both the critic and the policy-gradient direction are estimated directly from one-step transition data, yielding an implementable learning procedure that requires neither knowledge nor estimation of the continuous-time McKean--Vlasov drift and diffusion coefficients.

\begin{figure}
    \centering
    \includegraphics[width=0.95\linewidth]{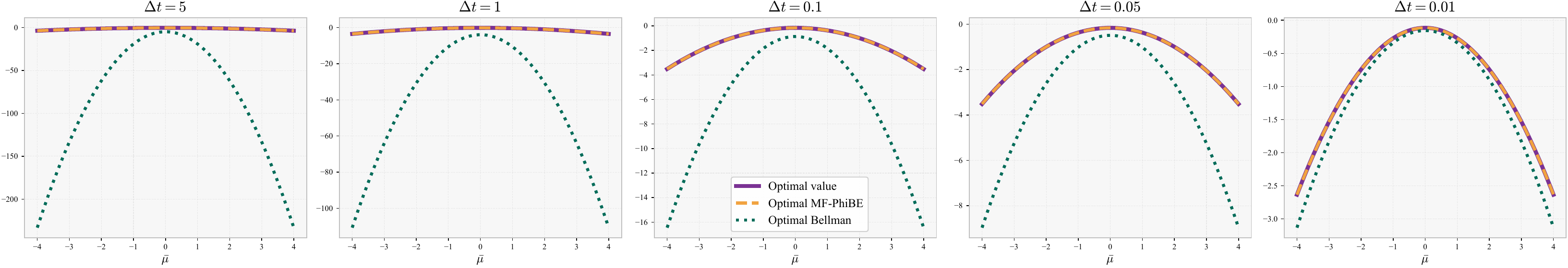}
   \caption{Illustration of the continuous-time value function obtained from different feedback policies. For the same time-step information, the policy produced by our framework provides a substantially closer approximation of the optimal continuous-time value function than the policy obtained from the time-discrete Bellman problem.}
    \label{fig:value-functions-intro}
\end{figure}

\smallskip
\noindent\textbf{Related work.}  
Continuous-time reinforcement learning has been studied through stochastic-control, policy-iteration, martingale, $q$-learning, policy-gradient, and actor--critic approaches \cite{doya2000,Munos2006PolicyGradient,VamvoudakisLewis2010,LeeSutton2021,wang2020,jia2022,jia2023,jia_qlearning2023,ZHAOPolicyOptimization}. Approaches designed to exploit continuous-time structure from discretely sampled transitions include \cite{KimShinYang2021,SettaiTakeishiYairi2025}, while sampling and time-discretization effects are analyzed in \cite{tallec2019,JiaOuyangZhang2025,reconciling2026,PhamZhangZhu2026}. Existing extensions to mean-field control broadly follow two directions. The first direction formulates both the mean-field dynamics and the learning problem directly in discrete time. Mean-field MDP formulations are studied in \cite{motte2022mean,bauerle2023mean}. Dynamic-programming and value-based learning methods are developed in \cite{angiuli2022unified,carmona2023model,gu2021,gu2023}, while finite approximations and $Q$-learning for mean-field-type multi-agent control are considered in \cite{BayraktarKaraBauerle}. Function-approximation methods are analyzed in \cite{BayraktarKaraJMLR}, and \cite{meunierpham2026} extends the classical REINFORCE algorithm to discrete-time mean-field control. By contrast, our method uses discrete-time observations to approximate the infinitesimal generator of an underlying continuous-time McKean--Vlasov control problem. 

The second direction preserves the continuous-time control structure through policy-gradient, actor--critic, and q-learning methods for mean-field control \cite{frikhaActorCritic,Wei2023ContinuousTQ}. These methods retain the continuous-time McKean--Vlasov structure when the dynamics are known or continuous-time observations are available. When only discrete-time data are available, however, their model-free implementations rely on temporal discretization and therefore lose the continuous-time generator structure. In addition, the model-free actor update in \cite{frikhaActorCritic} is directly implementable only for special dynamics, while \cite{Wei2023ContinuousTQ} requires a restrictive parametrization of the essential (q)-function and the estimation of an integrated (q)-function through a family of test policies. By contrast, our model-free algorithm uses MF-PhiBE to preserve the continuous-time generator structure even when only discrete-time data are available. Moreover, it applies to general McKean--Vlasov dynamics and requires neither an integrated (q)-function nor martingale conditions over test policies.

The key challenge addressed in this paper is how to exploit the smooth structure of continuous-time dynamics when only discrete-time information is available. 
The PhiBE frameworks \cite{zhu2024phibe,zhu2025optimalphibe} address this difficulty in the single agent setting: rather than replacing the continuous-time problem by a discrete-time Bellman equation, they use one-step transition information to approximate the infinitesimal generator and quantify the resulting time-discretization error. The present paper extends this generator-based approach to MKV control.

The entropy-regularized LQ mean-field problem is analyzed in \cite{frikha2024full}, where common noise is included, the optimal Gaussian policy is characterized through coupled Riccati equations, and policy-gradient algorithms are studied in both exact and model-free settings. Our general framework is not restricted to the LQR setting. The LQR problem is used instead as an analytically tractable benchmark for MF-PhiBE.

Finally, our work is related to the broader analytical and numerical literature on mean-field control and PDEs on Wasserstein space. Comparison principles and viscosity-solution theory for second-order equations on the Wasserstein space are developed in \cite{BayraktarEkrenZhangCPDE2025,BayraktarEkrenHeZhangJDE2026}, while particle approximations are analyzed in \cite{BayraktarEkrenZhangParticleSICON2025}. Numerical methods include PDE and augmented-Lagrangian schemes \cite{AchdouLauriereCongestionI,AchdouLauriereCongestionII}, FBSDE methods \cite{chassagneux2019}, Markov-chain approximations and asymptotic methods for related queueing mean-field games \cite{BayraktarBudhirajaCohenAAP,BudhirajaCohenBayraktar2018}, mean-field with common-noise \cite{RenWeiYuZhouCommonNoiseI,RenWeiYuZhouCommonNoiseII}, and neural-network-based solvers \cite{ruthotto2020,Carmona2022,soner2025learning}; see also \cite{lauriere2021} for a survey. These works do not address the setting considered here: continuous-time MKV dynamics with unknown coefficients, observed only through discrete-time transition data. 


\section{Problem formulation}
\label{sec:formulation}

Let $d,n,m\in\mathbb N$. The state space is $\mathbb R^d$, the action space is a Borel set $\mathcal A\subseteq \mathbb R^m$, and the population state is described by a measure in $\mathcal P_2(\mathbb R^d)$, the set of Borel probability measures on $\mathbb R^d$ with finite second moment. Let $\nu$ be a $\sigma$-finite reference measure on $(\mathcal A,\mathcal B(\mathcal A))$. We equip $\mathcal P_2(\mathbb R^d)$ with the $2$-Wasserstein distance
\[
  \mathcal W_2(\mu,\varrho)
  :=
  \inf_{\gamma\in\Gamma(\mu,\varrho)}
  \left(
    \int_{\mathbb R^d\times\mathbb R^d}
    |x-y|^2\,\gamma(dx,dy)
  \right)^{1/2}.
\]

In the present work, the control at a state-law pair $(x,\mu)$ is allowed to be randomized. Thus a policy assigns to each $(x,\mu)\in\mathbb R^d\times\mathcal P_2(\mathbb R^d)$ a probability measure on the action space, and is represented by a Markov kernel
\[
  \pi:\mathbb R^d\times\mathcal P_2(\mathbb R^d)
  \longrightarrow
  \mathcal P(\mathcal A),
  \qquad
  (x,\mu)\longmapsto \pi(\cdot\mid x,\mu).
\]
We define $\Pi_{\mathrm{add}}$ as the class of randomized feedback policies satisfying $\pi(\cdot\mid x,\mu)\ll\nu$ for every $(x,\mu)\in\mathbb R^d\times\mathcal P_2(\mathbb R^d)$. For $\pi\in\Pi_{\mathrm{add}}$, the density with respect to $\nu$ is denoted by $p_\pi(a\mid x,\mu)$, and is assumed to be jointly measurable in $(x,\mu,a)$ and strictly positive for $\nu$-a.e. $a\in\mathcal A$ and every $(x,\mu)$.

We consider an infinite-horizon discounted mean-field optimal control
problem. Starting from an initial condition
$\mu\in\mathcal P_2(\mathbb R^d)$, a social planner seeks to optimize
the collective performance of the population over time. This leads to
the value function
\begin{align}\label{eq:optimal_CP_mkv}
  \mathcal V^*(\mu)
  \,:=\,
  \sup_{\pi\in\Pi_{\mathrm{add}}}\mathcal V^\pi(\mu),
  \qquad
  \mu\in\mathcal P_2(\mathbb R^d).
\end{align}
Whenever there exists a policy $\pi^*\in\Pi_{\mathrm{add}}$ such that
$\mathcal V^*(\mu)=\mathcal V^{\pi^*}(\mu)$ for all
$\mu\in\mathcal P_2(\mathbb R^d)$, we call $\pi^*$ an optimal policy. Here $\mathcal V^\pi(\mu)$ denotes the value associated with the policy
$\pi$, when the system starts from the
initial law $\mu\in\mathcal P_2(\mathbb R^d)$. More precisely, the
corresponding value is defined by
\begin{align}\label{eq:value_fn_mkv}
  \mathcal V^\pi(\mu)
  \,:=\,
  \mathbb E\!\left[
    \int_0^\infty e^{-\beta t}
    r_\lambda^\pi(\mu_t)\,dt
    \;\bigg|\;
    \mu_0=\mu
  \right],
\end{align}
where $\beta>0$ is the discount factor, $r_\lambda^\pi$ is the running
reward, and $(\mu_t)_{t\ge0}$ denotes the controlled population state. The running reward in \eqref{eq:value_fn_mkv} corresponds to the
entropy-regularized functional
\begin{equation}\label{eq:reg_reward_avg_mkv}
  \begin{aligned}
    r_\lambda^\pi(\mu)&:=
\int_{\mathbb R^d}
\Bigl(r^\pi(x,\mu)+\lambda\,\mathcal H\bigl(\pi(\cdot\mid x,\mu)\bigr)\Bigr)\,\mu(dx).\\
    &:= \int_{\R^d}\int_{\mathcal A} \Bigl(r(x,\mu,a)
    -\lambda
    \log p_\pi(a\mid x,\mu)\Bigr)\pi(da\mid x,\mu)\mu(dx)
  \end{aligned}
\end{equation}
Here $r:\mathbb R^d\times\mathcal P_2(\mathbb R^d)\times\mathcal A\to\mathbb R$ denotes the running reward, $\lambda>0$ is the temperature parameter, and $\mathcal H(\pi(\cdot\mid x,\mu))$ denotes the Shannon entropy of the action distribution $\pi(\cdot\mid x,\mu)$ with respect to the reference measure $\nu$. In reinforcement learning, this entropy term is usually interpreted as an exploration regularizer: it penalizes overly concentrated action distributions and favors randomized policies, thereby encouraging broader exploration of the action space. Entropy-regularized formulations are widely used in continuous-time reinforcement learning and continuous-time mean-field learning; see, for instance, \cite{wang2020,jia2022,jia2023,jia_qlearning2023,Wei2023ContinuousTQ,frikhaActorCritic}.

We now introduce the controlled dynamics associated with a policy
$\pi\in\Pi_{\mathrm{add}}$. Let
$b:\mathbb R^d\times\mathcal P_2(\mathbb R^d)\times\mathcal A
\to\mathbb R^d$
and
$\sigma:\mathbb R^d\times\mathcal P_2(\mathbb R^d)\times\mathcal A
\to\mathbb R^{d\times n}$
be measurable coefficients, and consider the averaged drift and diffusion matrix
as
\begin{align}\label{eq:aggregated_coeffs_mkv}
  b^\pi(x,\mu)
  := \int_{\mathcal A} b(x,\mu,a)\,\pi(da\mid x,\mu),
  \qquad
  \Sigma^\pi(x,\mu)
  := \int_{\mathcal A}
  \sigma(x,\mu,a)\sigma(x,\mu,a)^\top\,\pi(da\mid x,\mu),
\end{align}
and we write
$\sigma^\pi(x,\mu):=\bigl(\Sigma^\pi(x,\mu)\bigr)^{1/2}
\in\mathbb R^{d\times d}$ for the symmetric positive semidefinite
square root of $\Sigma^\pi$. The population state then evolves
according to the McKean-Vlasov equation
\begin{align}\label{eq:MKV_only_SDE}
  ds_t
  = b^\pi(s_t,\mu_t)\,dt
  + \sigma^\pi(s_t,\mu_t)\,d B_t,
  \qquad
  s_0\sim\mu,
  \qquad
  \mu_t=\Law(s_t),
\end{align}
where $B$ is a $d$-dimensional Brownian motion. Thus, the
population law $(\mu_t)_{t\ge0}$ is generated by the controlled
McKean-Vlasov dynamics, and the objective functional
\eqref{eq:value_fn_mkv} evaluates the resulting collective
performance over an infinite horizon.

Formally, for each initial law $\mu\in\mathcal P_2(\mathbb R^d)$ and
each horizon $T>0$, we work on a probability space
$(\Omega,\mathcal F,\mathbb P)$ endowed with a filtration
$\mathbb F=(\mathcal F_t)_{0\le t\le T}$ such that $\mathcal F_0$
contains all $\mathbb P$-null sets of $\mathcal F$, the filtration is
right-continuous, and
$(\Omega,\mathcal F,\mathbb F,\mathbb P)$ supports a
$d$-dimensional $\mathbb F$-Brownian motion $B$, together with
an $\mathcal F_0$-measurable random variable $s_0$ satisfying
$\Law(s_0)=\mu$ and independent of $B$. Under the
standing assumptions introduced below, equation
\eqref{eq:MKV_only_SDE} admits a unique strong solution, and the law
flow $(\mu_t)_{0\le t\le T}$ is deterministic.

\begin{remark}[Interpretation of randomized feedback policies]\label{rem:randomized-feedback-interpretation}In our formulation, a randomized feedback policy $\pi\in\Pi_{\mathrm{add}}$ is interpreted in relaxed form, through the averaged drift and covariance in \eqref{eq:aggregated_coeffs_mkv}. Under weak uniqueness, this formulation has the same marginal laws as the continuously randomized MKV dynamics; see \cite{KarouiMeleard1990,Lacker,BayraktarCossoPham2018,Wei2023ContinuousTQ}. By contrast, policies sampled on a time grid and held fixed between sampling times generally differ at finite mesh and converge to the averaged dynamics as the mesh vanishes; see \cite{reconciling2026} for the single-agent case. 
\end{remark}

\smallbreak
\noindent
\textbf{Organization. }The rest of the paper is organized as follows. In
Section~\ref{sec:HJB_mkv} we analyze the associated stationary
Hamilton-Jacobi-Bellman equation used for policy evaluation.
Section~\ref{sec:policy-gradient} presents the policy gradient
formula for the McKean-Vlasov problem. In
Section~\ref{sec:MF-phibe} we introduce the mean-field PhiBE and study
its approximation properties. Section~\ref{sec:LQR}
specializes the analysis to the linear-quadratic setting and studies
the corresponding PhiBE error. Finally,
Section~\ref{sec:algorithm} presents the resulting
finite-dimensional policy-iteration scheme, and
Section~\ref{sec:numerics} presents the numerical examples.

\section{Policy gradient for McKean-Vlasov control}
\label{sec:policy-gradient}
This section develops the model-based policy-gradient framework for the McKean-Vlasov control problem introduced in \Cref{sec:formulation}. We first study the stationary HJB equation on $\mathcal P_2(\mathbb R^d)$ associated with a fixed policy $\pi$. This equation provides the policy-evaluation step: when the coefficients of the controlled MKV dynamics are known, the value functional $\mathcal V^\pi$ is characterized as the solution of a linear equation on the space of probability measures. We then derive the corresponding policy gradient theorem, which expresses the first variation of the performance functional in terms of the evaluated value function and an action-wise advantage kernel. Together, the HJB evaluation equation and the policy gradient formula define a model-based actor-critic scheme. Although this section assumes access to the infinitesimal generator, the same structure will be used in \Cref{sec:MF-phibe} as the basis for the model-free algorithm: the policy-evaluation equation will be replaced by the MF-PhiBE constructed from discrete-time data, while the actor update will still be driven by the policy gradient representation developed here, after replacing the unknown coefficients by the MF-PhiBE estimators.

\subsection{Stationary HJB on $\mathcal P_2(\mathbb R^d)$}
\label{sec:HJB_mkv}

In the McKean-Vlasov setting, the natural object of study is the
law-value functional $\mathcal V^\pi(\mu)$ introduced in
\eqref{eq:value_fn_mkv}. For a fixed policy $\pi\in\Pi_{\mathrm{add}}$, the
corresponding Hamilton-Jacobi-Bellman equation (HJB) is a linear equation posed on
$\mathcal P_2(\mathbb R^d)$.

We begin by introducing the infinitesimal generator of the law process
$(\mu_t)_{t\ge0}$ associated with \eqref{eq:MKV_only_SDE}. If
$ F:\mathcal P_2(\mathbb R^d)\to\mathbb R$ is sufficiently
smooth in the Lions sense, we define
\begin{equation}
\label{eq:L-pi-mkv}
  \begin{aligned}
    (\mathcal L_{b,\Sigma}^{\pi} F)(\mu)
    :=
    \int_{\mathbb R^d}
      \Bigl[
        b^{\pi}(\xi,\mu)\cdot \partial_{\mu} F(\mu)(\xi)
        + \tfrac{1}{2}\,\Sigma^{\pi}(\xi,\mu)
          : D_\xi\partial_{\mu} F(\mu)(\xi)
      \Bigr]\,\mu(d\xi).
  \end{aligned}
\end{equation}
We introduce the following regularity class
for functionals on $\mathcal P_2(\mathbb R^d)$.

\begin{definition}\label{def:C2poly_mkv}
A map $ F:\mathcal P_2(\mathbb R^d)\to\mathbb R$ belongs to
$\mathcal C^2(\mathcal P_2(\mathbb R^d))$ if it is Lions
differentiable on $\mathcal P_2(\mathbb R^d)$ and the maps
$(\mu,\xi)\mapsto \partial_\mu F(\mu)(\xi)$ and
$(\mu,\xi)\mapsto D_\xi\partial_\mu F(\mu)(\xi)$ are jointly
continuous.

We say that
$ F\in\mathcal C^2_{\mathrm{poly}}(\mathcal P_2(\mathbb R^d))$
if $ F\in\mathcal C^2(\mathcal P_2(\mathbb R^d))$ and there
exists a constant $C_{ F}>0$ such that for all
$(\mu,\xi)\in\mathcal P_2(\mathbb R^d)\times\mathbb R^d$,
\[
| F(\mu)|
\le C_{ F}\bigl(1+m_2(\mu)\bigr),
\]
\[
|\partial_\mu F(\mu)(\xi)|
\le C_{ F}\bigl(1+|\xi|+m_2(\mu)\bigr),
\qquad
\|D_\xi\partial_\mu F(\mu)(\xi)\|
\le C_{ F}\bigl(1+m_2(\mu)\bigr).
\]
\end{definition}

We now state the standing regularity hypothesis.

\begin{assumptions}\label{ass:_mkv_regularity}
Let $\pi\in\Pi_{\mathrm{add}}$ be fixed. Assume that, for every component
$h\in\{b_i^\pi,\sigma_{i,j}^\pi:i,j=1,\dots,d\}.$, the
derivatives $\nabla_s h(s,\mu)$, $D_{ss}^2h(s,\mu)$,
$\partial_\mu h(s,\mu)(\xi)$, and
$D_\xi\partial_\mu h(s,\mu)(\xi)$ exist for all
$(s,\mu,\xi)\in\mathbb R^d\times\mathcal P_2(\mathbb R^d)\times\mathbb R^d$,
are bounded by a constant $K_\pi$, and are locally Lipschitz in
$(s,\mu,\xi)$. Define $\ell^\pi_\lambda(s,\mu):=r^\pi(s,\mu)+\lambda\,\mathcal H\bigl(\pi(\cdot|s,\mu)\bigr)$, and assume in addition that
$\ell^\pi_\lambda\in
\mathcal C^{2,2}_{\mathrm{poly}}(\mathbb R^d\times\mathcal P_2(\mathbb R^d))$.
\end{assumptions}
\begin{remark}
Note that under \Cref{ass:_mkv_regularity}, $r^\pi_\lambda$ belongs to $\mathcal C^2_{\mathrm{poly}}(\mathcal P_2(\mathbb R^d))$. On the other hand, in the context of policy-gradient algorithms, the policy is often parametrized by a family $\{\pi_\omega\}_\omega$. In that case, \Cref{ass:_mkv_regularity} can be verified under smoothness assumptions on the parametrization and on the coefficients. For instance, \Cref{ass:_mkv_regularity} is satisfied by Gaussian policies of the form
\begin{align}\label{eq:genera_gaussina_parametrization}
    \pi_\omega(\cdot\mid s,\mu)
:=
\mathcal N\bigl(f_\omega(s,\mu),\Sigma_0\bigr),\qquad \Sigma_0 \succ 0,
\end{align}
with fixed covariance matrix, provided that $f_\omega$ is twice differentiable with respect to $(s,\mu)$ with bounded derivatives up to second order, and that the functions $b$, $\Sigma$, and $r$, together with their derivatives in $(s,\mu)$ up to second order, are bounded uniformly with respect to $a$. On compact state domains, and in particular on the torus, the boundedness of the spatial derivatives follows from smoothness; see \Cref{remark:compact_domains}. In the affine-quadratic LQR case studied in \Cref{sec:LQR}, \Cref{ass:_mkv_regularity} is verified directly from the explicit form of the averaged coefficients and reward, even when $f_\omega$ is affine and therefore unbounded.
\end{remark}

We can now state the well-posedness of the linear HJB equation, used for policy evaluation.

\begin{theorem}[Stationary HJB equation on $\mathcal P_2(\mathbb R^d)$]
\label{thm:MF_eval_infinite_mkv}
Let $\pi\in\Pi_{\mathrm{add}}$ be fixed and assume \Cref{ass:_mkv_regularity}. Then there exists a constant $\bar\beta_\pi>0$, depending on the constant $K_\pi>0$ from assumption~\ref{ass:_mkv_regularity}, such that, if $\beta>\bar\beta_\pi$, the value functional $\mathcal V^\pi$ defined in \eqref{eq:value_fn_mkv} is well defined and belongs to $\mathcal C^2_{\mathrm{poly}}(\mathcal P_2(\mathbb R^d))$. Moreover, $\mathcal V^\pi$ is the unique classical solution in $\mathcal C^2_{\mathrm{poly}}(\mathcal P_2(\mathbb R^d))$ of
\begin{equation}
\label{eq:HJB-eval-mkv}
  (\mathcal L_{b,\Sigma}^{\pi}-\beta)\mathcal V^\pi(\mu)
  + r_\lambda^\pi(\mu)=0,
  \qquad
  \mu\in\mathcal P_2(\mathbb R^d).
\end{equation}
Finally, there exists $C_\pi>0$, depending on the constants in
Assumption~\ref{ass:_mkv_regularity}, such that, for every
$(x,\nu)\in\mathbb R^d\times\mathcal P_2(\mathbb R^d)$,
\begin{equation}
\label{eq:value-derivative-bound-mkv}
|\partial_\mu\mathcal{V}^\pi(\nu)(x)|
+
\|D_x\partial_\mu\mathcal{V}^\pi(\nu)(x)\|
\le
\frac{C_\pi}{\beta-\bar\beta_\pi}
\bigl(1+|x|^2+m_2(\nu)\bigr).
\end{equation}
\end{theorem}

The proof of \Cref{thm:MF_eval_infinite_mkv} is provided in \Cref{app:HJB-proofs}.
\begin{remark}
\label{rem:hjb-from-coupled-system}
\Cref{thm:MF_eval_infinite_mkv} follows from the policy-evaluation result for the representative/population formulation in \cite{firstpaper}. The passage to the McKean-Vlasov formulation is obtained by integrating out the representative state, using the standard disintegration formula for functions on the Wasserstein space, as in \cite{BayraktarCossoPham2018}.
\end{remark}

\subsection{Policy gradient theorem}
\label{subsec:mkv-policy-gradient}

Let $\pi$ be a fixed feedback randomized policy whose density
$p_\pi(\cdot\mid s,\mu)$ with respect to the reference measure $\nu$ on
$\mathcal A$ is strictly positive. Let $\varphi$ be a measurable signed
kernel $\mathbb R^d\times\mathcal P_2(\mathbb R^d)\to\mathcal M(\mathcal A)$
satisfying
\begin{equation}\label{eq:signed-kernel-zero-mass}
    \varphi(\cdot \mid s,\mu) \ll \pi(\cdot \mid s,\mu),
    \qquad
    \int_{\mathcal A}\varphi(da\mid s,\mu)=0,
    \qquad
    (s,\mu)\in\mathbb R^d\times\mathcal P_2(\mathbb R^d),
\end{equation}
with bounded Radon-Nikodym derivative
$\psi:=d\varphi/d\pi\in L^\infty$. The zero-mass condition in
\eqref{eq:signed-kernel-zero-mass} ensures that $\pi^\varepsilon$
remains a probability kernel for small $\varepsilon$. For $|\varepsilon|$
small, define the perturbed kernel
\begin{equation}\label{eq:epsilon-perturbed-policy}
    \pi^\varepsilon(da\mid s,\mu)
    :=\pi(da\mid s,\mu)+\varepsilon\,\varphi(da\mid s,\mu)
    =
    \bigl(1+\varepsilon\psi(s,\mu,a)\bigr)\,\pi(da\mid s,\mu).
\end{equation}

\begin{assumptions}[Admissible signed perturbation.]\label{ass:signed_perturbation}
    We assume that \eqref{eq:signed-kernel-zero-mass}-\eqref{eq:epsilon-perturbed-policy} hold, and there exists $\varepsilon_0>0$ such that
$\pi^\varepsilon\in\Pi_{\mathrm{add}}$ for every
$|\varepsilon|\leq\varepsilon_0$, and that Assumption \ref{ass:_mkv_regularity} holds
for $\pi^\varepsilon$ with constants independent of $\varepsilon$.
Finally, assume that the perturbation $\varphi$ satisfies the domination
and differentiability hypotheses of \cite[Theorem~4.6]{firstpaper}.
\end{assumptions}

We next introduce the instantaneous advantage function associated with
the McKean-Vlasov problem.

\begin{definition}[Advantage function]
\label{def:essential-pointwise-kernels}
Let $\pi$ be a feedback randomized policy with law-value function
$\mathcal V^\pi$. For $(\mu,\xi,a)\in
\mathcal P_2(\mathbb R^d)\times\mathbb R^d\times\mathcal A$, define the
\emph{advantage function}
\begin{equation}\label{eq:qess-MV}
   q^\pi(\mu,\xi,a)
   :=
   r(\xi,\mu,a)-\lambda\log p_\pi(a\mid\xi,\mu)
   + b(\xi,\mu,a)\cdot \partial_\mu\mathcal V^\pi(\mu)(\xi)
   + \frac12\,\Sigma(\xi,\mu,a):D_\xi\partial_\mu\mathcal V^\pi(\mu)(\xi)-\beta\mathcal{V}^\pi(\mu).
\end{equation}
\end{definition}
Note that, by \cref{thm:MF_eval_infinite_mkv}, the kernel $q^\pi$ is centered
under the current policy in the sense that
\begin{equation*}
  \int_{\mathbb R^d}\int_{\mathcal A}
    q^\pi(\mu,\xi,a)\,\pi(da\mid \xi,\mu)\,\mu(d\xi)
    =0.
\end{equation*}
Thus, $q^\pi$ measures the infinitesimal gain of deviating from the current randomized policy. This centering property is the continuous-time mean-field
analogue of the usual zero-mean property of advantage functions in policy-gradient
methods.

\begin{remark}[Interpretation of the advantage function]
\label{rem:interpretation-advantage-function}
The function $q^\pi(\mu,\xi,a)$ represents the pointwise first-order contribution of choosing the action $a$ at state $\xi$ when the current population law is $\mu$. It combines the entropy-regularized instantaneous reward with the infinitesimal variation of the law-value function through the drift and diffusion coefficients. Thus, it is the McKean-Vlasov analogue of the continuous-time advantage rate: the usual state-action value function collapses in the continuous-time limit, and the dependence on the action remains only through the infinitesimal generator.

This object is related to, but different from, the q-functions introduced in \cite{Wei2023ContinuousTQ}. In that work, the learning procedure is based on an integrated q-function, defined on the population law and a test policy, together with an essential q-function, defined pointwise in the state and action. The integrated q-function is characterized by a weak martingale condition imposed over test policies in a neighbourhood of the target policy. Hence, its practical computation requires averaging the corresponding martingale loss over nearby test policies. Our approach uses the pointwise action-dependent quantity more directly. Once the current policy has been evaluated, the policy gradient theorem expresses the update direction in terms of $q^\pi(\mu,\xi,a)$ and the score function of the current policy. Therefore, the actor update only requires the evaluated value function, and the pointwise infinitesimal advantage. In particular, it does not require introducing an auxiliary integrated q-function, nor testing over a neighbourhood of policies.
\end{remark}

\begin{definition}[Discounted occupancy measure]\label{def:discounted_ocupacy_MKV}
Let $\pi$ be a fixed policy, let $\mu\in\mathcal P_2(\mathbb R^d)$, and
let $(\mu_t^\pi)_{t\ge0}$ be the McKean-Vlasov flow associated with
\eqref{eq:MKV_only_SDE} starting from $\mu$. The discounted occupancy
measure $\rho_\mu^\pi$ on $\mathcal P_2(\mathbb R^d)\times\mathbb R^d$ is
defined by
\[
\rho_\mu^\pi(B)
:=\beta\,
\int_0^\infty e^{-\beta t}
\int_{\mathbb R^d}
\mathbf 1_{B}(\mu_t^\pi,\xi)\,\mu_t^\pi(d\xi)\,dt,
\qquad
B\in\mathcal B\bigl(\mathcal P_2(\mathbb R^d)\times\mathbb R^d\bigr).
\]
Equivalently, we have 
\[
\rho_\mu^\pi(d\eta,d\xi)
=\beta\,
\int_0^\infty e^{-\beta t}\,
\delta_{\mu_t^\pi}(d\eta)\,
\mu_t^\pi(d\xi)\,dt,
\qquad
(\eta,\xi)\in \mathcal P_2(\mathbb R^d)\times\mathbb R^d .
\]
\end{definition}

We now state the policy gradient formula for the McKean-Vlasov
problem.

\begin{theorem}[Policy gradient for the McKean-Vlasov control]
\label{cor:policy-gradient-mkv}
Fix a baseline policy $\pi\in\Pi_{\mathrm{add}}$. Let $\varphi$ be a signed kernel satisfying
Assumption~\ref{ass:signed_perturbation}. Assume $\mu\in\mathcal{P}_2(\mathbb{R}^d)$ and $\beta>2\bar\beta_\pi$. Define
$\mathcal J(\varepsilon):=\mathcal V^{\pi^\varepsilon}(\mu)$. Then
$\mathcal J$ is differentiable at $\varepsilon=0$, and
\begin{align}\label{eq:gateaux-pg-mkv}
    \frac{d}{d\varepsilon}\mathcal J(\varepsilon)\bigg|_{\varepsilon=0}
=  \frac{1}{\beta}
\mathbb E_{(\eta,\xi)\sim\rho_\mu^\pi}\!\left[
\int_{\mathcal A}
q^\pi(\eta,\xi,a)\,
\varphi(da\mid\xi,\eta)
\right].
\end{align}
\end{theorem}

 The proof of \Cref{cor:policy-gradient-mkv} is given in \Cref{app:HJB-proofs}.
\begin{remark}[Baseline invariance]
\label{rem:baseline-invariance-gateaux}
The identity \eqref{eq:gateaux-pg-mkv} is invariant under the addition of action independent terms to $q^\pi$,
since $\int_{\mathcal A}\varphi(d a\mid \xi,\nu)=0$ for every $(\xi,\nu)$.
\end{remark}
\begin{remark}
\label{rem:policy-gradient-comparison}
\Cref{cor:policy-gradient-mkv} is a key structural ingredient of the
present paper. It gives a policy-gradient formula before a
finite-dimensional actor parametrization is chosen, and it describes the
first-order variation of the objective under arbitrary admissible
perturbations of the randomized policy.

This distinguishes our approach from existing policy-gradient and
q-function-based methods. Compared with \cite{frikhaActorCritic}, where
the gradient is obtained by differentiating the HJB equation with
respect to the policy parameter, our formula expresses the variation
directly in terms of the evaluated value functional. Compared with
\cite{Wei2023ContinuousTQ}, we use the pointwise infinitesimal advantage
$q^\pi(\eta,\xi,a)$ directly in the policy-gradient formula, whereas
their method requires computing integrated q-functions for test policies
in a neighbourhood of the target policy. Consequently, our actor update
does not require solving an additional sensitivity equation or
introducing an auxiliary integrated q-function over nearby policies.
\end{remark}

\subsection{Parametric policy gradient}
\label{subsec:parametric-pg}

The G\^ateaux gradient formula \eqref{eq:gateaux-pg-mkv} expresses the
directional derivative of $\mathcal J$ in terms of an arbitrary signed
kernel $\varphi$. In practice, policies are often parametrized by a
finite-dimensional vector $\omega\in\mathbb R^p$, and one wishes to
compute $\nabla_\omega \mathcal J(\omega)$ directly. The parametric
policy gradient theorem below shows that the gradient can be written as
a weighted expectation of the advantage function against the score
function. To state the result precisely, we require the following two
assumptions on the parametric family.
\medbreak

\begin{assumptions}\label{ass:A2}
Let $\{\pi_\omega\}_{\omega\in\mathbb R^p}\subset\Pi_{\mathrm{add}}$ be a
parametric family of randomized feedback policies. For $\nu$-a.e.
$a\in\mathcal A$ and every $(s,\mu)$, the map
$\omega\mapsto\log p_{\pi_\omega}(a\mid s,\mu)$ is continuously
differentiable. Moreover, for every compact $K\subset\mathbb R^p$, the maps $\omega\mapsto b^{\pi_\omega}(s,\mu),\,
\omega\mapsto\Sigma^{\pi_\omega}(s,\mu),$ and $
\omega\mapsto\ell_\lambda^{\pi_\omega}(s,\mu)$ are continuously differentiable, and there exist $C_K>0$ and
$g_K(\cdot,s,\mu)\in L^1(\mathcal A,\nu)$ such that, for all $\omega\in K$,
\begin{align*}
\|\nabla_\omega p_{\pi_\omega}(a\mid s,\mu)\|
&\le g_K(a,s,\mu),\\
\|\nabla_\omega b^{\pi_\omega}(s,\mu)\|
&\le C_K\bigl(1+|s|+\sqrt{m_2(\mu)}\bigr),\\
\|\nabla_\omega\Sigma^{\pi_\omega}(s,\mu)\|
+\|\nabla_\omega\ell_\lambda^{\pi_\omega}(s,\mu)\|
&\le C_K\bigl(1+|s|^2+m_2(\mu)\bigr).
\end{align*}
Finally, Assumption~\ref{ass:_mkv_regularity} holds for $\pi_\omega$ with
constants uniform in $\omega\in K$.
\end{assumptions}

\begin{remark}[Gaussian policies and assumptions]
\label{rem:gaussian-policy}
For Gaussian policies of the form
\eqref{eq:genera_gaussina_parametrization}, Assumptions
\ref{ass:_mkv_regularity} and \ref{ass:A2} follow, for instance, under
the bounded-coefficient conditions stated above when $f_\omega$ has
linear growth and bounded first- and second-order spatial and Lions
derivatives, locally uniformly in $\omega$. Smooth cylindrical
parametrizations on compact state and parameter sets provide a standard
example; see \cite[Section~2.2]{pham2022meanfield}. Affine Gaussian means
on $\mathbb R^d$ fall outside these globally bounded sufficient conditions
and are verified separately in the LQR section by a direct
polynomial-growth calculation.
\end{remark}

\begin{theorem}[Parametric policy gradient]
\label{thm:parametric-policy-gradient}
Let $\{\pi_\omega\}_{\omega\in\mathbb R^p}$ be a family of feedback
randomized policies satisfying Assumption~\ref{ass:A2}. Fix
$\omega\in\mathbb R^p$, set $\pi:=\pi_\omega$, and fix
$\mu\in\mathcal P_2(\mathbb R^d)$. Define $\mathcal J(\omega):=\mathcal V^{\pi_\omega}(\mu).$ Assume $\beta>2\bar\beta_\pi$. Then the map
$\omega\mapsto \mathcal J(\omega)$ is differentiable at $\omega$, and
\begin{equation}\label{eq:parametric-pg}
\begin{aligned}
\nabla_\omega \mathcal J(\omega)
&= \frac{1}{\beta}
\mathbb E_{(\eta,\xi)\sim\rho_\mu^\pi}\!\left[
\int_{\mathcal A}
q^\pi(\eta,\xi,a)\,
\nabla_\omega \log p_\pi(a\mid\xi,\eta)\,
\pi(da\mid\xi,\eta)
\right].
\end{aligned}
\end{equation}
Equivalently,
\begin{equation*}
\begin{aligned}
\nabla_\omega \mathcal J(\omega)
&=\frac{1}{\beta}
\mathbb E_{\substack{(\eta,\xi)\sim\rho_\mu^\pi\\ a\sim\pi(\cdot\mid\xi,\eta)}}\!\Bigl[
q^\pi(\eta,\xi,a)\,
\nabla_\omega \log p_\pi(a\mid\xi,\eta)
\Bigr],
\end{aligned}
\end{equation*}
where $\rho_\mu^\pi$ is the discounted occupancy measure associated with
$\pi$ and the initial law $\mu$, and $q^\pi$ is the advantage function
introduced in \Cref{def:essential-pointwise-kernels}.
\end{theorem}
 The proof of \Cref{thm:parametric-policy-gradient} is given in \Cref{app:HJB-proofs}.

\begin{remark}
\label{rem:parametric-policy-gradient-comparison}
The identity \eqref{eq:parametric-pg} has the classical policy-gradient form ``advantage times score''. Once the current policy has been evaluated, all the quantities entering the formula are obtained from the value functional $\mathcal V^\pi$, its Lions derivatives, the pointwise coefficients in the generator, and the score function $\nabla_\omega\log p_{\pi_\omega}$. Therefore, the actor direction can be approximated by sampling the discounted occupancy measure under the current policy and drawing actions from $\pi_\omega(\cdot\mid\xi,\eta)$.

This should be contrasted with the policy-gradient representation in \cite{frikhaActorCritic}. There, the gradient is obtained by differentiating the HJB equation with respect to the policy parameter. Consequently, the gradient is characterized through an additional linear equation whose source term depends on the value function under the current policy. Thus, computing the ascent direction requires solving an extra sensitivity equation, or equivalently an additional policy-evaluation-type problem. By contrast, \eqref{eq:parametric-pg} gives the actor direction explicitly once $\mathcal V^\pi$ and its derivatives are available.

Moreover, the differentiated HJB approach involves terms depending on $\nabla_\omega b^{\pi_\omega}$ and $\nabla_\omega\Sigma^{\pi_\omega}$. These terms require differentiating the averaged coefficients with respect to the policy parameters, which is restrictive when the dynamics are unknown or only accessible through data. In contrast, \eqref{eq:parametric-pg} is written directly in terms of the pointwise infinitesimal advantage $q^\pi$ and the score of the policy. This avoids differentiating the aggregated drift and diffusion with respect to $\omega$, and is therefore better aligned with the MF-PhiBE formulation developed below.
\end{remark}


\section{Mean-field PhiBE}
\label{sec:MF-phibe}

In this section we introduce a physics-informed Bellman equation for the
McKean-Vlasov control problem formulated in \Cref{sec:formulation}. The
linear HJB equation \eqref{eq:HJB-eval-mkv} provides a continuous-time
characterization of the value associated with a fixed policy. Its
generator \eqref{eq:L-pi-mkv}, however, depends on the averaged drift
$b^\pi$ and diffusion matrix $\Sigma^\pi$. Therefore, a direct use of
\eqref{eq:HJB-eval-mkv} requires knowledge of the coefficients $b$ and
$\sigma$. This requirement is restrictive when the dynamics are not known
explicitly. In many applications, one only has access to short-time
observations of the system, for instance
\[
\mathcal D
=
\bigl\{
(s_\ell,\mu_\ell,a_\ell,s_\ell',r_\ell)
\bigr\}_{\ell=1}^{N_D},
\]
where $s_\ell$ is the observed state, $\mu_\ell$ is the corresponding
population law, $a_\ell$ is the applied action, $s_\ell'$ is the state
observed after a time step $\Delta t$, and $r_\ell$ is the observed
running reward.

Our goal is to construct a surrogate of \eqref{eq:HJB-eval-mkv}
that preserves the continuous-time structure of the law-space HJB
equation while using the discrete-time data in $\mathcal D$. The resulting equation is called the mean-field Physics-Informed Bellman Equation (MF-PhiBE). In this section, we introduce the MF-PhiBE and its optimal counterpart.
We then prove that, if $\hat\pi^*$ is an optimizer of the optimal
MF-PhiBE, then the true value $\mathcal V^{\hat\pi^*}$ approximates the
optimal value function $\mathcal V^*$ with first-order accuracy in
$\Delta t$.

\subsection{Definition of MF-PhiBE}
\label{subsec:def-phibe}

We now introduce a data-driven approximation of the policy-evaluation
equation \eqref{eq:HJB-eval-mkv}. The guiding principle is to preserve
the continuous-time law-space structure of the HJB equation, while replacing
the unknown local characteristics in the generator by one-step quantities
that can be estimated from discrete-time observations.

Fix $\Delta t>0$. For $(s,\mu,a)$, let
$(\mu_\tau^a)_{\tau\in[0,\Delta t]}$ be the law flow of the frozen-action
MKV equation with coefficients $b(\cdot,\cdot,a)$ and
$\sigma(\cdot,\cdot,a)$, initialized at $\mu$, and let
$(s_\tau^a)_{\tau\in[0,\Delta t]}$ be the corresponding decoupled state
process initialized at $s$ and evolving along $(\mu_\tau^a)$; the first
two increment moments define
\begin{align}
\hat b(s,\mu,a)
&:=
\mathbb E\!\left[
\frac{s_{\Delta t}-s}{\Delta t}
\,\Big|\,
s_0=s,\ \mu_0=\mu,\, a_\tau=a,\quad\tau\in[0,\Delta t]
\right],
\label{eq:estimated-drift-policy-mkv22}
\\[4pt]
\hat \Sigma(s,\mu,a)
&:=
\mathbb E\!\left[
\frac{(s_{\Delta t}-s)(s_{\Delta t}-s)^\top}{\Delta t}
\,\Big|\,
s_0=s,\ \mu_0=\mu,\, a_\tau=a,\quad\tau\in[0,\Delta t]
\right].
\label{eq:estimated-covariance-policy-mkv2}
\end{align}

These quantities are the natural discrete-time proxies for the drift and
diffusion matrix appearing in the exact generator
\eqref{eq:L-pi-mkv}. For $\pi\in\Pi_{\mathrm{add}}$, define the corresponding averaged
estimated coefficients by
\begin{align}\label{eq:averaged_b_sigma_hat}
    \hat b^\pi(s,\mu)
  :=
  \int_{\mathcal A}\hat b(s,\mu,a)\,\pi(da\mid s,\mu),
  \qquad
  \hat\Sigma^\pi(s,\mu)
  :=
  \int_{\mathcal A}\hat\Sigma(s,\mu,a)\,\pi(da\mid s,\mu).
\end{align}

Accordingly, for a smooth functional
$\Phi:\mathcal P_2(\mathbb R^d)\to\mathbb R$, we define the estimated
generator by
\begin{equation}
\label{eq:estimated-generator-mkv}
\begin{aligned}
(\mathcal L_{\hat b,\hat\Sigma}^\pi \Phi)(\mu)
:=
\int_{\mathbb R^d}
\Bigl[
\hat b^\pi(\xi,\mu)\cdot \partial_\mu \Phi(\mu)(\xi)
+\frac12\,\hat\Sigma^\pi(\xi,\mu):D_\xi\partial_\mu\Phi(\mu)(\xi)
\Bigr]\mu(d\xi).
\end{aligned}
\end{equation}
This operator has exactly the same form as
$\mathcal L_{b,\Sigma}^\pi$ in \eqref{eq:L-pi-mkv}; only the unknown
coefficients are replaced by their one-step estimators.

We can now define the mean-field Physics-Informed Bellman Equation.

\begin{definition}[MF-PhiBE and optimal MF-PhiBE]
\label{def:MF-PhiBE-mkv}
For a fixed policy $\pi\in\Pi_{\mathrm{add}}$, the MF-PhiBE is
\begin{equation}
\label{eq:MF-PhiBE-evaluation}
(\mathcal L_{\hat b,\hat\Sigma}^\pi-\beta)\hat{\mathcal V}^\pi(\mu)
+r_\lambda^\pi(\mu)
=
0,
\qquad
\mu\in\mathcal P_2(\mathbb R^d).
\end{equation}
Whenever \eqref{eq:MF-PhiBE-evaluation} admits a unique solution in
$\mathcal C^2_{\mathrm{poly}}(\mathcal P_2(\mathbb R^d))$, we denote
this solution by $\hat{\mathcal V}^\pi$.

We define the optimal MF-PhiBE value by
\begin{equation}
\label{eq:MF-PhiBE-optimal}
\hat{\mathcal V}^*(\mu)
:=
\sup_{\pi\in\Pi_{\mathrm{add}}}
\hat{\mathcal V}^\pi(\mu),
\qquad
\mu\in\mathcal P_2(\mathbb R^d),
\end{equation}
where the supremum is taken over the policies for which
\eqref{eq:MF-PhiBE-evaluation} is well posed. If there exists
$\hat\pi^*\in\Pi_{\mathrm{add}}$ such that $\hat{\mathcal V}^*(\mu)
=
\hat{\mathcal V}^{\hat\pi^*}(\mu),$ for every $\mu\in\mathcal P_2(\mathbb R^d),$ then $\hat\pi^*$ is called an optimal MF-PhiBE policy.
\end{definition}

\begin{remark}
    Equation \eqref{eq:MF-PhiBE-evaluation} should be understood as a
generator-level approximation of \eqref{eq:HJB-eval-mkv}. In
particular, the approximation does not replace the continuous-time control
problem by a discrete-time Bellman recursion. Instead, it keeps the
stationary HJB structure on $\mathcal P_2(\mathbb R^d)$ and only substitutes
the unknown infinitesimal coefficients by quantities accessible from
discrete-time data. This is precisely the feature that makes MF-PhiBE
suitable for the present continuous-time mean-field setting.
\end{remark}

\subsection{Error analysis}
\label{subsec:wellposedness-phibe}

To establish the well-posedness of MF-PhiBE and its first-order accuracy, we consider the following assumptions. 
\begin{assumptions}[Coefficients regularity]
\label{ass:MF-PhiBE-structure}
There exist constants $K>0$ and $\alpha>0$ such that, for every
$a\in\mathcal A$, the coefficients $b(\cdot,\cdot,a)$ and
$\sigma(\cdot,\cdot,a)$ satisfy
\[
|b(s,\mu,a)-b(s',\mu',a)|
+
\|\sigma(s,\mu,a)-\sigma(s',\mu',a)\|
\le
K\bigl(|s-s'|+\mathcal W_2(\mu,\mu')\bigr),
\]
and
\[
|b(s,\mu,a)|+\|\sigma(s,\mu,a)\|
\le
K,
\qquad
\Sigma(s,\mu,a):=\sigma(s,\mu,a)\sigma(s,\mu,a)^\top
\succeq \alpha I_d .
\]
Moreover, for every component
$h\in\{b_i,\sigma_{i\ell}:i=1,\dots,d,\ \ell=1,\dots,n\}$, all spatial
and Lions derivatives of $(s,\mu)\mapsto h(s,\mu,a)$ up to order four
exist, are continuous, and are bounded uniformly with respect to
$a\in\mathcal A$.
\end{assumptions}

\begin{remark}[Compact domains and the torus setting]\label{remark:compact_domains}
Assumption~\ref{ass:MF-PhiBE-structure} is stated on $\mathbb{R}^d$ with globally bounded coefficients and globally bounded derivatives. This is imposed in order to isolate the consistency error of the one-step MF-PhiBE estimators from localization and moment-growth issues. If the state space is a compact smooth domain, and in particular if the state space is the flat torus $\mathbb{T}^d$, the same assumption can be read with $\mathbb{R}^d$ replaced by $\mathbb{T}^d$ and $\mathcal{P}_2(\mathbb{R}^d)$ replaced by $\mathcal{P}(\mathbb{T}^d)$. The Wasserstein distance is then induced by the periodic geodesic distance, and all spatial derivatives are interpreted as periodic derivatives. In the torus case, smooth coefficients and smooth randomized policy densities automatically have bounded spatial derivatives, and the moment-growth terms appearing in the unbounded Euclidean formulation disappear because $\mathbb{T}^d$ is compact. Thus the coefficient estimates and the averaged-coefficient estimates apply without additional moment assumptions. For unbounded state spaces, Assumption~\ref{ass:MF-PhiBE-structure} can be replaced by a localized or problem-specific growth condition, provided one has the corresponding moment estimates for the controlled McKean-Vlasov flow. The LQR analysis in Section~\ref{sec:LQR} is treated separately for this reason: its drift is not globally bounded, but the Riccati structure provides the required stability and error estimates directly.
\end{remark}

The following assumption concerns the optimal policies. 

\begin{assumptions}[Uniform regularity of the relevant policies]
\label{ass:optimal-policy-regularity}
Let $\Pi_*\subset\Pi_{\mathrm{add}}$ be a family of policies. For every
$\pi\in\Pi_*$, assume
$\pi(da\mid s,\mu)=p_\pi(a\mid s,\mu)\nu(da)$ with $p_\pi>0$
$\nu$-a.e., and assume that
$(s,\mu)\mapsto p_\pi(a\mid s,\mu)$ belongs to
$\mathcal C^{2,2}(\mathbb R^d\times\mathcal P_2(\mathbb R^d))$.
There exists a single function $\ell_*\in L^1(\nu)$ such that, uniformly
in $\pi\in\Pi_*$,
\[
|p_\pi|
+\|\nabla_s p_\pi\|
+\|D_{ss}^2p_\pi\|
+|\partial_\mu p_\pi(\xi)|
+\|D_\xi\partial_\mu p_\pi(\xi)\|
\le \ell_*(a).
\]
The same quantities are locally Lipschitz in $(s,\mu,\xi)$ with local
constants dominated by a single $L^1(\nu)$ function, uniformly in
$\pi\in\Pi_*$. Finally,
$\ell_\lambda^\pi\in
\mathcal C^{2,2}_{\mathrm{poly}}
(\mathbb R^d\times\mathcal P_2(\mathbb R^d))$
with polynomial-growth constants uniform in $\pi\in\Pi_*$.
\end{assumptions}

\begin{remark}\label{rem:policy-regularity-approximation}
    On a compact state domain, or for Gaussian policies with fixed uniformly elliptic covariance and globally bounded $\mathcal C^{2,2}$ mean, the domination condition in Assumption~\ref{ass:optimal-policy-regularity} is satisfied. For affine Gaussian means on the whole space $\mathbb{R}^d$, this domination is generally too strong; such policies are better treated either by localization, truncation, or by a problem-specific argument such as the LQR analysis in Section~\ref{sec:LQR}.
\end{remark}

We now state the main result of this section. It gives the first-order
accuracy of MF-PhiBE both at the level of policy evaluation and at the
level of policy improvement.

\begin{theorem}[Optimal MF-PhiBE error]
\label{thm:error_V_eval_optimal_pure_mkv}
Fix $\beta>0$ and assume Assumption~\ref{ass:MF-PhiBE-structure}. Suppose that there exists $\Delta_0\in(0,1]$ such that the exact optimizer $\pi^*$ and an MF-PhiBE optimizer $\hat\pi_{\Delta t}^*$ exist for every $\Delta t\in(0,\Delta_0]$, and that Assumption~\ref{ass:optimal-policy-regularity} holds for $\Pi_*:=\{\pi^*\}\cup \{\hat\pi_{\Delta t}^*:0<\Delta t\le\Delta_0\}.$ Then there exist $\bar\beta_*>0$ and $C>0$, independent of $\Delta t$, such that, after decreasing $\Delta_0$ if necessary, if $\beta>\max\{\bar\beta_*,\beta_0(2)\}$, then for every $\Delta t\in(0,\Delta_0]$ and $\mu\in\mathcal P_2(\mathbb R^d)$,
\begin{equation}
\label{eq:error_V_optimal_pure_mkv}
0\le
\mathcal V^*(\mu)-\mathcal V^{\hat\pi_{\Delta t}^*}(\mu)
\le
\frac{C(1+m_2(\mu))}
{(\beta-\bar\beta_*)(\beta-\beta_0(2))}\,\Delta t.
\end{equation}
Here $\beta_0(2)$ is chosen uniformly over $\Pi_*$ in the corresponding
second-moment estimate.
\end{theorem}
The proof of \Cref{thm:error_V_eval_optimal_pure_mkv} is provided in \Cref{app:error-proofs}.

\begin{remark}
The proof of \Cref{thm:error_V_eval_optimal_pure_mkv} first establishes
order-$\Delta t$ consistency of the one-step coefficients and their
policy averages, together with the regularity required for the exact and
MF-PhiBE evaluation equations. For each
$\pi\in\{\pi^*,\hat\pi_{\Delta t}^*\}$, subtracting these equations and
applying the It\^o--Lions formula, the derivative estimate, and the
uniform second-moment bound yields an order-$\Delta t$ fixed-policy
error. The result then follows from the optimality of $\pi^*$ and
$\hat\pi_{\Delta t}^*$.
\end{remark}

\section{Linear-quadratic mean-field control and MF-PhiBE error}
\label{sec:LQR}
\subsection{Problem setting}
\label{subsec:LQR-model}

We consider an infinite-horizon entropy-regularized linear-quadratic McKean-Vlasov control problem. Let $d,m\in\mathbb N$ and set $\mathcal A=\mathbb R^m$. Throughout the LQR section, the reference measure $\nu$ is Lebesgue
measure on $\mathbb R^m$.

Let $A,\bar A\in\mathbb R^{d\times d}$, $B\in\mathbb R^{d\times m}$, and $\gamma\in\mathbb R^d$ be fixed. In this section, the Brownian motion is one-dimensional, and we consider \eqref{eq:MKV_only_SDE} with
\begin{equation}
\label{eq:setting_lqr_for_phibe}
  b(x,\mu,a)=Ax+\bar A\,\bar\mu+Ba,
  \qquad
  \sigma(x,\mu,a)=\gamma,
  \qquad
  \bar\mu:=\int_{\mathbb R^d}x\,\mu(dx).
\end{equation}
Let $Q,\bar Q\in\mathbb S_+^d$ and $N\in\mathbb S_{++}^m$. The running reward is
\begin{equation}
\label{eq:reward_LQR_simplify}
  r(x,\mu,a)
  =
  -
  \left(
    x^\top Qx
    +
    \bar\mu^\top\bar Q\,\bar\mu
    +
    a^\top Na
  \right).
    \end{equation}
For a randomized feedback policy $\pi$, the entropy-regularized reward $r_\lambda^\pi$ is defined in \eqref{eq:reg_reward_avg_mkv}. The corresponding value functional and optimal value are
\begin{equation}
\label{eq:LQR-MKV-value-functional}
  \mathcal V^\pi(\mu)
  =
  \mathbb E\left[
    \int_0^\infty e^{-\beta t}
    r_\lambda^\pi(\mu_t)\,dt
    \,\bigg|\,
    \mu_0=\mu
  \right],
  \qquad
  \mathcal V^*(\mu)
  :=
  \sup_{\pi\in\Pi_{\mathrm{add}}}\mathcal V^\pi(\mu),
  \qquad
  \mu\in\mathcal P_2(\mathbb R^d).
\end{equation}
Throughout this section, $\Pi_{\mathrm{add}}$ is restricted to feedback policies $\pi:\mathbb R^d\times\mathcal P_2(\mathbb R^d)\to\mathcal P(\mathbb R^m)$ such that $\pi(\cdot\mid x,\mu)\ll\nu$ for every $(x,\mu)$, with strictly positive density $p_\pi(\cdot\mid x,\mu)$, $\mathcal V^\pi(\mu)$ is finite for every $\mu\in\mathcal P_2(\mathbb R^d)$, and the associated law flow satisfies $\lim_{t\to\infty}e^{-\beta t}m_2(\mu_t)=0$.

We now state the structural conditions used in the LQR analysis.

\begin{definition}[Stabilizability]
\label{def:LQR-MKV-stab-detect}
The pair $(A,B)\in\mathbb R^{d\times d}\times\mathbb R^{d\times m}$ is stabilizable if there exists $\Theta\in\mathbb R^{m\times d}$ such that $A+B\Theta$ is Hurwitz.
\end{definition}

\begin{assumptions}[LQR conditions]
\label{ass:LQR-MKV}
Let $A_\beta:=A-\frac{\beta}{2}I_d,$ and $\widetilde A_\beta:=A+\bar A-\frac{\beta}{2}I_d.$
Assume that the following conditions hold:
\begin{enumerate}
  \item[(H1)] $N\in\mathbb S_{++}^m$, $Q\in\mathbb S_{++}^d$, and $Q+\bar Q\in\mathbb S_{++}^d$.
  \item[(H2)] The pairs $(A_\beta,B)$ and $(\widetilde A_\beta,B)$ are stabilizable.
\end{enumerate}
\end{assumptions}

\begin{remark}
\label{rem:LQR-assumptions-specialization}
Condition \Cref{ass:LQR-MKV} {\rm (H1)} is the coercivity condition for the quadratic cost associated with the reward \eqref{eq:reward_LQR_simplify}. Condition \Cref{ass:LQR-MKV} {\rm (H2)} imposes discounted stabilizability of the centered dynamics and of the mean dynamics. These are the structural assumptions used to select the stabilizing solutions of the Riccati equations \eqref{eq:ARE_SIMPLIFY_RICCATI_K} and \eqref{eq:ARE_SIMPLIFY_RICCATI_LAMBDA}.
\end{remark}
Let $K,\Lambda\in\mathbb S_+^d$ and $R\in\mathbb R$ solve
\begin{empheq}[left=\empheqlbrace]{align}
  \beta K
  &=
  Q+KA+A^\top K
  -
  KB N^{-1}B^\top K, \label{eq:ARE_SIMPLIFY_RICCATI_K}\\
  \beta\Lambda
  &=
  Q+\bar Q
  +
  \Lambda(A+\bar A)
  +
  (A+\bar A)^\top\Lambda
  -
  \Lambda B N^{-1}B^\top\Lambda,\label{eq:ARE_SIMPLIFY_RICCATI_LAMBDA} \\
  \beta R
  &=
  \gamma^\top K\gamma
  -
  \frac{\lambda}{2}
  \left(
    m\log(\pi_0\lambda)-\log\det N
  \right),\label{eq:ARE_SIMPLIFY_RICCATI_R}
\end{empheq}
where $\pi_0=3.14159\ldots$ denotes the mathematical constant. The candidate law-value functional is
\begin{equation}
\label{eq:LQR-MKV-value-candidate}
  \mathcal V(\mu)
  =
  -
  \left[
    \operatorname{Tr}\bigl(K\,\mathrm{Cov}(\mu)\bigr)
    +
    \bar\mu^\top\Lambda\,\bar\mu
    +
    R
  \right],
\end{equation}
where $\operatorname{Tr}\bigl(K\,\mathrm{Cov}(\mu)\bigr)
  =
  \int_{\mathbb R^d}
  (x-\bar\mu)^\top K(x-\bar\mu)\,\mu(dx).$ The following proposition is a special case of \cite[Theorem~A.1]{frikha2024full} to the setting considered here.
\begin{proposition}[Optimal policy for the entropy-regularized MKV-LQR problem]\label{prop:LQR-MKV-optimal-policy}
Assume \Cref{ass:LQR-MKV}. Let $(K,\Lambda,R)$ be the stabilizing solution of \eqref{eq:ARE_SIMPLIFY_RICCATI_K}-\eqref{eq:ARE_SIMPLIFY_RICCATI_R}, and let $\mathcal V$ be given by \eqref{eq:LQR-MKV-value-candidate}. Define
\begin{equation}
\label{eq:LQR-MKV-optimal-policy-candidate}
  \pi^*(\cdot\mid x,\mu)
  =
  \mathcal N\left(
    -N^{-1}B^\top
    \bigl(K(x-\bar\mu)+\Lambda\bar\mu\bigr),
    \frac{\lambda}{2}N^{-1}
  \right).
\end{equation}
Then $\pi^*\in\Pi_{\mathrm{add}}$ and
\[
  \mathcal V^*(\mu)
  =
  \mathcal V^{\pi^*}(\mu)
  =
  \mathcal V(\mu),
  \qquad
  \mu\in\mathcal P_2(\mathbb R^d).
\]
\end{proposition}

\subsection{LQR transition kernel and MF-PhiBE error}
The linear structure of the LQR dynamics yields an explicit frozen-action
transition kernel and hence an explicit drift-only LQR-MF-PhiBE. We first
derive that transition law.

Fix $\Delta t>0$. On the interval $[0,\Delta t]$, we take the control to be constant and equal to $a\in\mathbb R^m$. Taking the mean in the McKean-Vlasov dynamics \eqref{eq:MKV_only_SDE} gives
\[
    \frac{d}{d\tau}\bar\mu_\tau
    =
    \tilde A\,\bar\mu_\tau+Ba,
    \qquad
    \bar\mu_0=\bar\mu,
    \qquad
    \tilde A:=A+\bar A.
\]
Hence
\[
    \bar\mu_\tau
    =
    e^{\tilde A\tau}\bar\mu
    +
    M_{\tilde A,\tau}Ba,
    \qquad
    M_{\tilde A,\tau}
    :=
    \int_0^\tau e^{\tilde A(\tau-r)}\,dr.
\]
Solving the linear state equation on $[0,\Delta t]$ gives the
one-step transition law
\begin{align}
\label{eq:LQR-phibe-transition-law}    
  \rho_{\Delta t}(s'\mid s,\bar\mu,a)
    =
    \mathcal N\!\left(
        (I_d+\Delta t\,\hat A)s
        +
        \Delta t\,\hat{\bar A}\,\bar\mu
        +
        \Delta t\,\hat B a,\,
        C_{A,\Delta t}
    \right),
\end{align}
where
\begin{equation}
\label{eq:definition_coef_HAT}
    \hat A
    :=
    \frac{e^{A\Delta t}-I_d}{\Delta t},
    \quad
    \hat{\bar A}
    :=
    \frac{1}{\Delta t}
    \int_0^{\Delta t}
    e^{A(\Delta t-\tau)}
    \bar A\,e^{\tilde A\tau}\,d\tau,\quad \hat B
    :=
    \frac{1}{\Delta t}
    \int_0^{\Delta t}
    e^{A(\Delta t-\tau)}
    \bigl(\bar A\,M_{\tilde A,\tau}B+B\bigr)\,d\tau
\end{equation}
and \begin{equation*}C_{A,\Delta t} := \int_0^{\Delta t} e^{A(\Delta t-\tau)}\,\gamma\gamma^\top\,e^{A^\top(\Delta t-\tau)}\,d\tau. 
\end{equation*}

\begin{remark}[Non-identifiability of the transition mean]
\label{rem:nonidentifiability_mkv}
For a fixed time step $\Delta t$, consider one-step data of the form $(s,\mu,a,s')$, where $s'\sim \rho_{\Delta t}(\cdot\mid s,\mu,a)$. Since the LQR-MF-PhiBE below is based on the first-moment estimator, the relevant question is whether the one-step transition mean determines the continuous-time matrices $(A,\bar A,B)$ uniquely. This leads us to study the map
\[
    (A,\bar A,B)\longmapsto \mathcal M_{\rho_{\Delta t}}(s,\mu,a)
    =
    (I_d+\Delta t\,\hat A)s
    +
    \Delta t\,\hat{\bar A}\,\bar\mu
    +
    \Delta t\,\hat B a.
\]
Using \eqref{eq:definition_coef_HAT}, the transition mean can be rewritten as
\[
    \mathcal M_{\rho_{\Delta t}}(s,\mu,a)
    =
    e^{A\Delta t}(s-\bar\mu)
    +
    e^{\tilde A\Delta t}\bar\mu
    +
    M_{\tilde A,\Delta t}Ba.
\]
Therefore, the map $(A,\bar A,B)\mapsto \mathcal M_{\rho_{\Delta t}}$ depends on the matrices only through $e^{A\Delta t}$, $e^{\tilde A\Delta t}$, and $M_{\tilde A,\Delta t}B$. Since the matrix exponential is not injective, and since the same aliasing mechanism can be arranged simultaneously in the term $M_{\tilde A,\Delta t}B$, this map is in general not injective. Hence, the continuous-time matrices cannot, in general, be uniquely recovered from one-step transition means sampled on a fixed time grid.

The following two-dimensional example illustrates this non-identifiability. Fix $d=m=2$ and let $A=-I_2$, $\bar A=\alpha I_2$, and $B=I_2$, so that $\tilde A=(\alpha-1)I_2$, with $\alpha\neq1$. Let
\[
    J=
    \begin{pmatrix}
        0 & 1\\
        -1 & 0
    \end{pmatrix},
    \qquad
    \omega_k:=\frac{2\pi k}{\Delta t},
    \qquad
    k\in\mathbb Z\setminus\{0\}.
\]
Define $A^{(k)}:=-I_2+\omega_kJ$, $\tilde A^{(k)}:=(\alpha-1)I_2+\omega_kJ$, $\bar A^{(k)}:=\tilde A^{(k)}-A^{(k)}=\alpha I_2$, and $B^{(k)}:=(\alpha-1)^{-1}\tilde A^{(k)}$. Since $e^{2\pi kJ}=I_2$, we have $e^{A^{(k)}\Delta t}=e^{-\Delta t}e^{2\pi kJ}=e^{-\Delta t}I_2=e^{A\Delta t}$, and similarly $e^{\tilde A^{(k)}\Delta t}=e^{\tilde A\Delta t}$. Moreover, using $M_{\tilde A,\Delta t}=\tilde A^{-1}(e^{\tilde A\Delta t}-I_2)$, we obtain
\[
    M_{\tilde A^{(k)},\Delta t}B^{(k)}
    =
    (\alpha-1)^{-1}
    \bigl(e^{\tilde A^{(k)}\Delta t}-I_2\bigr)
    =
    (\alpha-1)^{-1}
    \bigl(e^{\tilde A\Delta t}-I_2\bigr)
    =
    M_{\tilde A,\Delta t}B.
\]
Consequently, $\mathcal M_{\rho_{\Delta t}}^{(k)}(s,\mu,a)
    =
    \mathcal M_{\rho_{\Delta t}}(s,\mu,a)$ for all $(s,\mu,a)$ even though $(A^{(k)},\bar A^{(k)},B^{(k)})\neq(A,\bar A,B)$. Thus, the matrices are not identifiable from one-step transition means. The same obstruction persists for any finite collection of one-step observations sampled on the same time grid when only first-moment information is used. This motivates formulating the LQR-MF-PhiBE directly in terms of the one-step coefficients $\hat A,\hat{\bar A},\hat B$.
\end{remark}

The following observation is the key reason why the LQR-MF-PhiBE can be simplified in the additive-noise setting.

\begin{remark}
\label{rem:additive-noise-same-feedback}
In the additive-noise setting considered in \eqref{eq:setting_lqr_for_phibe}, the diffusion coefficient
$\gamma$ enters the Riccati system only through the scalar equation for
$R$. Since the constant $R$ does not enter the feedback law, the
deterministic dynamics obtained by setting $\gamma=0$ and the stochastic
dynamics with $\gamma\neq0$ induce the same optimal policy. They differ
only through the constant term of the value function.
\end{remark}

Motivated by \Cref{rem:additive-noise-same-feedback}, we now introduce a
drift-only approximation of the optimal PhiBE in the LQR setting. The
purpose is not to approximate the full infinitesimal generator, but to
recover the optimal feedback law. Since the additive diffusion affects
only the constant part of the value function, we estimate the drift from
one-step transition data and impose a zero second-order term in the
surrogate problem.

Now, in order to introduce MF-PhiBE in the LQR setting, let us note that by taking the first moment in \eqref{eq:LQR-phibe-transition-law} and using
\eqref{eq:estimated-drift-policy-mkv22}, we obtain
\begin{equation}
\label{eq:LQR-phibe-drift-estimator}
    \hat b(s,\mu,a)
    =
    \hat A s+\hat{\bar A}\,\bar\mu+\hat B a.
\end{equation}

\begin{definition}[LQR-MF-PhiBE]
\label{def:LQR-MF-PhiBE}
For a fixed policy $\pi\in\Pi_{\mathrm{add}}$, the LQR-MF-PhiBE is the drift-only MF-PhiBE
\begin{equation}
\label{eq:LQR-MF-PhiBE-evaluation}
(\mathcal L_{\hat b,0}^{\pi}-\beta)\hat{\mathcal V}^{\pi}(\mu)
+
r_\lambda^\pi(\mu)
=
0,
\qquad
\mu\in\mathcal P_2(\mathbb R^d).
\end{equation}
Here $\mathcal L_{\hat b,0}^{\pi}$ denotes the mean-field generator in
\eqref{eq:L-pi-mkv} with drift $\hat b^\pi$, obtained from
\eqref{eq:LQR-phibe-drift-estimator}, and with zero diffusion matrix.
Whenever \eqref{eq:LQR-MF-PhiBE-evaluation} admits a unique solution in
$\mathcal C^2_{\mathrm{poly}}(\mathcal P_2(\mathbb R^d))$ represented by
the auxiliary flow introduced in
\eqref{eq:LQR-MF-PhiBE-auxiliary-dynamics}, we denote that solution by
$\hat{\mathcal V}^\pi$. The corresponding optimal value is
\begin{equation}
\label{eq:LQR-MF-PhiBE-optimal-value}
\hat{\mathcal V}^*(\mu)
:=
\sup_{\pi}\hat{\mathcal V}^{\pi}(\mu),
\qquad
\mu\in\mathcal P_2(\mathbb R^d),
\end{equation}
where the supremum is taken over the policies for which this
representation is finite and the auxiliary transversality condition
holds. If a policy $\hat\pi^*$ attains this supremum for every $\mu$, it
is called an optimal LQR-MF-PhiBE policy.
\end{definition}
 
The following theorem quantifies the error between the optimal feedback mean of the exact entropy-regularized LQR problem and the optimal feedback mean obtained from the LQR-MF-PhiBE approximation.

\begin{theorem}[Error estimate between the exact and LQR-MF-PhiBE optimal policies]
\label{thm:policy_error_phibe_mkv}
Assume \Cref{ass:LQR-MKV}. Then there exists $\Delta t_*>0$ such that, for every $\Delta t\in(0,\Delta t_*]$, the optimal LQR-MF-PhiBE policy is given by
\[
  \hat\pi^*(\cdot\mid s,\mu)
  =
  \mathcal N\left(
    -N^{-1}\hat B^\top
    \bigl(
      \hat K(s-\bar\mu)+\hat\Lambda\bar\mu
    \bigr),
    \frac{\lambda}{2}N^{-1}
  \right),
\]
where $(\hat K,\hat\Lambda)$ are the stabilizing solutions of \eqref{eq:ARE_SIMPLIFY_RICCATI_K}-\eqref{eq:ARE_SIMPLIFY_RICCATI_LAMBDA} with $(A,\bar A,B)$ replaced by $(\hat A,\hat{\bar A},\hat B)$. Let $(K,\Lambda)$ be the stabilizing solution of \eqref{eq:ARE_SIMPLIFY_RICCATI_K}-\eqref{eq:ARE_SIMPLIFY_RICCATI_LAMBDA}. Then, for every $\Delta t\in(0,\Delta t_*]$, every $s\in\mathbb R^d$, and every $\mu\in\mathcal P_2(\mathbb R^d)$, we have
\begin{align}
\label{eq:error_policy_mean_lqr_phibe}
  \begin{aligned}
        \mathcal{W}_2(\hat \pi^*(\cdot\mid s,\mu),\pi^*(\cdot\mid s,\mu))
    &\le
    \|N^{-1}\|\Delta t 
    \Biggl[
    \Biggl(
      2\frac{\|B\|+C_B\Delta t}{\lambda_{\min}(Q)^2}
      \left(
         C_{K,1}
        +
        \beta \,C_{K,2}
      \right)
      +
      C_{K,3}
    \Biggr)
    \|s-\bar\mu\|
    \\
    &\qquad\qquad\qquad
    +
    \Biggl(
      2\frac{\|B\|+C_B\Delta t}
      {\lambda_{\min}(Q+\bar Q)^2}
      \left(
        C_{\Lambda,1}
        +
        \beta\, C_{\Lambda,2}
      \right)
      +
      C_{\Lambda,3}
    \Biggr)
    \|\bar\mu\|
    \Biggr].
  \end{aligned}
\end{align}
Here $C_B$ and $C_{K,i},C_{\Lambda,i}$, $i=1,2,3$, are defined in \eqref{eq:constant_C_A_b}-\eqref{eq:constant_C_K_1_C_Lambd_2}. These constants are independent of $\Delta t$ and are expressed as products of $\|A\|$, $\|\bar A\|$, $\|B\|$, $\|K\|$, $\|\Lambda\|$, and $\|N^{-1}\|$. The factors involving $\beta$, $\lambda_{\min}(Q)$, and $\lambda_{\min}(Q+\bar Q)$ are kept separate in \eqref{eq:error_policy_mean_lqr_phibe}. For the value function, there exist constants $C>0$ and $\Delta t_*>0$ such that, for every $\Delta t\in(0,\Delta t_*]$ and every $\mu\in\mathcal P_2(\mathbb R^d)$,
\begin{align}\label{eq:error_value_functions_LQR}
    \left|
\mathcal V^{\pi^*}(\mu)-\mathcal V^{\hat\pi^*}(\mu)
\right|
\le
C\Delta t^2\bigl(1+m_2(\mu)\bigr).
\end{align}
Let $\bar\pi$ and $\widehat{\bar\pi}$ denote the means of $\pi$ and $\hat\pi$, respectively. Finally, for the corresponding deterministic non-entropic problem with $\beta=\gamma=\lambda=0$, if $\Lambda\tilde A=\tilde A^\top\Lambda$, then
\begin{align}
\label{eq:same_expectation_optimal_policies}
      \mathbb{E}_{s\sim\mu}\bigl[\widehat{\bar\pi}^*(s,\mu)-\bar\pi^*(s,\mu)\bigr]=0,
  \qquad
  \mu\in\mathcal P_2(\mathbb R^d).
\end{align}
\end{theorem}
The proof of \Cref{thm:policy_error_phibe_mkv} is given in \Cref{app:LQR-proofs}.

\begin{remark}
\label{rem:undiscounted-LQR-MF-PhiBE}
\Cref{thm:policy_error_phibe_mkv} provides an $O(\Delta t)$ policy error. In the deterministic undiscounted case
$\beta=\gamma=\lambda=0$, under the additional condition
$\Lambda\tilde A=\tilde A^\top\Lambda$, the proof yields
$\hat\Lambda\hat B=\Lambda B$, which implies exact equality of the population-averaged feedback means.  Note that this condition holds automatically when the state is one-dimensional. The statement with $\lambda=0$ is understood for the corresponding deterministic control problem. The conditions $\gamma=0$ and $\lambda=0$ concern the interpretation as
an undiscounted infinite-horizon control problem. Indeed,
\Cref{lem:m2-bound-beta-zero-lqr-mkv} gives
\[
m_2(\mu_t)
\le
Ce^{-2\kappa t}m_2(\mu_0)+C|\gamma|^2,
\]
so a nonzero additive diffusion generally prevents integrability of the
undiscounted quadratic running cost. Moreover, when $\beta=0$, the
scalar equation for $R$ becomes a compatibility condition. In the
deterministic non-entropic problem $\gamma=\lambda=0$, the additive
constant may be fixed by imposing, for example,
$\mathcal V(\delta_0)=0$.
\end{remark}
\Cref{fig:conver_error_LQR_MF_PHIBE} illustrates the convergence rates predicted by \Cref{thm:policy_error_phibe_mkv}. The left plot shows the error between $\pi^*$ and $\widehat\pi^*$ in the $\mathcal{E}_2$-distance, defined in \eqref{eq:difinition_E_2}. In view of \eqref{eq:relation_W_2_E_2}, this is consistent with the linear convergence rate in the $\mathcal{W}_2$-distance. The middle plot corroborates the quadratic convergence rate for the value function. Finally, the right plot shows that, for each displayed value of $\Delta t$, as $\beta\to0$ one has
$\mathbb{E}_{s\sim\mu}[\widehat{\bar\pi}^*(s,\mu)]\to \mathbb{E}_{s\sim\mu}[\bar\pi^*(s,\mu)]$,
with the limit attained at $\beta=0$.

\begin{figure}
    \centering
    \includegraphics[width=\linewidth]{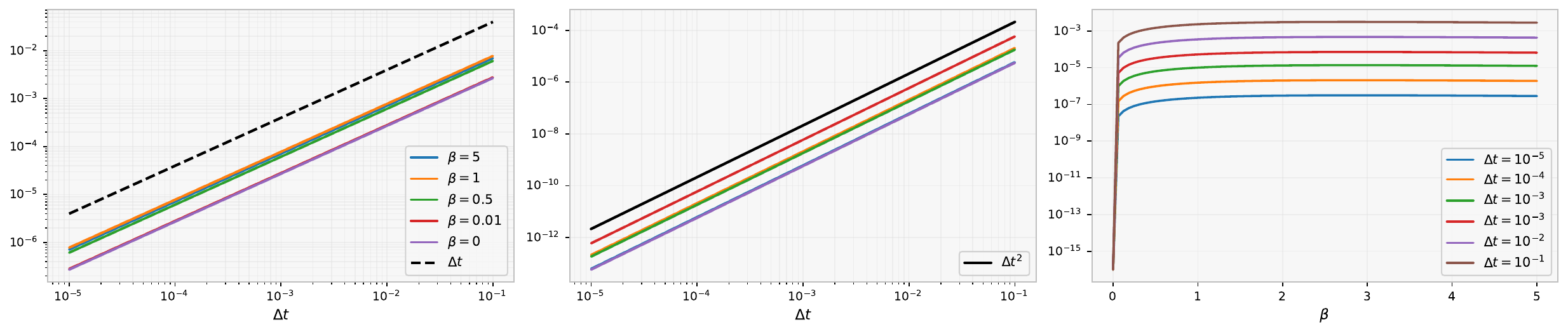}
    \caption{Numerical validation of \Cref{thm:policy_error_phibe_mkv}. Left: error between $\pi^*$ and $\hat \pi^*$ in the $\mathcal{E}_2$ distance. Center: error between the optimal value $V^{\pi^*}$ and the value $V^{\hat\pi^*}$ induced by the LQR-MF-PhiBE optimal policy. Right: error between the corresponding expected feedback means, $\mathbb E[\bar\pi^*]$ and $\mathbb E[\widehat{\bar\pi}^*]$.}\label{fig:conver_error_LQR_MF_PHIBE}
\end{figure}

\section{Model-free actor-critic algorithm}
\label{sec:algorithm}

We now describe a model-free actor-critic implementation based on the
MF-PhiBE introduced in \Cref{sec:MF-phibe} and the policy gradient
theorem introduced in \Cref{sec:policy-gradient}. The algorithm is
model-free in the sense that it does not require knowledge of $b$ and
$\sigma$. Instead, it relies on observed data.

Fix a parametrized family of randomized feedback policies $\{\pi_\omega\}_{\omega\in\mathbb R^p}$. For a given $\omega$, we assume access to state--law pairs $(s_\eta,\mu_\eta)$ visited under $\pi_\omega$. At each pair, an action $a_\eta\sim \pi_\omega(\cdot\mid s_\eta,\mu_\eta)$ is sampled and a local frozen-action transition of length $\Delta t$ is
generated as in
\eqref{eq:estimated-drift-policy-mkv22}--
\eqref{eq:estimated-covariance-policy-mkv2}. The resulting dataset is
\[
\mathcal D
=
\left\{
(s_\eta,\mu_\eta,a_\eta,s_\eta',r_\eta)
\right\}_{\eta=1}^{N_D}.
\]
Here $s_\eta\sim\mu_\eta$, $s_\eta'$ is the endpoint of the corresponding
local frozen-action transition, and $r_\eta=r(s_\eta,\mu_\eta,a_\eta)$. The index $\eta$ may enumerate independent branching transitions or a
flattened collection obtained from several visited law trajectories.

The law $\mu_\eta$ may be represented empirically. A finite-dimensional
law representation is sufficient only when the chosen policy, reward,
transition sampler, and critic can all be evaluated from those law
features. For example, the first two moments are sufficient in the LQR
setting.

\smallbreak

For notational convenience, we attach a nonnegative weight $w_\eta$ to
each sample. These weights specify the empirical measure with respect to
which the MF-PhiBE residual is projected. For instance, one may take
uniform weights $w_\eta=1/N_D$, or, when the data come from a time
trajectory and one wants to approximate a discounted time integral, one
may take $w_\eta=e^{-\beta t_\eta}\Delta t$. The formulas below are
independent of this particular choice. 

For each sample we define the one-step estimates
\begin{align}
\label{eq:coef_est}
     \widehat b_\eta
  :=
  \frac{s_\eta'-s_\eta}{\Delta t},
  \qquad
  \widehat\Sigma_\eta
  :=
  \frac{(s_\eta'-s_\eta)(s_\eta'-s_\eta)^\top}{\Delta t}.
\end{align}
Conditionally on $(s_\eta,\mu_\eta,a_\eta)$, the quantities
$\widehat b_\eta$ and $\widehat\Sigma_\eta$ are unbiased estimators of
$\hat b(s_\eta,\mu_\eta,a_\eta)$ and
$\hat\Sigma(s_\eta,\mu_\eta,a_\eta)$, introduced in
\eqref{eq:estimated-drift-policy-mkv22} and
\eqref{eq:estimated-covariance-policy-mkv2}, respectively. We also
define the entropy-regularized observed reward
\[
  r_{\lambda,\eta}
  :=
  r_\eta
  -
  \lambda\log p_{\pi_\omega}(a_\eta\mid s_\eta,\mu_\eta).
\]\subsection{Policy evaluation}
\label{subsec:critic-algorithm}

For a fixed policy $\pi_\omega$, we approximate the MF-PhiBE using the
dataset $\mathcal D$. For any smooth functional
$F:\mathcal P_2(\mathbb R^d)\to\mathbb R$, define
\[
  \widehat{\mathcal L}_\eta F
  :=
  \widehat b_\eta\cdot
  \partial_\mu F(\mu_\eta)(s_\eta)
  +
  \frac12\,\widehat\Sigma_\eta:
  D_\xi\partial_\mu F(\mu_\eta)(s_\eta).
\]
This quantity should be understood as the sample contribution associated
with the averaged generator in the MF-PhiBE. More precisely, after
averaging over $s_\eta\sim\mu_\eta$, over
$a_\eta\sim\pi_\omega(\cdot\mid s_\eta,\mu_\eta)$, and over the
one-step transition, it gives
\[
  \mathbb E\!\left[
    \widehat{\mathcal L}_\eta F
    \,\middle|\,
    \mu_\eta
  \right]
  =
  (\mathcal L_{\hat b,\hat\Sigma}^{\pi_\omega}F)(\mu_\eta).
\]
Similarly, $r_{\lambda,\eta}$ is the corresponding sample contribution
for the entropy-regularized reward. Therefore, the model-free approximation of the MF-PhiBE is imposed at
the sampled laws $\{\mu_\eta\}_{\eta=1}^{N_D}$ as
\begin{align}\label{eq:empirical_pde_mfphibe}
  \widehat{\mathcal L}_\eta \hat{\mathcal V}^{\pi_\omega}
  -
  \beta\hat{\mathcal V}^{\pi_\omega}(\mu_\eta)
  +
  r_{\lambda,\eta}
  =
  0,
  \qquad
  \eta=1,\dots,N_D.
\end{align}
This is a sampled, model-free counterpart of the MF-PhiBE. The operator $\widehat{\mathcal L}_\eta$ is not a pointwise generator, but a sample contribution whose conditional expectation recovers the averaged estimated generator. The empirical equation \eqref{eq:empirical_pde_mfphibe} can be handled with several critic approximation schemes, including nonlinear function approximation by neural networks; see, for instance, \cite{pham2025actor}. For simplicity and illustration, we solve \eqref{eq:empirical_pde_mfphibe} by a Galerkin projection, so that the value functional is approximated by a finite-dimensional ansatz. The resulting finite-dimensional critic is related to linear function approximation, finite-dimensional approximations, and Q-learning methods for mean-field control \cite{BayraktarKaraJMLR,BayraktarKaraBauerle}; here, however, it is used to approximate the continuous-time MF-PhiBE critic. Let
$\{\phi_1,\dots,\phi_n\}$ be smooth functions on
$\mathcal P_2(\mathbb R^d)$, and define
$\mathbb V_n:=\operatorname{span}\{\phi_1,\dots,\phi_n\}$ and
$\Phi(\mu):=(\phi_1(\mu),\dots,\phi_n(\mu))^\top$. We approximate the
value functional $\hat{\mathcal V}^{\pi_\omega}$ by
\begin{equation}\label{eq:critic-linear-ansatz-mkv}
  \mathcal V_\theta(\mu)
  :=
  \theta^\top\Phi(\mu)
  =
  \sum_{i=1}^n\theta_i\phi_i(\mu),
\end{equation}
with $\theta\in\mathbb R^n$. The empirical Galerkin formulation consists
in finding $\theta^*(\omega)\in\mathbb R^n$ such that
\begin{equation}\label{eq:empirical-galerkin-mkv}
  \sum_{\eta=1}^{N_D}
  w_\eta
  \Bigl[
    \widehat{\mathcal L}_\eta\mathcal V_{\theta^*(\omega)}
    -
    \beta\mathcal V_{\theta^*(\omega)}(\mu_\eta)
    +
    r_{\lambda,\eta}
  \Bigr]
  \phi_j(\mu_\eta)
  =
  0,\qquad\text{for every }j=1,\dots,n.
\end{equation}
As described in \cite[Section~5]{firstpaper}, a convenient choice for
the basis functions is given by cylindrical features of the law. Specifically, let
$\zeta_1,\dots,\zeta_K\in C_b^2(\mathbb R^d)$ and define
\[
  m(\mu)
  :=
  \bigl(m_1(\mu),\dots,m_K(\mu)\bigr),
  \qquad
  m_k(\mu)
  :=
  \int_{\mathbb R^d}\zeta_k(\xi)\,\mu(d\xi).
\]
In particular, if the basis functions defining $\mathbb V_n$ are chosen
as cylindrical functions, namely
$\phi_i(\mu)=\psi_i(m(\mu))$, with
$\psi_i:\mathbb R^K\to\mathbb R$ smooth, then
\[
  \partial_\mu\phi_i(\mu)(\xi)
  =
  \sum_{k=1}^K
  \partial_{m_k}\psi_i(m(\mu))\,\nabla\zeta_k(\xi),
  \qquad
  D_\xi\partial_\mu\phi_i(\mu)(\xi)
  =
  \sum_{k=1}^K
  \partial_{m_k}\psi_i(m(\mu))\,D^2\zeta_k(\xi).
\]
Then, setting $m_\eta:=m(\mu_\eta)$ for $\eta=1,\dots,N_D$, equation
\eqref{eq:empirical-galerkin-mkv} can be written as the linear system
\begin{equation}\label{eq:critic-linear-system-mkv}
  A(\omega)\theta^*(\omega)=b(\omega),
\end{equation}
where
\[
\begin{aligned}
  A_{ji}(\omega)
  :=
  \sum_{\eta=1}^{N_D}
  w_\eta
  \psi_j(m_\eta)
  \bigg(
    \beta\psi_i(m_\eta)
    &-
    \widehat b_\eta\cdot
    \sum_{k=1}^K
    \partial_{m_k}\psi_i(m_\eta)\nabla\zeta_k(s_\eta)-
    \frac12\widehat\Sigma_\eta:
    \sum_{k=1}^K
    \partial_{m_k}\psi_i(m_\eta)D^2\zeta_k(s_\eta)
  \bigg),
\end{aligned}
\]
and
\[
  b_j(\omega)
  :=
  \sum_{\eta=1}^{N_D}
  w_\eta\,
  r_{\lambda,\eta}\,
  \psi_j(m_\eta).
\]
Whenever $A(\omega)$ is nonsingular,
\eqref{eq:critic-linear-system-mkv} defines
$\theta^*(\omega)=A(\omega)^{-1}b(\omega)$ and hence the critic
$\hat{\mathcal V}^{\pi_\omega}_n
:=
\mathcal V_{\theta^*(\omega)}$. 
\subsection{Policy gradient}
\label{subsec:actor-algorithm}

We now update the policy parameter using the policy gradient formula in
\Cref{thm:parametric-policy-gradient}. For the parametric policy
$\pi_\omega$, the exact gradient direction is expressed in terms of the
instantaneous advantage function $q^{\pi_\omega}$ and the score
$\nabla_\omega\log p_{\pi_\omega}$. In the model-free implementation,
$q^{\pi_\omega}$ is replaced by a sample approximation computed from the
critic $\hat{\mathcal V}^{\pi_\omega}_n$ obtained in the policy
evaluation step.

For each sample $\eta=1,\dots,N_D$, define
\begin{align}\label{eq:sample_q_function}
      \widehat q_\eta^\omega
  :=
  r_{\lambda,\eta}
  +
  \widehat b_\eta\cdot
  \partial_\mu\hat{\mathcal V}^{\pi_\omega}_n(\mu_\eta)(s_\eta)
  +\frac12\,\widehat\Sigma_\eta:
D_\xi\partial_\mu
\hat{\mathcal V}^{\pi_\omega}_n(\mu_\eta)(s_\eta)
-\beta\hat{\mathcal V}^{\pi_\omega}_n(\mu_\eta).
\end{align}
This quantity is the model-free approximation of the instantaneous
advantage $q^{\pi_\omega}(\mu_\eta,s_\eta,a_\eta)$, obtained by
replacing the exact value functional by the critic
$\hat{\mathcal V}^{\pi_\omega}_n$ and the infinitesimal coefficients by
the one-step estimates. Consequently, the empirical policy-gradient direction reads as
\begin{equation}\label{eq:model-free-gradient-mkv}
  \widehat{\nabla J}(\omega)
  :=
  \sum_{\eta=1}^{N_D}
  w_\eta\,
  \widehat q_\eta^\omega\,
  \nabla_\omega
  \log p_{\pi_\omega}(a_\eta\mid s_\eta,\mu_\eta).
\end{equation}
Following \Cref{thm:parametric-policy-gradient}, the weights in \eqref{eq:model-free-gradient-mkv} should be taken as $w_\eta=e^{-\beta t_\eta}\Delta t$. In the infinite-horizon setting, the discount factor turns the time-occupancy measure into a finite measure. In numerical implementations, however, the data are collected over a finite time horizon. Since the time integral is then already truncated, one may also use uniform weights $w_\eta=1/N_D$ and interpret the resulting direction as a finite-horizon empirical surrogate of the discounted policy gradient. 

Finally, in the cylindrical case, with $m_\eta=m(\mu_\eta)$, the sample
advantage can be written explicitly as
\begin{align}\label{eq:advantage_func}
\begin{aligned}
  \widehat q_\eta^\omega
  =
  r_{\lambda,\eta}
  +
  \sum_{i=1}^n\theta_i^*(\omega)
  \bigg[
\widehat b_\eta\cdot
\sum_{k=1}^K
\partial_{m_k}\psi_i(m_\eta)\nabla\zeta_k(s_\eta)
+\frac12\widehat\Sigma_\eta:
\sum_{k=1}^K
\partial_{m_k}\psi_i(m_\eta)D^2\zeta_k(s_\eta)
-\beta\psi_i(m_\eta)
\bigg].
\end{aligned}
\end{align}
Therefore, in practice, the actor step only requires the
finite-dimensional features $m_\eta$, the one-step transition estimates,
the entropy-regularized observed rewards, and the score of the policy.
\subsubsection{Some gradient-based algorithms}\label{subsec:some_gb_algo}

For completeness, we summarize standard policy-gradient variants that can be used in the improvement step with the estimator \eqref{eq:model-free-gradient-mkv}. All these variants can be used in the policy-improvement step without changing the MF-PhiBE critic. 

Let us introduce the corresponding Fisher information matrix, formally given by
\[
  F(\omega)
  :=
  \int_{\mathcal P_2(\mathbb R^d)\times\mathbb R^d}
  \int_{\mathcal A}
  \nabla_\omega\log p_{\pi_\omega}(a\mid \xi,\eta)
  \nabla_\omega\log p_{\pi_\omega}(a\mid \xi,\eta)^\top
  \pi_\omega(da\mid \xi,\eta)\,
  \rho_{\mu_0}^{\pi_\omega}(d\eta,d\xi),
\]
where $\rho_{\mu_0}^{\pi_\omega}$ is the discounted occupancy measure introduced in
\Cref{def:discounted_ocupacy_MKV}.   In practice, $F(\omega)$ is replaced by the corresponding empirical weighted estimator computed from the same samples used in \eqref{eq:model-free-gradient-mkv}.

\smallskip

\noindent\textbf{Vanilla policy gradient.}
The parameters are updated by gradient ascent,
\[
  \omega_{k+1}=\omega_k+\alpha_k \widehat{\nabla J}(\omega_k),
\]
where $\alpha_k>0$ is a learning rate.

\smallskip

\noindent\textbf{Natural policy gradient.}
The gradient is preconditioned by the Fisher metric,
\[
  \omega_{k+1}=\omega_k+\alpha_k F(\omega_k)^{-1}\widehat{\nabla J}(\omega_k).
\]
If $F(\omega_k)$ is singular or ill-conditioned, one may replace $F(\omega_k)^{-1}$ by a regularized inverse, for instance $(F(\omega_k)+\varepsilon I)^{-1}$ with $\varepsilon>0$.

\smallskip

\noindent\textbf{Trust region policy optimization.}
A trust-region step is obtained by maximizing a first-order approximation of the objective while constraining the average KL divergence between the old and new policies. In the present notation, this gives
\[
  \max_{\omega}\ \langle \widehat{\nabla J}(\omega_k),\omega-\omega_k\rangle
  \quad\text{subject to}\quad
  \int_{\mathcal P_2(\mathbb R^d)\times\mathbb R^d}
  \mathrm{KL}\!\left(
    \pi_{\omega_k}(\cdot\mid \xi,\mu)\,\|\,\pi_{\omega}(\cdot\mid \xi,\mu)
  \right)\rho_{\mu_0}^{\pi_{\omega_k}}(d\mu,d\xi)
  \le \delta,
\]
where $\delta>0$ is the trust-region radius. Using a second-order expansion of the KL constraint around $\omega_k$ yields the quadratic subproblem
\[
  \max_{\Delta\omega}\ \langle \widehat{\nabla J}(\omega_k),\Delta\omega\rangle
  \quad\text{subject to}\quad
  \frac12\,\Delta\omega^\top F(\omega_k)\Delta\omega\le\delta.
\]
When $F(\omega_k)$ is invertible and $\widehat{\nabla J}(\omega_k)^\top F(\omega_k)^{-1}\widehat{\nabla J}(\omega_k)>0$, the corresponding step is
\[
  \Delta\omega_k
  =
  \sqrt{\frac{2\delta}{\widehat{\nabla J}(\omega_k)^\top F(\omega_k)^{-1}\widehat{\nabla J}(\omega_k)}}\,
  F(\omega_k)^{-1}\widehat{\nabla J}(\omega_k),
  \qquad
  \omega_{k+1}=\omega_k+\Delta\omega_k.
\]

\smallskip

\noindent\textbf{Proximal policy optimization.}
PPO replaces the hard KL constraint by a clipped surrogate objective. Given an old parameter $\omega_{\mathrm{old}}$, define the likelihood ratio
\[
  R_\omega(\xi,\mu,a)
  :=
  \frac{p_{\pi_\omega}(a\mid \xi,\mu)}
       {p_{\pi_{\omega_{\mathrm{old}}}}(a\mid \xi,\mu)}.
\]
Using the sample advantage $\widehat q_\eta^{\omega_{\mathrm{old}}}$ from \eqref{eq:model-free-gradient-mkv}, the empirical clipped objective takes the form
\[
  L^{\mathrm{CLIP}}(\omega)
  :=
  \sum_{\eta=1}^{N_D}w_\eta
  \min\!\left\{
    R_\omega(s_\eta,\mu_\eta,a_\eta)\,\widehat q_\eta^{\omega_{\mathrm{old}}},
    \operatorname{clip}\!\left(R_\omega(s_\eta,\mu_\eta,a_\eta),1-\varepsilon,1+\varepsilon\right)\widehat q_\eta^{\omega_{\mathrm{old}}}
  \right\},
\]
where $\varepsilon>0$ is the clipping parameter. The policy is then updated by approximately maximizing $L^{\mathrm{CLIP}}$ over one or several optimization steps.

\begin{algorithm}[H]
\caption{Model-free actor-critic iteration based on MF-PhiBE}
\label{alg:model-free-mfphibe-actor-critic}
\begin{algorithmic}[1]
\State Choose a parametrized randomized policy family
$\{\pi_\omega\}_{\omega\in\mathbb R^p}$.
\State Choose a finite-dimensional critic space
$\mathbb V_n=\operatorname{span}\{\phi_1,\dots,\phi_n\}$ on
$\mathcal P_2(\mathbb R^d)$, preferably using cylindrical features as in
\Cref{subsec:critic-algorithm}.
\State Initialize the actor parameter $\omega^{(0)}$.
\For{$k=0,1,\dots,K-1$}
    \State \textbf{Data collection:}
    collect an on-policy dataset
    \[
    \mathcal D_k
    =
    \bigl\{
    (s_\eta,\mu_\eta,a_\eta,s_\eta',r_\eta)
    \bigr\}_{\eta=1}^{N_D}
    \]
    under the policy $\pi_{\omega^{(k)}}$.
    \State Compute the one-step estimates
    $\widehat b_\eta$, $\widehat\Sigma_\eta$ and the
    entropy-regularized rewards $r_{\lambda,\eta}$ as in
    \eqref{eq:coef_est}.
    \State Choose empirical weights $w_\eta$, for instance
    $w_\eta=e^{-\beta t_\eta}\Delta t$ or
    $w_\eta=1/N_D$.
    \State \textbf{Policy evaluation:}
    solve the empirical MF-PhiBE residual \eqref{eq:empirical_pde_mfphibe}, for instance through the Galerkin linear system \eqref{eq:critic-linear-system-mkv}, and define the critic
    \[
    \hat{\mathcal V}^{\pi_{\omega^{(k)}}}_n
    :=
    \mathcal V_{\theta^{(k)}}.
    \]
    \State \textbf{Policy improvement:}
    construct the sample advantages
    $\widehat q_\eta^{\omega^{(k)}}$ using \eqref{eq:sample_q_function} and
    $\hat{\mathcal V}^{\pi_{\omega^{(k)}}}_n$.
    \State Estimate the policy-gradient direction
    $\widehat{\nabla J}(\omega^{(k)})$ using
    \eqref{eq:model-free-gradient-mkv}.
    \State Update the actor parameter, for example by vanilla policy-gradient,
    \[
    \omega^{(k+1)}
    =
    \omega^{(k)}
    +
    \alpha_k\,\widehat{\nabla J}(\omega^{(k)}).
    \]
    Alternatively, replace this step with a natural policy-gradient, TRPO, PPO, or any other gradient-based update.
\EndFor
\State Return the final policy $\pi_{\omega^{(K)}}$ and the associated
critic $\hat{\mathcal V}^{\pi_{\omega^{(K)}}}_n$.
\end{algorithmic}
\end{algorithm}

\section{Numerical examples}\label{sec:numerics}

\subsection{Example 1: Model-free systemic risk}
\label{subsec:numerics-lqr-phibe}

We first consider the systemic-risk specialization of the LQR McKean-Vlasov control problem studied in \Cref{sec:LQR}. This example admits an explicit Riccati characterization of the continuous-time optimal policy and therefore provides a benchmark to compare our algorithm.

\subsubsection{Problem setup}

We work in dimension $d=m=1$ and impose the systemic-risk relation $\bar A=-A$ in \eqref{eq:setting_lqr_for_phibe}. The controlled population dynamics are therefore
\[
  d s_t
  =
  \bigl(A(s_t-\bar\mu_t)+B a_t\bigr)\,dt
  +
  \gamma\,dB_t,
  \qquad
  \bar\mu_t:=\int_{\mathbb R}x\,\mu_t(dx),
  \qquad
  \mu_t=\Law(s_t).
\]
The running reward is the quadratic reward introduced in \eqref{eq:reward_LQR_simplify}, and the entropy-regularized averaged reward is defined by \eqref{eq:reg_reward_avg_mkv}. The exact optimal policy is obtained from the Riccati system in \Cref{sec:LQR}. We use this explicit solution as a reference to illustrate the model-free actor-critic algorithm introduced in \Cref{sec:algorithm}.

\subsubsection{Policy parametrization, value approximation and gradient direction}

Following the explicit LQR structure, we use the Gaussian policy class
\begin{align}\label{eq:definition_paramtric_policy_ex1}
     \pi_\omega(\cdot\mid s,\mu)
  =
  \mathcal N\left(
    \omega_1(s-\bar\mu)+\omega_2\bar\mu,\,
    \frac{\lambda}{2}N^{-1}
  \right),
  \qquad
  \omega=(\omega_1,\omega_2)\in\mathbb R^2.
\end{align}
The critic is parametrized by the quadratic law-value ansatz
\begin{align}\label{eq:value_anzatz_ex1}
      \mathcal V_\theta(\mu)
  =
  \theta_1+\theta_2\,\bar\mu^2+\theta_3\,\operatorname{Var}(\mu),
  \qquad
  \theta=(\theta_1,\theta_2,\theta_3)\in\mathbb R^3.
\end{align}
In this finite-dimensional setting, the Lions derivative of the critic is $\partial_\mu \mathcal V_\theta(\mu)(s)
  =
  2\theta_2\bar\mu
  +
  2\theta_3(s-\bar\mu).$ As in the LQR-MF-PhiBE approximation considered in \Cref{sec:LQR}, we use the drift-only surrogate and set $\widehat\Sigma=0$ in the empirical advantage. Thus, for a data point $(s_\eta,\mu_\eta,a_\eta,s'_\eta,r_\eta)$, with $\bar\mu_\eta=\int x\,\mu_\eta(dx)$, the sample advantage used in the actor update is
\begin{align}\label{eq:explicit_q_LQR}
      \widehat q_\eta
  =
  r_\eta
  -
  \lambda\log p_{\pi_\omega}(a_\eta\mid s_\eta,\mu_\eta)
  +\widehat b_\eta
\left(2\theta_2\bar\mu_\eta+2\theta_3(s_\eta-\bar\mu_\eta)\right)
-\beta\left(
\theta_1+\theta_2\bar\mu_\eta^2
+\theta_3\operatorname{Var}(\mu_\eta)
\right),
\end{align}
 with $\hat{b}$ defined in \eqref{eq:coef_est}. Since the policy variance is fixed, the score function is
\begin{align}\label{eq:score_explicit_lqr}
  \nabla_\omega\log p_{\pi_\omega}(a_\eta\mid s_\eta,\mu_\eta)
  =
  \frac{a_\eta-\omega_1(s_\eta-\bar\mu_\eta)-\omega_2\bar\mu_\eta}
       {\lambda(2N)^{-1}}
  \begin{pmatrix}
    s_\eta-\bar\mu_\eta\\
    \bar\mu_\eta
  \end{pmatrix}.
\end{align}
Hence the empirical policy-gradient direction used in the experiment is computed by substituting \eqref{eq:explicit_q_LQR} and \eqref{eq:score_explicit_lqr} into \eqref{eq:model-free-gradient-mkv}, with either discounted weights $w_\eta=e^{-\beta t_\eta}\Delta t$ or uniform weights $w_\eta=1/N_D$.

\subsubsection{Implementation and numerical results}
Since the policy parametrization in \eqref{eq:definition_paramtric_policy_ex1} and the value-function ansatz \eqref{eq:value_anzatz_ex1} depend on the law only through $\bar\mu_\eta$ and $\operatorname{Var}(\mu_\eta)$, in this example we do not need a full empirical representation of $\mu_\eta$. Hence, for a fixed policy $\pi_\omega$, we assume that the available data set is
\[
 \mathcal D_\omega
  =
  \left\{
  \bigl(
  s_\ell^i,\bar\mu^i,\operatorname{Var}(\mu^i),
  a_\ell^i,r_\ell^i,(s_\ell^i)'
  \bigr)
  :
  i=1,\ldots,N_{\rm steps},\ \ell=1,\ldots,L
  \right\}.
\]
Here $N_{\rm steps}$ is the number of observed one-step transitions along each trajectory, and $L$ is the number of independent trajectories, obtained from independent samples of the initial law. We have $s_\ell^i\sim\mu^i$, $a_\ell^i\sim\pi_\omega(\cdot\mid s_\ell^i,\mu^i)$, $r_\ell^i=r(s_\ell^i,\mu^i,a_\ell^i)$, and $(s_\ell^i)'$ denotes the state observed after one time step $\Delta t$.

To reduce artificial sources of noise, the data are generated exactly from the LQR dynamics, thereby avoiding an additional numerical time-discretization error. More precisely, $(s_\ell^i)'$ is sampled from the explicit transition law \eqref{eq:LQR-phibe-transition-law}, while the moments of the law are propagated through the closed equations associated with the current policy: $\dot{\bar\mu}_t=(A+\bar A+B\omega_2)\bar\mu_t$ and $\frac{d}{dt}\operatorname{Var}(\mu_t)=2(A+B\omega_1)\operatorname{Var}(\mu_t)+\gamma^2$.

In the numerical experiment we use $N_{\rm steps}=50$, $L=200$, and $\Delta t=0.05$. The numerical parameters are
\begin{align}\label{eq:parameters_example1}
      A=-1,\quad
  \bar A=1,\quad
  B=1,\quad
  Q=1,\quad
  \bar Q=1,\quad
  N=0.5,\quad\beta=1,\quad
  \gamma=0.5,\quad
  \lambda=0.2.
\end{align}
The initial law is $\mu=\mathcal N(1,1)$. The empirical estimators entering \eqref{eq:explicit_q_LQR}, \eqref{eq:score_explicit_lqr} and \eqref{eq:model-free-gradient-mkv} are computed from $\mathcal D_\omega$.
\begin{figure}
    \centering
    \includegraphics[width=0.32\linewidth]{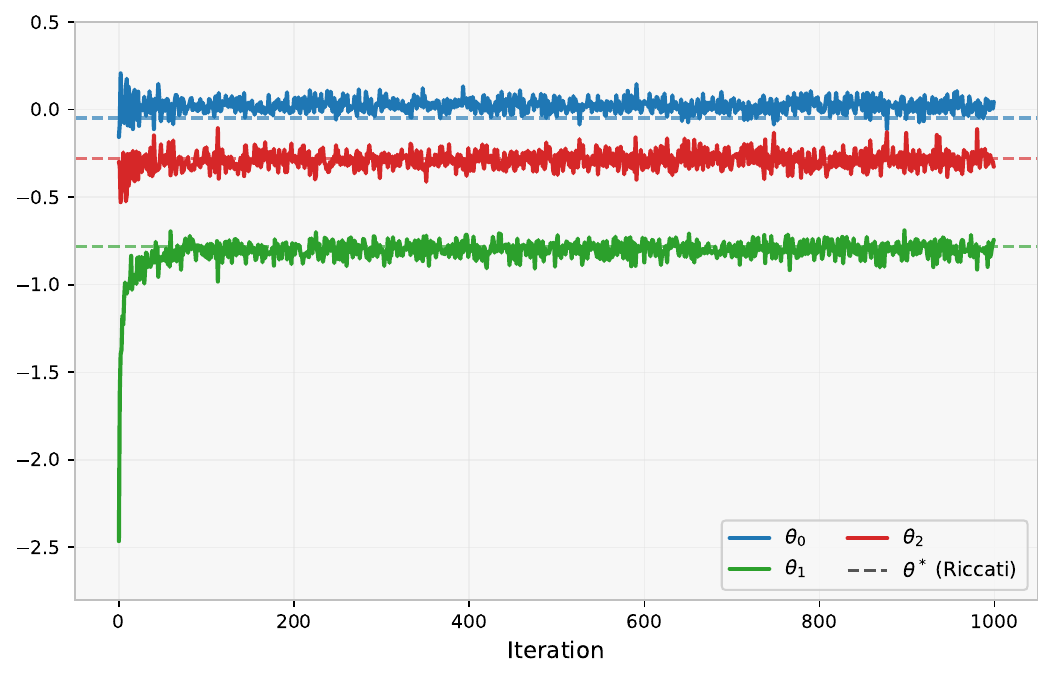}
    \includegraphics[width=0.32\linewidth]{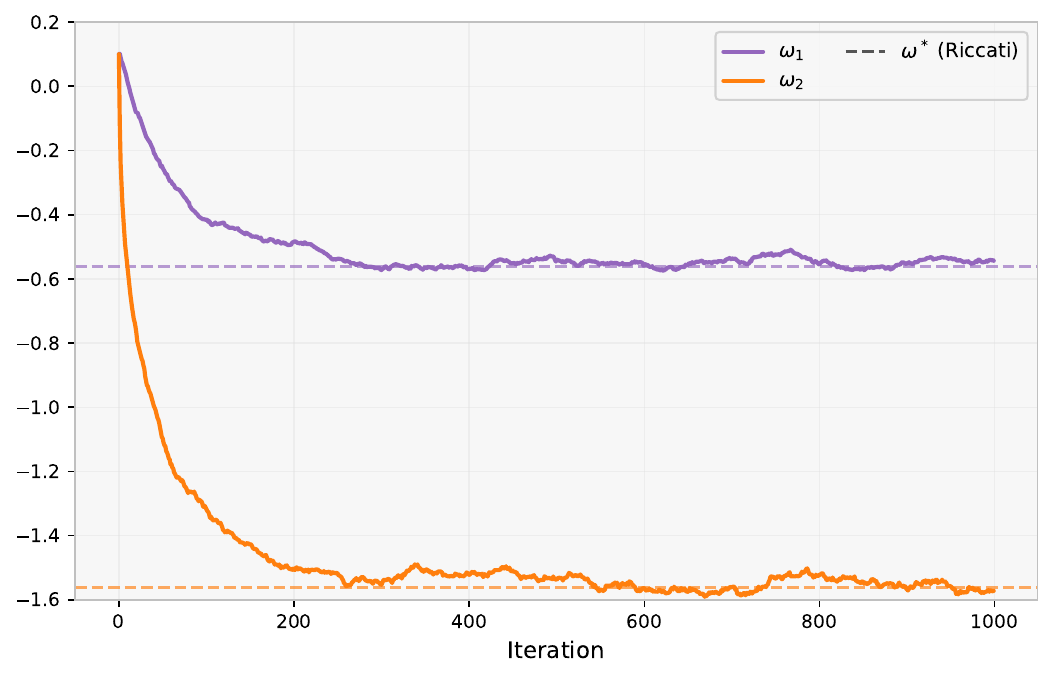}
    \includegraphics[width=0.32\linewidth]{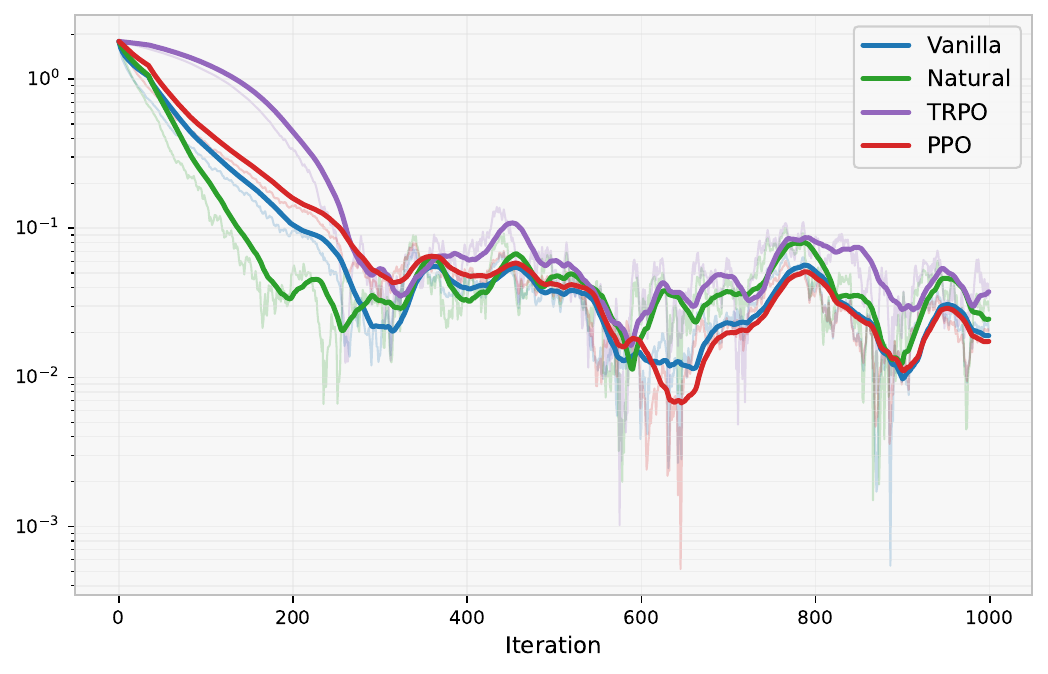}
    \caption{
    Model-free LQR-MF-PhiBE actor-critic convergence. Left: critic parameters obtained from the empirical MF-PhiBE evaluation, compared with the Riccati reference. Center: actor parameters produced by the vanilla policy-gradient update, compared with the optimal feedback parameters. Right: policy-parameter error for different update rules.
    }
    \label{fig:lqr-phibe-actor-critic-convergence}
\end{figure}

\Cref{fig:lqr-phibe-actor-critic-convergence} illustrates the behavior of the actor-critic scheme in the LQR benchmark when the discounted occupancy measure $w_\eta=e^{-\beta t_\eta}\Delta t$ is used. The left panel shows that the empirical critic recovers the quadratic value structure predicted by the Riccati solution. The center panel shows that the actor parameters approach the optimal feedback coefficients. The right panel compares several update rules, namely vanilla gradient descent, natural gradient descent, TRPO, and PPO; see \Cref{subsec:some_gb_algo}. This panel shows a consistent reduction of the policy-parameter error, with residual oscillations caused by the finite data set used at each update. Thus, the experiment confirms that the empirical MF-PhiBE critic provides a usable policy-gradient direction for recovering the continuous-time LQR feedback. Similarly, the left and center panels of \Cref{fig:lqr-phibe-sample-complexity} show analogous results for the uniform finite-horizon empirical weights $w_\eta=1/N_D$, where convergence is also observed.

\begin{figure}
    \centering
    \includegraphics[width=0.32\linewidth]{images_paper_2/example1/fig3_phibe_theta_convergence_pure_mkv_vanilla.pdf}\includegraphics[width=0.32\linewidth]{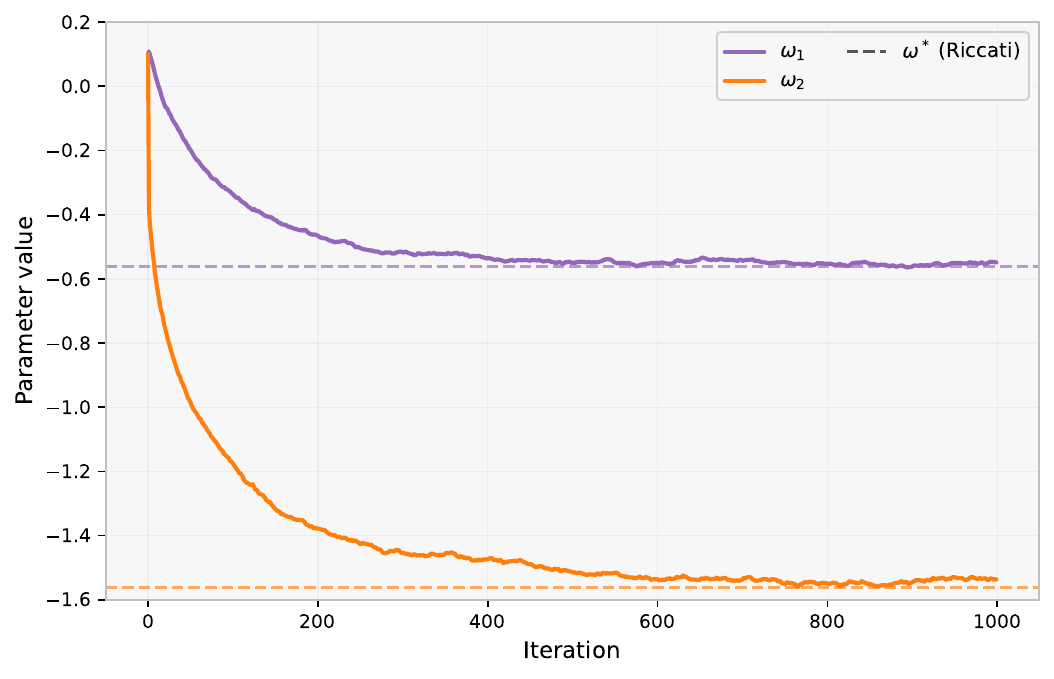}\includegraphics[width=0.32\linewidth]{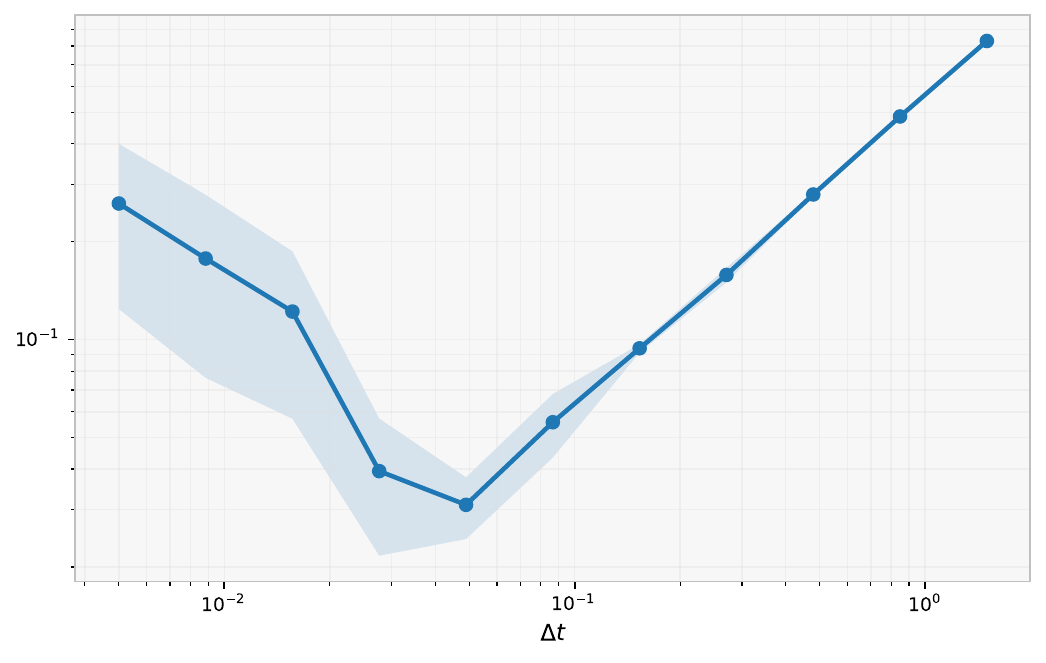}
    \caption{Uniform-weight finite-horizon performance and sample-complexity effect. Left: critic parameters obtained from the empirical MF-PhiBE using the undiscounted occupancy measure. Center: actor parameters produced by the vanilla policy-gradient update using the undiscounted occupancy measure. Right: continuous-time value error as a function of the observation step $\Delta t$, using a fixed data budget $N_D=4800$.}
    \label{fig:lqr-phibe-sample-complexity}
\end{figure}

The right panel of \Cref{fig:lqr-phibe-sample-complexity} shows the effect of decreasing $\Delta t$ while keeping the total number of observations fixed. For large and moderate values of $\Delta t$, reducing the observation step improves the approximation, as expected from the consistency of the one-step MF-PhiBE estimator. However, once $\Delta t$ becomes too small, the error starts increasing. In this regime the dominant error is no longer the finite-$\Delta t$ approximation error, but the statistical error produced by the fixed data budget. Thus, the curve exhibits a threshold beyond which more accurate time resolution is not beneficial unless the number of samples is also increased.

\subsubsection{Comparison with the time-discrete Bellman problem}
\label{subsubsec:numerics-phibe-vs-discretization}

The previous experiment evaluates the LQR-MF-PhiBE approximation of the
continuous-time control problem. We now compare this construction with a
different way of using observations at time step $\Delta t$: instead of
incorporating the one-step information into the continuous-time HJB
structure, one may first discretize the controlled dynamics and then solve
the resulting time-discrete Bellman problem.

Let $\rho_{\Delta t}(\cdot\mid s,\mu,a)$ denote the one-step transition
kernel induced by \eqref{eq:LQR-phibe-transition-law}. The corresponding
discretized entropy-regularized objective is
\[
  \mathcal V_{\Delta t}^{\pi}(\mu)
  :=
  \mathbb E\left[
    \sum_{i=0}^{\infty}
    e^{-\beta i\Delta t}\,
    \Delta t\, r_\lambda^\pi(\mu_i)
    \,\bigg|\,
    \mu_0=\mu
  \right],
\]
where the controlled Markov chain is defined by
\[
  s_{i+1}\sim \rho_{\Delta t}(\cdot\mid s_i,\mu_i,a_i),
  \qquad
  a_i\sim\pi(\cdot\mid s_i,\mu_i),
  \qquad
  \mu_i=\Law(s_i).
\]

The discrete-time problem still admits an explicit Gaussian optimizer. Set
$\delta:=e^{-\beta\Delta t}$ and $B_{\Delta t}:=\Delta t\,\hat B$, where
$\hat B$ is the action coefficient in the one-step drift estimator
\eqref{eq:LQR-phibe-drift-estimator}. Set
\[
  A_{\Delta t}:=I+\Delta t\,\hat A,
  \qquad
  \widetilde A_{\Delta t}:=I+\Delta t(\hat A+\hat{\bar A}),
  \qquad
  B_{\Delta t}:=\Delta t\,\hat B,
  \qquad
  \delta:=e^{-\beta\Delta t}.
\]
Then, the optimal policy of the
discretization problem is given by
\[
  \pi_{\Delta t}^*(\cdot\mid s,\mu)
  =
  \mathcal N\left(
    \bar{\pi}_{\Delta t}^*(s,\mu),
    \Sigma_{\Delta t}^*
  \right),
  \qquad
  \Sigma_{\Delta t}^*
  =
  \frac{\lambda\Delta t}{2}
  \left(
    \Delta t\,N
    +
    \delta
    B_{\Delta t}^{\top}K_{\Delta t}B_{\Delta t}
  \right)^{-1},
\]
with
\begin{align*}
  \bar{\pi}_{\Delta t}^*(s,\mu)
  =
  &-
  \left[
  \left(
    \Delta t\,N
    +
    \delta
    B_{\Delta t}^{\top}K_{\Delta t}B_{\Delta t}
  \right)^{-1}
  \delta
  B_{\Delta t}^{\top}K_{\Delta t}A_{\Delta t}
  \right](s-\bar\mu)
\\
  &-
  \left[
  \left(
    \Delta t\,N
    +
    \delta
    B_{\Delta t}^{\top}\Lambda_{\Delta t}B_{\Delta t}
  \right)^{-1}
  \delta
  B_{\Delta t}^{\top}\Lambda_{\Delta t}\widetilde A_{\Delta t}
  \right]\bar\mu .
\end{align*}
Here the matrices $K_{\Delta t}$ and $\Lambda_{\Delta t}$ are the stabilizing
solutions of
\begin{align*}
K_{\Delta t}
&=
\Delta t\,Q
+
\delta
A_{\Delta t}^{\top}
K_{\Delta t}
A_{\Delta t}
-
\delta^2
A_{\Delta t}^{\top}
K_{\Delta t}B_{\Delta t}
\left(
  \Delta t\,N
  +
  \delta
  B_{\Delta t}^{\top}K_{\Delta t}B_{\Delta t}
\right)^{-1}
B_{\Delta t}^{\top}K_{\Delta t}
A_{\Delta t},
\\
\Lambda_{\Delta t}
&=
\Delta t\,(Q+\bar Q)
+
\delta
\widetilde A_{\Delta t}^{\top}
\Lambda_{\Delta t}
\widetilde A_{\Delta t}
-
\delta^2
\widetilde A_{\Delta t}^{\top}
\Lambda_{\Delta t}B_{\Delta t}
\left(
  \Delta t\,N
  +
  \delta
  B_{\Delta t}^{\top}\Lambda_{\Delta t}B_{\Delta t}
\right)^{-1}
B_{\Delta t}^{\top}\Lambda_{\Delta t}
\widetilde A_{\Delta t}.
\end{align*}
The optimal policy characterization is obtained by combining the time-discrete mean-field LQ framework of \cite{carmona2019linear} with the entropy-regularized one-step Bellman minimization of \cite{GuoLiXu2025}. The former motivates the discrete mean-field LQ structure, while the latter yields a Gaussian optimizer over randomized feedback policies, with covariance depending on $\Delta t$.

For fixed $\Delta t>0$, neither the discrete optimizer
$\pi_{\Delta t}^*$ nor the LQR-MF-PhiBE optimizer $\hat\pi^*$ is expected
to coincide exactly with the continuous-time optimizer $\pi^*$. The
question is therefore which of the two finite-$\Delta t$ approximations is
closer to $\pi^*$. To compare policies, for Gaussian feedback laws
\[
  \pi_i(\cdot\mid s,\mu)
  =
  \mathcal N\left(
    -L_c^i(s-\bar\mu)-L_m^i\bar\mu,
    \Sigma_i
  \right),
  \qquad i=1,2,
\]
we use
\begin{align}\label{eq:difinition_E_2}
    \mathcal{E}_2(\pi_1,\pi_2)
    :=
    \left(
    \|L_c^1-L_c^2\|^2
    +
    \|L_m^1-L_m^2\|^2
    +
    \|\Sigma_1^{1/2}-\Sigma_2^{1/2}\|_F^2
    \right)^{1/2}.
\end{align}
In particular, for every compact set $K\subset \mathbb R^d\times\mathbb R^d$,
there exists $C_K>0$ such that
\begin{align}\label{eq:relation_W_2_E_2}
      \sup_{(s,\bar\mu)\in K}
    \mathcal W_2\bigl(
      \pi_1(\cdot\mid s,\mu),
      \pi_2(\cdot\mid s,\mu)
    \bigr)
    \le
    C_K\mathcal{E}_2(\pi_1,\pi_2).
\end{align}

\begin{figure}
  \centering
  \includegraphics[width=\linewidth]{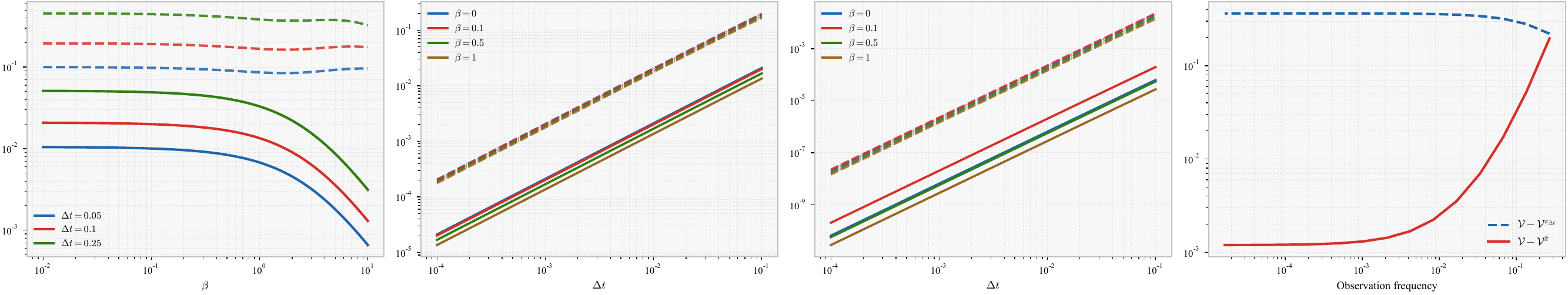}
  \caption{
  Comparison between LQR-MF-PhiBE and the time-discrete Bellman problem.
  Solid curves correspond to the LQR-MF-PhiBE approximation and dashed curves
  to the discrete-time Bellman problem. Left: policy errors
  $\mathcal{E}_2(\pi^*,\hat\pi^*)$ and
  $\mathcal{E}_2(\pi^*,\pi_{\Delta t}^*)$ as functions of the discount factor
  $\beta$, for several values of $\Delta t$. Left center: convergence of the
  policy errors as $\Delta t\to0$, for different values of $\beta$. Right center:
  continuous-time value gap
  $|\mathcal{V}^{\pi^*}-\mathcal{V}^{\pi_{\mathrm{method}}}|$ as a function
  of $\Delta t$. Right: exact continuous-time performance errors
  $\mathcal{V}-\mathcal{V}^{\hat\pi}$ and
  $\mathcal{V}-\mathcal{V}^{\pi_{\Delta t}}$ when both policies are first
  computed with $\Delta t=0.5$ and then kept fixed while the observation step
  $\Delta t_2$ is varied.
  In all panels, the errors are computed with respect to the continuous-time
  optimal value or policy.
  }
  \label{fig:lqr-phibe-vs-discrete-combined}
\end{figure}

In \Cref{fig:lqr-phibe-vs-discrete-combined}, the first panel shows how the policy error depends on the discount factor for fixed values of $\Delta t$. Throughout this experiment, we use the same parameters as in \eqref{eq:parameters_example1}. For the parameter values displayed in this experiment, the LQR-MF-PhiBE policy remains closer to $\pi^*$ than the discrete-time optimizer over the plotted range of $\beta$. The second panel confirms that both policy errors decrease as $\Delta t\to0$, as expected from the consistency of both finite-$\Delta t$ approximations. The third panel shows the same comparison at the level of the original continuous-time objective: the value gap decreases with $\Delta t$, and the relative ordering of the two methods is consistent with the policy-error plots. Finally, the fourth panel evaluates the performance of the two optimal policies as observations become more frequent. More precisely, the LQR-MF-PhiBE and discrete-time policies are computed once with $\Delta t=0.5$ and then applied to the continuous dynamics with observations taken every $h$ units of time. The plot shows that, as observations become more frequent, the LQR-MF-PhiBE policy yields a smaller exact value error than the discrete-time policy. Moreover, as $h$ decreases, the observed errors stabilize, which indicates that the remaining discrepancy is mainly due to the policy approximation itself rather than to the frequency at which the closed-loop system is observed.

Let $\operatorname{Err}_{\mathrm{PhiBE}}:=\mathcal{E}_2(\pi^*,\hat\pi^*)
$ and $\operatorname{Err}_{\mathrm{disc}}:=\mathcal{E}_2(\pi^*,\pi_{\Delta t}^*)$. To identify the parameter regimes in which one approximation is more accurate than the other, \Cref{fig:lqr-phibe-vs-discrete-heatmaps} plots the logarithmic ratio $\log(\operatorname{Err}_{\mathrm{PhiBE}}/\operatorname{Err}_{\mathrm{disc}})$. Negative values mean that LQR-MF-PhiBE has the smaller policy error, while positive values mean that the optimizer of the time-discrete Bellman problem is closer to $\pi^*$. The black curve marks the zero level set, where both errors coincide. The heatmaps in \Cref{fig:lqr-phibe-vs-discrete-heatmaps} show that the relative accuracy depends on the parameter regime. In the displayed experiments, LQR-MF-PhiBE is more accurate in a large part of the parameter range, whereas direct discretization becomes more accurate only in the regime when $A$ and $B$ are small enough. This confirms that the time-discrete Bellman problem and LQR-MF-PhiBE define two different finite-$\Delta t$ approximations of the same continuous-time optimizer.

\begin{figure}
  \centering
  \includegraphics[width=\linewidth]{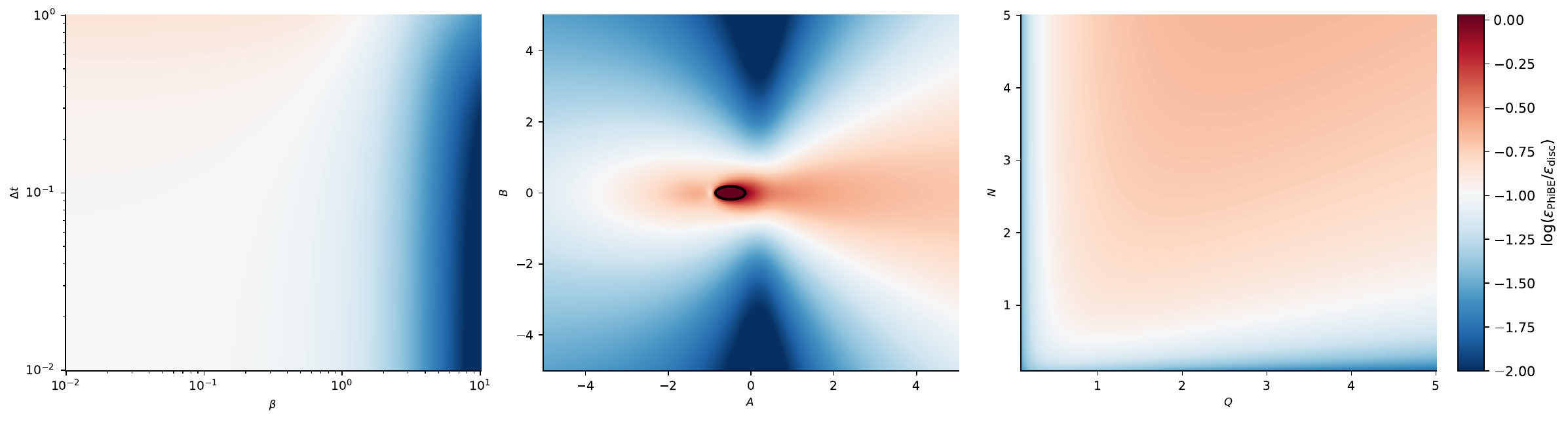}
  \caption{
  Parameter regimes for the comparison between LQR-MF-PhiBE and direct
  discretization. Left: logarithmic error ratio in the $(\beta,\Delta t)$
  plane. Center: logarithmic error ratio in the $(A,B)$ plane. Right:
  logarithmic error ratio in the $(Q,N)$ plane. The black line denotes the
  zero level set.
  }
  \label{fig:lqr-phibe-vs-discrete-heatmaps}
\end{figure}
\subsection{Example 2: Crowd aversion transport}
\label{subsec:numerics-crowd-aversion}

We next test the model-free MF-PhiBE actor-critic method on a nonlinear two-dimensional McKean-Vlasov control problem. The objective is to transport the population toward a prescribed target location while penalizing concentration in regions of high crowd density. This experiment is formulated directly at the level of the population law.

\subsubsection{Problem setup}
\label{subsec:crowd-problem-setup}

The controlled population law $\mu_t$ is generated by
\[
ds_t=m_\omega(s_t,\mu_t)\,dt+\sigma\,dW_t,
\qquad
m_\omega(s,\mu):=
\int_{\mathbb R^2}a\,\pi_\omega(da\mid s,\mu),
\]
where $s_t\in\mathbb R^2$, $a_t\in\mathbb R^2$, and $\sigma>0$ is fixed. The running reward is
\[
r(s,\mu,a)
=
-\left(
\frac{1}{2}\|a\|^2
+
\frac{\kappa}{2}\|s-s_{\rm tar}\|^2
+
\gamma (K_\eta * \mu)(s)
+
\frac{\rho}{2}\|s\|^2
\right).
\]
Here $s_{\rm tar}\in\mathbb R^2$ is the target, $\kappa>0$ controls target attraction, $\gamma>0$ controls crowd aversion, and $\rho\geq 0$ is a mild confinement parameter used for numerical stability. The crowd-density term is
\[
(K_\eta * \mu)(s)
=
\int_{\mathbb R^2}
\exp\left(-\frac{\|s-y\|^2}{2\eta^2}\right)\mu(dy),
\]
where $\eta>0$ is the kernel bandwidth. We consider the entropy-regularized reward introduced in \eqref{eq:reg_reward_avg_mkv}. The law-level objective is the discounted infinite-horizon value $\mathcal V^{\pi_\omega}$ in \eqref{eq:value_fn_mkv}.

\begin{remark}[Interpretation of the reward]\label{remark:interpretation_reward_crowd}
The four terms in the running cost correspond, respectively, to control effort, target attraction, crowd aversion, and mild confinement. Since the controlled equation is linear and the control acts directly as a velocity field, the learned policy is expected to transport the population toward $s_{\rm tar}$ while avoiding excessive concentration of mass, with the confinement term preventing particles from drifting too far from the target region.
\end{remark}

\subsubsection{Policy and value approximation}
\label{subsec:crowd-policy-critic}

Motivated by the crowd-density term in the reward, we use Gaussian kernel features of the population law. Let $c_1,\ldots,c_{M_\mu}\in\mathbb R^2$ be grid centers and let $h>0$ be a bandwidth. We define
\[
u_j(\mu)
=
\int_{\mathbb R^2}
\exp\left(-\frac{\|y-c_j\|^2}{2h^2}\right)\mu(dy),
\qquad
j=1,\ldots,M_\mu.
\]
We use a Gaussian policy with fixed exploration covariance,
\[
\pi_\omega(\cdot\mid s,\mu)
=
\mathcal N\left(m_\omega(s,\mu),\sigma_a^2 I_2\right),
\qquad
m_\omega(s,\mu)=W f(s,\mu),
\qquad
\omega\equiv W\in\mathbb R^{2\times d_f},
\]
where the feature vector $f(s,\mu)$ is given by
\[
f(s,\mu)
=
\Bigl(
1,\,
s,\,
s-s_{\rm tar},\,
\|s-s_{\rm tar}\|^2,\,
u_1(\mu),\ldots,u_{M_\mu}(\mu)
\Bigr)^\top .
\]
Thus, in this two-dimensional example, the number of actor features is $d_f=6+M_\mu$. For the critic, we use the cylindrical approximation
\[
\mathcal V_\theta(\mu)=\theta^\top \Phi(u(\mu)),
\qquad
u(\mu)=(u_1(\mu),\ldots,u_{M_\mu}(\mu)).
\]
The feature map $\Phi$ contains the constant, the linear terms in $u$, and all quadratic products:
\[
\Phi(u)
=
\Bigl(
1,\,
u_1,\ldots,u_{M_\mu},\,
\{u_i u_j\}_{1\leq i\leq j\leq M_\mu}
\Bigr)^\top .
\]
Thus, the critic has $1+M_\mu+M_\mu(M_\mu+1)/2$ features. The Lions derivatives $\partial_\mu \mathcal V_\theta(\mu)(x)$ and $D_x\partial_\mu \mathcal V_\theta(\mu)(x)$ are computed analytically from this cylindrical representation.

\subsubsection{Empirical MF-PhiBE critic and policy improvement}
\label{subsec:crowd-law-critic}

At policy iteration $k$, we simulate $L$ independent empirical population trajectories, each consisting of $M$ particles and $N_{\rm steps}$ time steps. For trajectory $\ell$ and time index $i$, the empirical law is
\[
\widehat\mu_\ell^i
=
\frac{1}{M}\sum_{j=1}^M \delta_{s_{\ell,j}^i}.
\]
Given $\widehat\mu_\ell^i$, we sample $a_{\ell,j}^i
\sim
\pi_{\omega_k}(\cdot\mid s_{\ell,j}^i,\widehat\mu_\ell^i),
\qquad j=1,\ldots,M,$ and generate the frozen-action transitions
\[
(s_{\ell,j}^i)'
=
s_{\ell,j}^i+a_{\ell,j}^i\Delta t
+\sigma\sqrt{\Delta t}\,Z_{\ell,i,j},
\qquad
Z_{\ell,i,j}\sim\mathcal N(0,I_2).
\]
Because the pointwise controlled drift in this example is
$b(s,\mu,a)=a$ and the diffusion is constant, this frozen-action
transition is exact and introduces no additional time-discretization
error. The data available at policy iteration $k$ are
\[
 \mathcal D_{\omega_k}
  =
  \left\{
  \bigl(
  s_{\ell,j}^i,\widehat\mu_\ell^i,
  a_{\ell,j}^i,r_{\ell,j}^i,(s_{\ell,j}^i)'
  \bigr)
  :
  i=1,\ldots,N_{\rm steps},\ \ell=1,\ldots,L,\ j=1,\ldots,M
  \right\}.
\]
The critic parameter $\theta$ is obtained from the Galerkin problem \eqref{eq:critic-linear-system-mkv} using the one-step coefficient estimators $\widehat b_{\ell,i,j}$ and $\widehat\Sigma_{\ell,i,j}$ in \eqref{eq:coef_est}. We use either discounted weights $w_{\ell,i}=e^{-\beta t_i}\Delta t/L$ or finite-horizon weights $w_{\ell,i}=\Delta t/L$.

Once the MF-PhiBE critic has been computed, we evaluate the sample infinitesimal advantages $\widehat q_{\ell,i,j}$ in \eqref{eq:advantage_func} and compute the empirical policy-gradient direction
\[
\widehat{\nabla J}(\omega_k)
=
\sum_{\ell=1}^{L}
\sum_{i=0}^{N_{\rm steps}-1}
w_{\ell,i}\,
\frac1M\sum_{j=1}^{M}
\left(
\widehat q_{\ell,i,j}
-
\overline q_{\ell,i}
\right)
\nabla_\omega
\log p_{\pi_{\omega_k}}
\left(
a_{\ell,j}^i
\mid
s_{\ell,j}^i,\widehat\mu_\ell^i
\right),
\]
where $\overline q_{\ell,i}
:=
\frac1M\sum_{j=1}^{M}\widehat q_{\ell,i,j},$ and with score term given by the explicit expression
\[
\nabla_W\log p_{\pi_{\omega_k}}(a\mid s,\mu)
=
\frac{1}{\sigma_a^2}
\left(a-m_{\omega_k}(s,\mu)\right)f(s,\mu)^\top.
\]

\subsubsection{Implementation and numerical results}

For the numerical simulations, we use
\[
\sigma=0.2,
\qquad
\beta=0.1,
\qquad
\kappa=1000,
\qquad
\gamma=10,
\qquad
\eta=0.8,
\qquad
\rho=0.1,
\qquad
\lambda=2.
\]
The initial law is concentrated near $x_0=(-2,0)^\top$, with standard deviation $0.02$, and the target is $x_{\rm tar}=(2,0)^\top$. We train with $\Delta t=0.1$, $T=4.5$, $L=50$ empirical trajectories per policy update, $N_{\rm steps}=45$ time-steps, $M=40$ particles per empirical law,  and $M_\mu=9$ law features, corresponding to a $3\times 3$ grid of centers $c_1,\ldots,c_9$. The policy uses fixed covariance $\sigma_a^2 I_2$, with $\sigma_a=0.25$.

\begin{figure}
    \centering
    \includegraphics[width=0.4\linewidth]{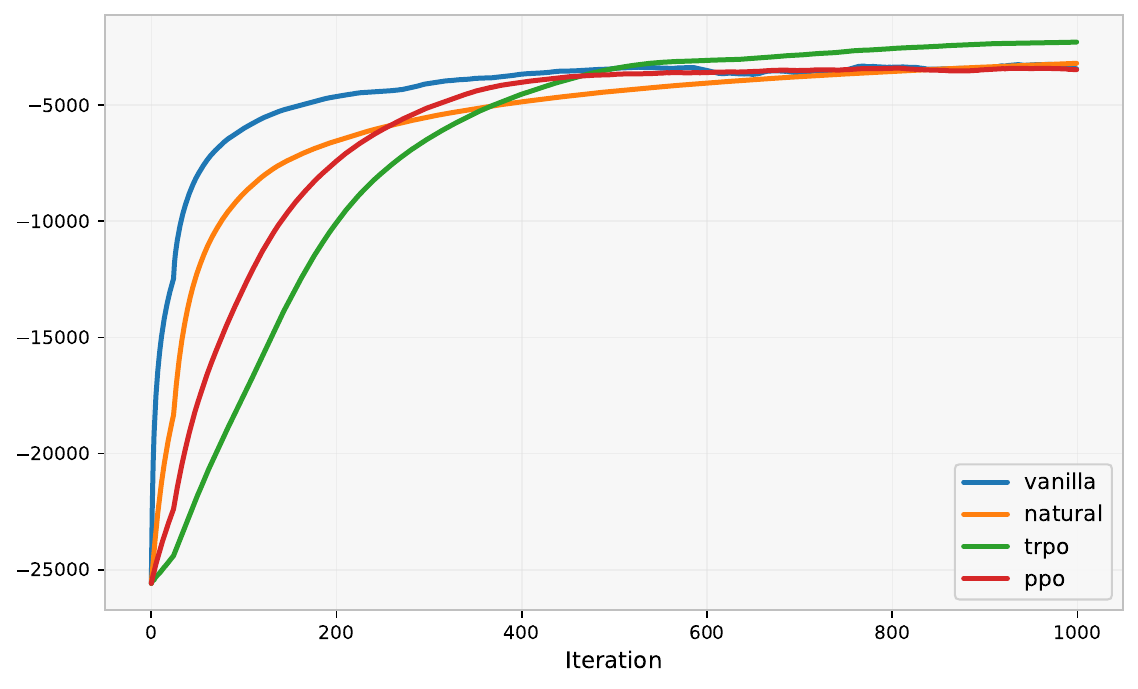}
    \includegraphics[width=0.4\linewidth]{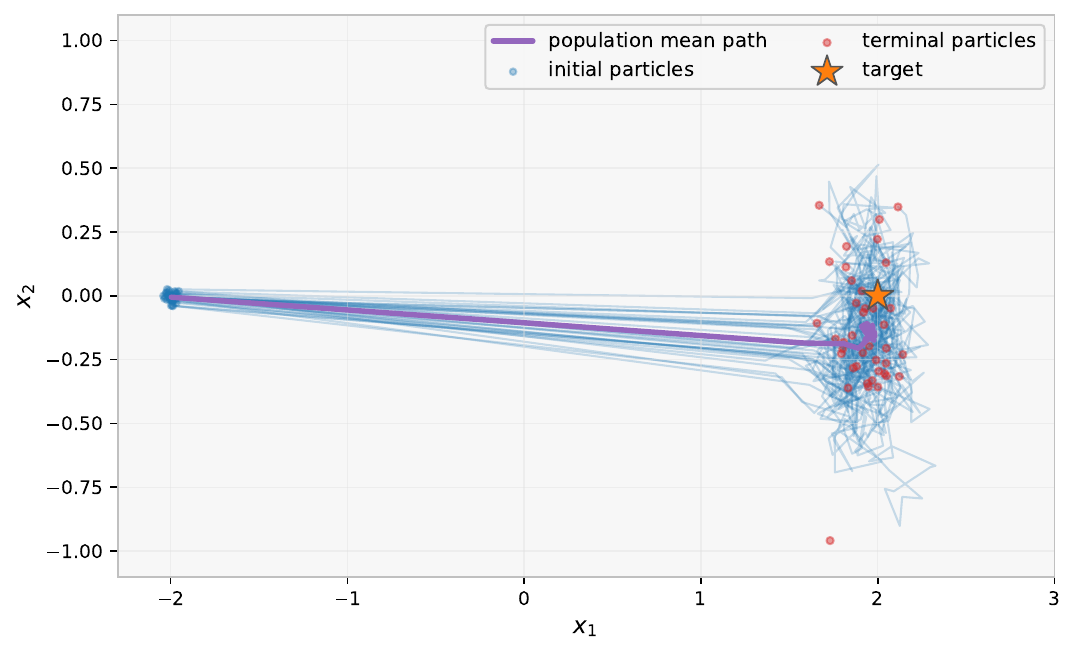}
    \caption{Crowd aversion experiment. Left: training curves for several actor updates. Right: simulated population trajectories under the learned policy.}
    \label{fig:crowd-gradient-comparison}
\end{figure}

\smallbreak
\Cref{fig:crowd-gradient-comparison} illustrates the behavior of the model-free MF-PhiBE actor-critic scheme in the crowd aversion problem. The left panel compares the training curves obtained with vanilla policy gradient, natural policy gradient, TRPO, and PPO. All methods improve the value during training, with different transient behavior depending on the update rule. The right panel shows representative trajectories under the learned policy. The population is transported from the initial region near $(-2,0)$ toward the target near $(2,0)$, while the empirical distribution spreads around the target due to the diffusion, the exploration covariance, and the crowd-aversion penalty, consistently with the interpretation of the reward given in \Cref{remark:interpretation_reward_crowd}.


\section{Acknowledgements}
E. Bayraktar is supported in part by NSF grant DMS-2602036 and in part by the Susan M. Smith Professorship. Y. Zhu and M. Hernandez are supported in part by NSF grant No. 2529107.     

\appendix

\section{Proof of the main results}
\subsection{Proofs for Section~\ref{sec:HJB_mkv}}
\label{app:HJB-proofs}

\begin{proof}[\textbf{Proof of \Cref{thm:MF_eval_infinite_mkv}}]
Let $V^\pi(s,\mu)$ be the value function of the coupled
representative-population problem in \cite{firstpaper}, associated with
the same fixed policy $\pi$. By \cite[Theorem 3.3]
{firstpaper}, under \Cref{ass:_mkv_regularity}, the coupled value
function satisfies $ V^\pi\in
  \mathcal C^{2,2}_{\mathrm{poly}}
  (\mathbb R^d\times\mathcal P_2(\mathbb R^d))$ (the space of functions with polynomial growth that are twice differentiable in $s$ and $\mu$; see \cite[Definition 3.1]{firstpaper}) and solves
\begin{equation}
\label{eq:coupled-HJB-used-mkv}
  (\mathscr L_{b,\Sigma}^{\pi}-\beta)V^\pi(s,\mu)
  +\ell_\lambda^\pi(s,\mu)=0,
\end{equation}
where $\mathscr L_{b,\Sigma}^{\pi}$ being the generator on
$\mathbb R^d\times\mathcal P_2(\mathbb R^d)$ defined as
\begin{equation}
  \label{eq:L-pi}
  \begin{aligned}
    (\mathscr{L}_{b,\Sigma}^{\pi}F)(s,\mu)
    &:=
    b^{\pi}(s,\mu)\cdot\nabla_s F(s,\mu)
    + \tfrac{1}{2}\,\Sigma^{\pi}(s,\mu) : \nabla^2_{ss} F(s,\mu)\\
    &\quad + \int_{\mathbb{R}^d}
      \Bigl[
        b^{\pi}(\xi,\mu)\cdot \partial_{\mu}F(s,\mu)(\xi)
        + \tfrac{1}{2}\,\Sigma^{\pi}(\xi,\mu)
          : D_\xi\partial_{\mu}F(s,\mu)(\xi)
      \Bigr]\,\mu(d\xi),
  \end{aligned}
\end{equation}
On the other hand, since the initial condition in
\eqref{eq:MKV_only_SDE} satisfies $s_0\sim\mu$, conditioning on
$s_0=s$, using Fubini's theorem, we have
\begin{align}\label{eq:mkv-value-as-integrated-coupled-value}
\begin{aligned}
  \mathcal V^\pi(\mu)
  &=
  \mathbb E\left[
    \int_0^\infty e^{-\beta t}
    \ell_\lambda^\pi(s_t,\mu_t)\,dt
    \,\bigg|\,\mu_0=\mu
  \right] =
  \int_{\mathbb R^d}
  \mathbb E\left[
    \int_0^\infty e^{-\beta t}
    \ell_\lambda^\pi(s_t,\mu_t)\,dt
    \,\bigg|\,s_0=s,\ \mu_0=\mu
  \right]\mu(ds)                                             \\
  &=
  \int_{\mathbb R^d}V^\pi(s,\mu)\,\mu(ds).
\end{aligned}
\end{align}
Since $V^\pi\in\mathcal C^{2,2}_{\mathrm{poly}}$, applying the Lions
derivative formula for maps of the form
$\mu\mapsto\int_{\mathbb R^d} V^\pi(s,\mu)\,\mu(ds)$
\cite[Section 5.6]{carmona2018probabilistic1} yields
\begin{equation}
\label{eq:mkv-value-lions-derivative}
  \partial_\mu\mathcal V^\pi(\mu)(x)
  =
  \nabla_sV^\pi(x,\mu)
  +
  \int_{\mathbb R^d}
  \partial_\mu V^\pi(s,\mu)(x)\,\mu(ds),
\end{equation}
and
\begin{equation}
\label{eq:mkv-value-lions-second-derivative}
  D_x\partial_\mu\mathcal V^\pi(\mu)(x)
  =
  \nabla^2_{ss}V^\pi(x,\mu)
  +
  \int_{\mathbb R^d}
  D_x\partial_\mu V^\pi(s,\mu)(x)\,\mu(ds).
\end{equation}
The polynomial bounds in the definition of
$\mathcal C^2_{\mathrm{poly}}(\mathcal P_2(\mathbb R^d))$ follow
directly from the corresponding bounds for $V^\pi$. Hence
$\mathcal V^\pi\in
\mathcal C^2_{\mathrm{poly}}(\mathcal P_2(\mathbb R^d))$. Using \eqref{eq:mkv-value-lions-derivative} and
\eqref{eq:mkv-value-lions-second-derivative}, we compute
\[
\begin{aligned}
  (\mathcal L_{b,\Sigma}^{\pi}\mathcal V^\pi)(\mu)
  &=
  \int_{\mathbb R^d}
  \left[
    b^\pi(x,\mu)\cdot\nabla_sV^\pi(x,\mu)
    +
    \frac12\Sigma^\pi(x,\mu):\nabla^2_{ss}V^\pi(x,\mu)
  \right]\mu(dx)                                            \\
  &\quad+
  \int_{\mathbb R^d}\int_{\mathbb R^d}
  \left[
    b^\pi(x,\mu)\cdot\partial_\mu V^\pi(s,\mu)(x)
    +
    \frac12\Sigma^\pi(x,\mu):
    D_x\partial_\mu V^\pi(s,\mu)(x)
  \right]\mu(dx)\mu(ds)                                     \\
  &=
  \int_{\mathbb R^d}
  (\mathscr L_{b,\Sigma}^{\pi}V^\pi)(s,\mu)\,\mu(ds).
\end{aligned}
\]
Integrating \eqref{eq:coupled-HJB-used-mkv} with respect to $\mu(ds)$
and using \eqref{eq:mkv-value-as-integrated-coupled-value} gives
\[
  (\mathcal L_{b,\Sigma}^{\pi}\mathcal V^\pi)(\mu)
  -
  \beta\mathcal V^\pi(\mu)
  +
  r_\lambda^\pi(\mu)
  =
  0.
\]
Thus $\mathcal V^\pi$ solves \eqref{eq:HJB-eval-mkv}.

We now prove uniqueness. Let
$U\in\mathcal C^2_{\mathrm{poly}}(\mathcal P_2(\mathbb R^d))$ be
another classical solution of \eqref{eq:HJB-eval-mkv}. Applying the
It\^o-Lions formula to $e^{-\beta t}U(\mu_t)$ gives
\[
  \frac{d}{dt}
  \left(e^{-\beta t}U(\mu_t)\right)
  =
  e^{-\beta t}
  \left[
    (\mathcal L_{b,\Sigma}^{\pi}U)(\mu_t)
    -
    \beta U(\mu_t)
  \right]
  =
  -e^{-\beta t}r_\lambda^\pi(\mu_t).
\]
Therefore, for every $T>0$,
\[
  U(\mu)
  =
  e^{-\beta T}U(\mu_T)
  +
  \int_0^T e^{-\beta t}r_\lambda^\pi(\mu_t)\,dt.
\]
By the polynomial growth of $U$, there exists $C_U>0$ such that
\[
  |U(\nu)|\le C_U\bigl(1+m_2(\nu)\bigr),
  \qquad \nu\in\mathcal P_2(\mathbb R^d).
\]
Moreover, by \cite[Lemma A.1 with $p=2$]{firstpaper}, there exists
$C>0$ such that
\[
  m_2(\mu_T)\le C e^{\beta_0(2)T}\bigl(1+m_2(\mu)\bigr),
  \qquad T\ge0.
\]
Since $\beta>\bar\beta_\pi$ and, by \cite[Remark 3.4]{firstpaper},
$\bar\beta_\pi\ge \beta_0(2)$, we obtain
\[
  e^{-\beta T}|U(\mu_T)|
  \le
  C e^{-\beta T}
  +
  C e^{-(\beta-\beta_0(2))T}\bigl(1+m_2(\mu)\bigr)
  \longrightarrow 0,
  \qquad T\to\infty.
\]
Hence
\[
  U(\mu)
  =
  \int_0^\infty e^{-\beta t}r_\lambda^\pi(\mu_t)\,dt
  =
  \mathcal V^\pi(\mu).
\]
It remains to prove \eqref{eq:value-derivative-bound-mkv}. In
\cite[Step 3 of the proof of Theorem 3.3]{firstpaper},
the finite-horizon derivatives of the coupled value functions are
represented through the variational flows. The estimates \cite[(71)-(72)]{firstpaper} therein imply,
after integration against $e^{-\beta t}$, that
\[
\begin{aligned}
  |\nabla_sV^\pi(s,\nu)|
  +
  |\partial_\mu V^\pi(s,\nu)(x)|
  &\le
  \frac{C_\pi}{\beta-\bar\beta_\pi}
  \bigl(1+|s|+|x|+m_2(\nu)\bigr),                                  \\
  \|\nabla^2_{ss}V^\pi(s,\nu)\|
  +
  \|D_x\partial_\mu V^\pi(s,\nu)(x)\|
  &\le
  \frac{C_\pi}{\beta-\bar\beta_\pi}
  \bigl(1+m_2(\nu)\bigr).
\end{aligned}
\]
The passage from the finite-horizon bounds to the infinite-horizon
bounds follows from the locally uniform convergence of the derivatives
proved in \cite[Step 4 of the proof of Theorem 3.3]
{firstpaper}.

Combining these estimates with
\eqref{eq:mkv-value-lions-derivative} gives
\[
\begin{aligned}
  |\partial_\mu\mathcal V^\pi(\nu)(x)|
  &\le
  |\nabla_sV^\pi(x,\nu)|
  +
  \int_{\mathbb R^d}
  |\partial_\mu V^\pi(s,\nu)(x)|\,\nu(ds)\le
  \frac{C_\pi}{\beta-\bar\beta_\pi}
  \bigl(1+|x|^2+m_2(\nu)\bigr).
\end{aligned}
\]
Similarly, using \eqref{eq:mkv-value-lions-second-derivative},
\[
\begin{aligned}
  \|D_x\partial_\mu\mathcal V^\pi(\nu)(x)\|
  &\le
  \|\nabla^2_{ss}V^\pi(x,\nu)\|
  +
  \int_{\mathbb R^d}
  \|D_x\partial_\mu V^\pi(s,\nu)(x)\|\,\nu(ds)     \le
  \frac{C_\pi}{\beta-\bar\beta_\pi}
  \bigl(1+|x|^2+m_2(\nu)\bigr).
\end{aligned}
\]
This proves \eqref{eq:value-derivative-bound-mkv}.
\end{proof}

\begin{proof}[\textbf{Proof of \Cref{cor:policy-gradient-mkv}}]
Let $V^{\pi^\varepsilon}$ denote the value function of the coupled
state-law problem studied in \cite{firstpaper}. Set
\[
  J(\varepsilon)
  :=
  \int_{\mathbb R^d}
  V^{\pi^\varepsilon}(s,\mu)\,\mu(ds).
\]
By \eqref{eq:mkv-value-as-integrated-coupled-value},
$J(\varepsilon)=\mathcal V^{\pi^\varepsilon}(\mu)=\mathcal J(\varepsilon)$.
It is therefore enough to compute the derivative of $J$ at
$\varepsilon=0$.

We apply \cite[Theorem~4.6]{firstpaper} with initial population law
$\mu$ and initial representative law also equal to $\mu_0=\mu$. Denote by
$\widehat\rho_\mu^\pi$ the corresponding discounted occupancy measure of
the coupled state-law process. Then $J$ is differentiable at
$\varepsilon=0$, and
\[
\begin{aligned}
\frac{d}{d\varepsilon}J(\varepsilon)\bigg|_{\varepsilon=0}
&=\frac{1}{\beta}
\mathbb E_{(s,\eta)\sim\widehat\rho^\pi}\!\left[
\int_{\mathcal A}
q_{\mathrm{rep}}^\pi(s,\eta,a)\,
\varphi(da\mid s,\eta)
+
\int_{\mathbb R^d}\!\int_{\mathcal A}
q_{\mathrm{pop}}^\pi(s,\eta,\xi,a)\,
\varphi(da\mid \xi,\eta)\,\eta(d\xi)
\right],
\end{aligned}
\]
where, according to \cite[Definition~4.2]{firstpaper},
\[
\begin{aligned}
q_{\mathrm{rep}}^\pi(s,\eta,a)
&=
r(s,\eta,a)-\lambda\log p_\pi(a\mid s,\eta)
+b(s,\eta,a)\cdot\nabla_s V^\pi(s,\eta)
\\
&\qquad\qquad\qquad+\frac12\,\Sigma(s,\eta,a):D^2_{ss}V^\pi(s,\eta)-\beta V^\pi(s,\eta),
\\
q_{\mathrm{pop}}^\pi(s,\eta,\xi,a)
&=
b(\xi,\eta,a)\cdot\partial_\mu V^\pi(s,\eta)(\xi)
+\frac12\,\Sigma(\xi,\eta,a):D_\xi\partial_\mu V^\pi(s,\eta)(\xi).
\end{aligned}
\]

Since the initial representative law is $\mu$, the representative
component in the coupled state-law system has law $\mu_t^\pi$ at time
$t$. Hence, for every measurable $F$ for which the following quantities
are well defined,
\begin{align}\label{eq:rho-transfer-mkv}
\mathbb E_{(s,\eta)\sim\widehat\rho_\mu^\pi}[F(\eta,s)]
=
\beta\int_0^\infty e^{-\beta t}
\int_{\mathbb R^d}
F(\mu_t^\pi,\xi)\,\mu_t^\pi(d\xi)\,dt
=
\mathbb E_{(\eta,\xi)\sim\rho_\mu^\pi}[F(\eta,\xi)].
\end{align}
Applying this identity with
\[
F(\eta,\xi)
:=
\int_{\mathcal A}
\left[
q_{\mathrm{rep}}^\pi(\xi,\eta,a)
+
\int_{\mathbb R^d}
q_{\mathrm{pop}}^\pi(s,\eta,\xi,a)\,\eta(ds)
\right]
\varphi(da\mid \xi,\eta),
\]
we obtain
\begin{align}\label{eq:derivative_j_almost_finished}
    \begin{aligned}
\frac{d}{d\varepsilon}\mathcal J(\varepsilon)\bigg|_{\varepsilon=0}
&=\frac{1}{\beta}
\mathbb E_{(\eta,\xi)\sim\rho_\mu^\pi}\!\left[
\int_{\mathcal A}
\left[
q_{\mathrm{rep}}^\pi(\xi,\eta,a)
+
\int_{\mathbb R^d}
q_{\mathrm{pop}}^\pi(s,\eta,\xi,a)\,\eta(ds)
\right]
\varphi(da\mid \xi,\eta)
\right].
\end{aligned}
\end{align}
It remains to identify the term in brackets. By \eqref{eq:mkv-value-lions-derivative} and
\eqref{eq:mkv-value-lions-second-derivative},
\begin{align}\label{eq:realtion_q_paper_1_neq_q}
&q_{\mathrm{rep}}^\pi(\xi,\eta,a)
+\int_{\mathbb R^d}q_{\mathrm{pop}}^\pi(s,\eta,\xi,a)\,\eta(ds)=
q^\pi(\eta,\xi,a)
+\beta\left(
\mathcal V^\pi(\eta)-V^\pi(\xi,\eta)
\right).
\end{align}
The additional term is independent of $a$. Hence, by
$\int_{\mathcal A}\varphi(da\mid\xi,\eta)=0$,
\[
\int_{\mathcal A}
\beta\bigl(\mathcal V^\pi(\eta)-V^\pi(\xi,\eta)\bigr)
\,\varphi(da\mid\xi,\eta)=0.
\]
Substituting \eqref{eq:realtion_q_paper_1_neq_q} into
\eqref{eq:derivative_j_almost_finished} therefore gives
\eqref{eq:gateaux-pg-mkv}.
\end{proof}
\begin{proof}[\textbf{Proof of \Cref{thm:parametric-policy-gradient}}]
Fix $\omega\in\mathbb R^p$ and set $\pi:=\pi_\omega$. Define $J(\omega):=\int_{\mathbb R^d}V^{\pi_\omega}(s,\mu)\,\mu(ds).$ As in the proof of \Cref{cor:policy-gradient-mkv}, one has
$J(\omega)=\mathcal V^{\pi_\omega}(\mu)=\mathcal J(\omega)$, so it is
enough to compute the gradient of $J$.

By the parametric policy gradient theorem for $J$ in \cite[Theorem 4.9]{firstpaper},
$J$ is differentiable at $\omega$ and
\[
\begin{aligned}
\nabla_\omega J(\omega)
&=
\frac{1}{\beta}\mathbb E_{(s,\eta)\sim\widehat\rho^\pi}\!\left[
\int_{\mathcal A}
q_{\mathrm{rep}}^\pi(s,\eta,a)\,
\nabla_\omega\log p_\pi(a\mid s,\eta)\,
\pi(da\mid s,\eta)
\right]\\
&\qquad+
\frac{1}{\beta}\mathbb E_{(s,\eta)\sim\widehat\rho^\pi}\!\left[
\int_{\mathbb R^d}\!\int_{\mathcal A}
q_{\mathrm{pop}}^\pi(s,\eta,\xi,a)\,
\nabla_\omega\log p_\pi(a\mid \xi,\eta)\,
\pi(da\mid \xi,\eta)\,
\eta(d\xi)
\right].
\end{aligned}
\]
Using \eqref{eq:rho-transfer-mkv} and integrating the population term
first with respect to the representative state yields
\[
\nabla_\omega J(\omega)
=\frac1\beta
\mathbb E_{(\eta,\xi)\sim\rho_\mu^\pi}
\left[
\int_{\mathcal A}
\left(
q_{\mathrm{rep}}^\pi(\xi,\eta,a)
+\int_{\mathbb R^d}q_{\mathrm{pop}}^\pi(s,\eta,\xi,a)\,\eta(ds)
\right)
\nabla_\omega\log p_\pi(a\mid\xi,\eta)
\,\pi(da\mid\xi,\eta)
\right].
\]
By \eqref{eq:realtion_q_paper_1_neq_q}, the term in parentheses equals
$q^\pi(\eta,\xi,a)$ plus an action-independent baseline. The baseline
vanishes because Assumption~\ref{ass:A2} justifies differentiation of
$\int p_{\pi_\omega}\,d\nu=1$ and therefore
\[
\int_{\mathcal A}
\nabla_\omega\log p_\pi(a\mid\xi,\eta)
\,\pi(da\mid\xi,\eta)=0.
\]
This proves \eqref{eq:parametric-pg}; its conditional-expectation form
follows immediately.
\end{proof}

\subsection{Proofs for Section~\ref{sec:MF-phibe}}
\label{app:error-proofs}

\begin{definition}
   We use the following pointwise notation. If $f=f(s,\mu)$ is
$\mathcal C^{2,2}$ in $(s,\mu)$, then
\begin{equation}
\label{eq:evaluation-seminorm}
\begin{aligned}
\tnorm{f(s,\mu)}_{\xi}
:=
&\,\|f(s,\mu)\|
+\|\nabla_s f(s,\mu)\|
+\|D_{ss}^2 f(s,\mu)\|
+
|\partial_\mu f(s,\mu)(\xi)|
+
\|D_\xi\partial_\mu f(s,\mu)(\xi)\|.
\end{aligned}
\end{equation}
The derivatives in \eqref{eq:evaluation-seminorm} are always taken with
respect to $(s,\mu)$. Any additional variables, such as $a$ or $z$, are
treated as fixed parameters.
\end{definition}

\begin{lemma}[Local expansion for the frozen-action estimators]
\label{lem:local-dynkin-expansion-frozen-action-mkv}
Assume \Cref{ass:MF-PhiBE-structure}. Fix $a\in\mathcal A$ and
$(s,\mu)\in\R^d\times\mathcal P_2(\R^d)$. Let
$(\mu_t)_{t\in[0,\Delta t]}$ be the law flow associated with the
frozen-action version of \eqref{eq:MKV_only_SDE}, initialized at $\mu$,
that is, with $b^\pi$ and $\sigma^\pi$ replaced by
$b(\cdot,\cdot,a)$ and $\sigma(\cdot,\cdot,a)$. Given this flow, let
$(s_t)_{t\in[0,\Delta t]}$ solve
\[
ds_t
=
b(s_t,\mu_t,a)\,dt
+
\sigma(s_t,\mu_t,a)\,dB_t,
\qquad
s_0=s.
\]
For $F\in\mathcal C^{2,2}_{\mathrm{poly}}
(\R^d\times\mathcal P_2(\R^d))$, define
\begin{equation}
\label{eq:frozen-action-generator-mkv}
\begin{aligned}
(\mathscr L_{b,\Sigma}^aF)(x,\nu)
&:=
b(x,\nu,a)\cdot\nabla_xF(x,\nu)
+\frac12\,\Sigma(x,\nu,a):D_{xx}^2F(x,\nu)
\\
&\quad+
\int_{\R^d}
\Bigl[
b(\eta,\nu,a)\cdot\partial_\nu F(x,\nu)(\eta)
+
\frac12\,\Sigma(\eta,\nu,a):D_\eta\partial_\nu F(x,\nu)(\eta)
\Bigr]\nu(d\eta),
\end{aligned}
\end{equation}
where $\Sigma(x,\nu,a):=\sigma(x,\nu,a)\sigma(x,\nu,a)^\top$.
Assume that $\mathscr L_{b,\Sigma}^aF\in
\mathcal C^{2,2}_{\mathrm{poly}}(\R^d\times\mathcal P_2(\R^d))$ and
\[
\sup_{\tau\in[0,\Delta t]}
\mathbb E\!\left[
|(\mathscr L_{b,\Sigma}^a)^2F(s_\tau,\mu_\tau)|
\right]<\infty .
\]
Then
\begin{equation}
\label{eq:local-expansion-general-mkv}
\frac{\mathbb E[F(s_{\Delta t},\mu_{\Delta t})]-F(s,\mu)}{\Delta t}
=
(\mathscr L_{b,\Sigma}^aF)(s,\mu)
+
R_F^{a,\Delta t}(s,\mu),
\end{equation}
where
\begin{equation}
\label{eq:local-expansion-remainder-mkv}
|R_F^{a,\Delta t}(s,\mu)|
\le
\frac{\Delta t}{2}
\sup_{\tau\in[0,\Delta t]}
\mathbb E\!\left[
|(\mathscr L_{b,\Sigma}^a)^2F(s_\tau,\mu_\tau)|
\right].
\end{equation}
Moreover, for $F_k(x,\nu):=x_k$ and
$G_{ij}^s(x,\nu):=(x_i-s_i)(x_j-s_j)$, one has
\begin{equation}
\label{eq:Fk-identities-mkv}
\frac{\mathbb E[F_k(s_{\Delta t},\mu_{\Delta t})]-F_k(s,\mu)}{\Delta t}
=
\hat b_k(s,\mu,a),
\qquad
(\mathscr L_{b,\Sigma}^aF_k)(s,\mu)=b_k(s,\mu,a),
\end{equation}
and
\begin{equation}
\label{eq:Gij-identities-mkv}
\frac{\mathbb E[G_{ij}^s(s_{\Delta t},\mu_{\Delta t})]-G_{ij}^s(s,\mu)}
{\Delta t}
=
\hat\Sigma_{ij}(s,\mu,a),
\qquad
(\mathscr L_{b,\Sigma}^aG_{ij}^s)(s,\mu)=\Sigma_{ij}(s,\mu,a).
\end{equation}
In particular, whenever the remainders in
\eqref{eq:local-expansion-general-mkv} are finite for $F_k$ and
$G_{ij}^s$,
\[
\hat  b_k(s,\mu,a)-b_k(s,\mu,a)
=
R_{F_k}^{a,\Delta t}(s,\mu), 
\qquad
\hat\Sigma_{ij}(s,\mu,a)-\Sigma_{ij}(s,\mu,a)
=
R_{G_{ij}^s}^{a,\Delta t}(s,\mu).
\]
\end{lemma}\begin{proof}
Since $(\mu_t)_{t\in[0,\Delta t]}$ is deterministic, the
It\^o-Lions formula applied to $F(s_t,\mu_t)$ gives
\[
F(s_t,\mu_t)
=
F(s,\mu)
+
\int_0^t
(\mathscr L_{b,\Sigma}^aF)(s_\tau,\mu_\tau)\,d\tau
+
M_t^F,
\]
with
\[
M_t^F
=
\int_0^t
\nabla_xF(s_\tau,\mu_\tau)
\sigma(s_\tau,\mu_\tau,a)\,dB_\tau .
\]
The polynomial growth of $\nabla_xF$, the boundedness of the coefficients
in \Cref{ass:MF-PhiBE-structure}, and the moment estimates for the
frozen-action dynamics imply that $M^F$ is a square-integrable
martingale. Hence
\[
\mathbb E[F(s_t,\mu_t)]-F(s,\mu)
=
\int_0^t
\mathbb E[
(\mathscr L_{b,\Sigma}^aF)(s_\tau,\mu_\tau)
]\,d\tau .
\]
Applying the same argument to $\mathscr L_{b,\Sigma}^aF$ yields, for
$\tau\in[0,\Delta t]$,
\[
\mathbb E[
(\mathscr L_{b,\Sigma}^aF)(s_\tau,\mu_\tau)
]
=
(\mathscr L_{b,\Sigma}^aF)(s,\mu)
+
\int_0^\tau
\mathbb E[
(\mathscr L_{b,\Sigma}^a)^2F(s_u,\mu_u)
]\,du .
\]
Combining the two identities gives
\[
\mathbb E[F(s_t,\mu_t)]-F(s,\mu)
=
t(\mathscr L_{b,\Sigma}^aF)(s,\mu)
+
\int_0^t\int_0^\tau
\mathbb E[
(\mathscr L_{b,\Sigma}^a)^2F(s_u,\mu_u)
]\,du\,d\tau .
\]
Taking $t=\Delta t$ and dividing by $\Delta t$ proves
\eqref{eq:local-expansion-general-mkv}, with
\[
R_F^{a,\Delta t}(s,\mu)
:=
\frac1{\Delta t}
\int_0^{\Delta t}\int_0^\tau
\mathbb E[
(\mathscr L_{b,\Sigma}^a)^2F(s_u,\mu_u)
]\,du\,d\tau .
\]
The estimate \eqref{eq:local-expansion-remainder-mkv} follows from
\[
\frac1{\Delta t}\int_0^{\Delta t}\int_0^\tau du\,d\tau
=
\frac{\Delta t}{2}.
\]

For $F_k(x,\nu):=x_k$,
\[
\mathbb E[F_k(s_{\Delta t},\mu_{\Delta t})]-F_k(s,\mu)
=
\mathbb E[(s_{\Delta t})_k-s_k].
\]
Thus the first identity in \eqref{eq:Fk-identities-mkv} follows from
the definition of $\hat b_k$. Since $\nabla_xF_k=e_k$,
$D_{xx}^2F_k=0$, and $\partial_\nu F_k=0$, we also have
\[
(\mathscr L_{b,\Sigma}^aF_k)(s,\mu)
=
b(s,\mu,a)\cdot e_k
=
b_k(s,\mu,a).
\]

For $G_{ij}^s(x,\nu):=(x_i-s_i)(x_j-s_j)$,
$G_{ij}^s(s,\mu)=0$ and
\[
\mathbb E[G_{ij}^s(s_{\Delta t},\mu_{\Delta t})]
=
\mathbb E[
((s_{\Delta t})_i-s_i)((s_{\Delta t})_j-s_j)
].
\]
Hence the first identity in \eqref{eq:Gij-identities-mkv} follows from
the definition of $\hat\Sigma_{ij}$. Moreover,
\[
\nabla_xG_{ij}^s(x,\nu)
=
e_i(x_j-s_j)+e_j(x_i-s_i),
\qquad
D_{xx}^2G_{ij}^s(x,\nu)
=
e_ie_j^\top+e_je_i^\top,
\qquad
\partial_\nu G_{ij}^s=0.
\]
Therefore
\[
(\mathscr L_{b,\Sigma}^aG_{ij}^s)(x,\nu)
=
b_i(x,\nu,a)(x_j-s_j)
+
b_j(x,\nu,a)(x_i-s_i)
+
\Sigma_{ij}(x,\nu,a).
\]
Evaluating at $(x,\nu)=(s,\mu)$ gives the second identity in
\eqref{eq:Gij-identities-mkv}.

The final identities follow by applying
\eqref{eq:local-expansion-general-mkv} to $F_k$ and $G_{ij}^s$, and
then using \eqref{eq:Fk-identities-mkv} and
\eqref{eq:Gij-identities-mkv}.
\end{proof}

\begin{lemma}[Differentiated remainder bounds]
\label{lem:differentiated-remainder-bounds-mkv}
Assume \Cref{ass:MF-PhiBE-structure}. Fix $a\in\mathcal A$ and
$\Delta t\in(0,1]$. Let $(s_t,\mu_t)_{t\in[0,\Delta t]}$ be defined as
in \Cref{lem:local-dynkin-expansion-frozen-action-mkv}. For a smooth
function $H=H(x,\nu)$, define
\begin{equation}
\label{eq:RFk-definition-mkv}
R_H^{a,\Delta t}(s,\mu)
:=
\frac1{\Delta t}
\int_0^{\Delta t}\int_0^\tau
\mathbb E\!\left[
(\mathscr L_{b,\Sigma}^a)^2H(s_u,\mu_u)
\right]du\,d\tau .
\end{equation}
Then, for $F_k(x,\nu)=x_k$ and
$G_{ij}^s(x,\nu)=(x_i-s_i)(x_j-s_j)$, there exists $C>0$, depending
only on the constants in \Cref{ass:MF-PhiBE-structure} and on $d$, such
that, for every $(s,\mu,\xi)\in\R^d\times\mathcal P_2(\R^d)\times\R^d$,
\begin{equation}
\label{eq:RFk-differentiated-bound-mkv}
\tnorm{R_{F_k}^{a,\Delta t}(s,\mu)}_{\xi}
\le
C\Delta t,
\qquad
\tnorm{R_{G_{ij}^s}^{a,\Delta t}(s,\mu)}_{\xi}
\le
C\Delta t .
\end{equation}
Moreover, for every fixed $\Delta t\in(0,1]$ and every $R>0$, there
exists $C_{R,\Delta t}>0$, independent of $a$, such that every component
entering the seminorm in \eqref{eq:evaluation-seminorm} for
$R_{F_k}^{a,\Delta t}$ and $R_{G_{ij}^s}^{a,\Delta t}$ is locally
Lipschitz on
\[
|s|+\sqrt{m_2(\mu)}+|\xi|\le R
\]
with Lipschitz constant at most $C_{R,\Delta t}$.

\end{lemma}
\begin{proof}
\emph{Step 1: bounds for the short-time operator $P_u^a$.}
For $u\in[0,1]$, define
\[
P_u^a\Phi(s,\mu):=\mathbb E[\Phi(s_u,\mu_u)],
\]
where $(s_u,\mu_u)$ is the frozen-action dynamics introduced in \Cref{lem:local-dynkin-expansion-frozen-action-mkv}. We first prove that
\begin{equation}
\label{eq:semigroup-uniform-bound-mkv}
\tnorm{P_u^a\Phi(s,\mu)}_{\xi}\le C C_\Phi,
\qquad u\in[0,1],
\end{equation}
whenever
\[
\tnorm{\Phi(x,\nu)}_{\eta}\le C_\Phi
\qquad
\text{for all }(x,\nu,\eta)\in\R^d\times\mathcal P_2(\R^d)\times\R^d.
\]
By \cite[Lemma A.4, estimate (52)]{firstpaper}, applied to the frozen coefficients $b(\cdot,\cdot,a)$ and $\sigma(\cdot,\cdot,a)$, the derivatives of the frozen-action flow with respect to the initial state and the initial law satisfy
\[
\sup_{u\in[0,1]}\mathbb E\!\left[
\|\nabla_s s_u\|^2+\|D_{ss}^2s_u\|^2+\|\partial_\mu s_u(\xi)\|^2+\|D_\xi\partial_\mu s_u(\xi)\|^2
\right]\le C,
\]
where $C$ is independent of $a$, $u$, $s$, $\mu$, $\xi$, and $\Delta t$. Therefore, differentiating $P_u^a\Phi(s,\mu)=\mathbb E[\Phi(s_u,\mu_u)]$ and using the chain rule gives
\[
\|\nabla_s P_u^a\Phi(s,\mu)\|+\|D_{ss}^2P_u^a\Phi(s,\mu)\|\le C C_\Phi.
\]
The same argument, using the Lions chain rule and the estimates for $\partial_\mu s_u(\xi)$ and $D_\xi\partial_\mu s_u(\xi)$ in \cite[Lemma A.4, estimate (52)]{firstpaper}, gives
\[
|\partial_\mu P_u^a\Phi(s,\mu)(\xi)|+\|D_\xi\partial_\mu P_u^a\Phi(s,\mu)(\xi)\|\le C C_\Phi.
\]
Together with $|P_u^a\Phi(s,\mu)|\le C_\Phi$, this proves \eqref{eq:semigroup-uniform-bound-mkv}.
We also use the following linear-growth variant of \eqref{eq:semigroup-uniform-bound-mkv}. If
\[
\tnorm{\Phi_z(x,\nu)}_{\eta}\le C_\Phi(1+|x-z|)
\qquad
\text{for all }(x,\nu,\eta,z),
\]
then
\begin{equation}
\label{eq:semigroup-linear-growth-bound-mkv}
\tnorm{P_u^a\Phi_z(s,\mu)}_{\xi}
\le
C C_\Phi(1+|s-z|+u^{1/2}),
\qquad u\in[0,1].
\end{equation}
Indeed, the proof of \eqref{eq:semigroup-uniform-bound-mkv} gives the same estimate with $1+\mathbb E[|s_u-z|]$ on the right-hand side. Moreover,
\[
\mathbb E[|s_u-z|]\le |s-z|+\mathbb E[|s_u-s|]\le |s-z|+Cu^{1/2},
\]
where the last estimate follows from the boundedness of $b$ and $\sigma$ and the Burkholder-Davis-Gundy inequality. This proves \eqref{eq:semigroup-linear-growth-bound-mkv}.
\emph{Step 2: the remainder for $F_k$.}
Let $F_k(x,\nu)=x_k$ and set
\[
  \Phi_k^a(x,\nu):=(\mathscr L_{b,\Sigma}^a)^2F_k(x,\nu).
\]
Since $F_k$ is linear and the coefficients have bounded derivatives up
to order four, uniformly in $a$, applying $\mathscr L_{b,\Sigma}^a$
twice gives
\[
  \tnorm{\Phi_k^a(x,\nu)}_{\xi}\le C.
\]
By Step 1,
\[
  \tnorm{P_u^a\Phi_k^a(s,\mu)}_{\xi}\le C,
  \qquad u\in[0,1].
\]
Using \eqref{eq:RFk-definition-mkv},
\[
R_{F_k}^{a,\Delta t}(s,\mu)
=
\frac1{\Delta t}
\int_0^{\Delta t}\int_0^\tau
P_u^a\Phi_k^a(s,\mu)\,du\,d\tau .
\]
Therefore,
\[
\tnorm{R_{F_k}^{a,\Delta t}(s,\mu)}_{\xi}
\le
\frac1{\Delta t}
\int_0^{\Delta t}\int_0^\tau C\,du\,d\tau
\le
C\Delta t .
\]

\emph{Step 3: the structure of the remainder for $G_{ij}^s$.}
Fix first $z\in\R^d$ and define
\[
  G_{ij}^z(x,\nu):=(x_i-z_i)(x_j-z_j).
\]
A direct computation gives
\[
(\mathscr L_{b,\Sigma}^aG_{ij}^z)(x,\nu)
=
b_i(x,\nu,a)(x_j-z_j)
+
b_j(x,\nu,a)(x_i-z_i)
+
\Sigma_{ij}(x,\nu,a).
\]
This formula is the key point. The only terms that can grow in $x$ or
$z$ are the factors $x_j-z_j$ and $x_i-z_i$. All coefficients and all
their derivatives needed below are bounded by
\Cref{ass:MF-PhiBE-structure}.

Set
\[
  H_{ij}^{a,z}:=\mathscr L_{b,\Sigma}^aG_{ij}^z,
  \qquad
  \Phi_{ij}^{a,z}:=\mathscr L_{b,\Sigma}^aH_{ij}^{a,z}.
\]
From the displayed formula for $H_{ij}^{a,z}$, and since the operator
$\mathscr L_{b,\Sigma}^a$ does not depend on $z$, differentiating with
respect to $z$ gives, for every multi-index $\ell$ with $|\ell|\le2$,
\[
  \partial_z^\ell\Phi_{ij}^{a,z}
  =
  \mathscr L_{b,\Sigma}^a(\partial_z^\ell H_{ij}^{a,z}).
\]
Now $\partial_z^\ell H_{ij}^{a,z}$ is a sum of bounded coefficients
multiplied either by $1$ or by one of the factors $x_r-z_r$. Therefore,
using again the boundedness of the coefficients and of their derivatives,
\[
  \tnorm{\partial_z^\ell\Phi_{ij}^{a,z}(x,\nu)}_{\xi}
  \le
  C(1+|x-z|),
  \qquad |\ell|\le2.
\]
This is the desired bound: it follows from the explicit expression of
$\mathscr L_{b,\Sigma}^aG_{ij}^z$ and from the fact that applying
$\mathscr L_{b,\Sigma}^a$ adds only bounded coefficients and bounded
derivatives.
\emph{Step 4: integration of the bound for $G_{ij}^z$.}
By \eqref{eq:semigroup-linear-growth-bound-mkv} applied to $\partial_z^\ell\Phi_{ij}^{a,z}$, and using the bound proved in Step 3, we obtain
\begin{equation}
\label{eq:semigroup-bound-Gij-uniform-mkv}
\sum_{|\ell|\le2}
\tnorm{
P_u^a(\partial_z^\ell\Phi_{ij}^{a,z})(s,\mu)
}_{\xi}
\le
C(1+|s-z|+u^{1/2}),
\qquad u\in[0,1].
\end{equation}
Define
\[
R_{ij}^{a,\Delta t}(s,\mu,z)
:=
\frac1{\Delta t}
\int_0^{\Delta t}\int_0^\tau
P_u^a\Phi_{ij}^{a,z}(s,\mu)\,du\,d\tau .
\]
Differentiating under the integral with respect to $z$ and using \eqref{eq:semigroup-bound-Gij-uniform-mkv}, we get
\[
\sum_{|\ell|\le2}
\tnorm{
\partial_z^\ell R_{ij}^{a,\Delta t}(s,\mu,z)
}_{\xi}
\le
\frac1{\Delta t}
\int_0^{\Delta t}\int_0^\tau
C(1+|s-z|+u^{1/2})\,du\,d\tau .
\]
Since $\Delta t\le1$,
\begin{equation}
\label{eq:RGij-z-bound-uniform-mkv}
\sum_{|\ell|\le2}
\tnorm{
\partial_z^\ell R_{ij}^{a,\Delta t}(s,\mu,z)
}_{\xi}
\le
C(1+|s-z|)\Delta t .
\end{equation}

\emph{Step 5: set $z=s$.}
By definition,
\[
  R_{G_{ij}^s}^{a,\Delta t}(s,\mu)
  =
  R_{ij}^{a,\Delta t}(s,\mu,s).
\]
For the spatial derivatives, the chain rule gives
\[
\nabla_s R_{G_{ij}^s}^{a,\Delta t}(s,\mu)
=
\left.
(\nabla_s+\nabla_z)R_{ij}^{a,\Delta t}(s,\mu,z)
\right|_{z=s},
\]
and
\[
D_{ss}^2R_{G_{ij}^s}^{a,\Delta t}(s,\mu)
=
\left.
(D_{ss}^2+D_{sz}^2+D_{zs}^2+D_{zz}^2)
R_{ij}^{a,\Delta t}(s,\mu,z)
\right|_{z=s}.
\]
For the Lions derivatives, $z=s$ is independent of $\mu$, so
\[
\partial_\mu R_{G_{ij}^s}^{a,\Delta t}(s,\mu)(\xi)
=
\left.
\partial_\mu R_{ij}^{a,\Delta t}(s,\mu,z)(\xi)
\right|_{z=s},
\]
and the same identity holds for $D_\xi\partial_\mu$. Evaluating \eqref{eq:RGij-z-bound-uniform-mkv} at $z=s$ gives
\[
  \tnorm{R_{G_{ij}^s}^{a,\Delta t}(s,\mu)}_{\xi}
  \le C\Delta t .
\]
Together with the estimate for $F_k$, this proves
\eqref{eq:RFk-differentiated-bound-mkv}.

It remains to prove the local Lipschitz assertion. Fix
$\Delta t\in(0,1]$ and $R>0$, and write
\[
d(Z,\bar Z)
:=
|s-\bar s|+\mathcal W_2(\mu,\bar\mu)+|\xi-\bar\xi|,
\qquad
Z=(s,\mu,\xi),\quad
\bar Z=(\bar s,\bar\mu,\bar\xi).
\]
Under Assumption~\ref{ass:MF-PhiBE-structure}, the frozen-action flow
has the third-order state and mixed state--Lions variational derivatives
obtained by differentiating the variational equations in
\cite[Lemma~A.4]{firstpaper}. The coefficients of those equations are
bounded uniformly in $a$, because the derivatives of $b$ and $\sigma$
up to order four are bounded uniformly in $a$. Consequently, for
\[
D\in
\left\{
I,\ \nabla_s,\ D_{ss}^2,\
\partial_\mu(\,\cdot\,)(\xi),\
D_\xi\partial_\mu(\,\cdot\,)(\xi)
\right\},
\]
there exists $L_{R,\Delta t}<\infty$, independent of $a$, such that
\[
\begin{aligned}
&|D\hat b(s,\mu,a)-D\hat b(\bar s,\bar\mu,a)|\\
&\quad+
|D\hat\Sigma(s,\mu,a)-D\hat\Sigma(\bar s,\bar\mu,a)|
\le
L_{R,\Delta t}d(Z,\bar Z)
\end{aligned}
\]
whenever both triples belong to the radius-$R$ set above. For
$\hat\Sigma$, this follows by differentiating
\[
\hat\Sigma(s,\mu,a)
=
\frac1{\Delta t}
\mathbb E\!\left[(s_{\Delta t}^a-s)(s_{\Delta t}^a-s)^\top\right]
\]
and applying Cauchy--Schwarz together with the corresponding variational
flow estimates. The same local Lipschitz property holds for $b$ and
$\Sigma$ because their next derivatives are bounded under
Assumption~\ref{ass:MF-PhiBE-structure}. Finally,
\[
R_{F_k}^{a,\Delta t}=\hat b_k-b_k,
\qquad
R_{G_{ij}^s}^{a,\Delta t}=\hat\Sigma_{ij}-\Sigma_{ij},
\]
up to the immaterial sign convention in
\eqref{eq:local-expansion-general-mkv}. The preceding estimates therefore
give the asserted local Lipschitz property with a constant
$C_{R,\Delta t}$ independent of $a$.
\end{proof}
\begin{lemma}[First-order coefficient error]
\label{lem:first-order-coeff-error-regularity-norm-mkv}
Assume \Cref{ass:MF-PhiBE-structure}. Then there exists a constant
$C>0$, depending only on the constants in \Cref{ass:MF-PhiBE-structure}
and on $d$, such that, for every $a\in\mathcal A$, $\Delta t\in(0,1]$,
and $(s,\mu,\xi)\in\R^d\times\mathcal P_2(\R^d)\times\R^d$,
\begin{equation}
\label{eq:b-error-regularity-norm-mkv}
\tnorm{\hat b(s,\mu,a)-b(s,\mu,a)}_{\xi}\le C\Delta t,
\qquad
\tnorm{\hat\Sigma(s,\mu,a)-\Sigma(s,\mu,a)}_{\xi}\le C\Delta t.
\end{equation}
\end{lemma}

\begin{proof}
Fix $a\in\mathcal A$, $\Delta t\in(0,1]$, and
$(s,\mu,\xi)\in\R^d\times\mathcal P_2(\R^d)\times\R^d$.

For the drift, let $F_k(x,\nu):=x_k$. By
\eqref{eq:Fk-identities-mkv}, \eqref{eq:local-expansion-general-mkv},
and the definition of the remainder,
\[
\hat b_k(s,\mu,a)-b_k(s,\mu,a)
=
R_{F_k}^{a,\Delta t}(s,\mu),
\qquad k=1,\dots,d.
\]
Hence, by \Cref{lem:differentiated-remainder-bounds-mkv},
\[
\tnorm{\hat b_k(s,\mu,a)-b_k(s,\mu,a)}_{\xi}
=
\tnorm{R_{F_k}^{a,\Delta t}(s,\mu)}_{\xi}
\le C\Delta t.
\]
Summing over $k=1,\dots,d$ and using equivalence of norms in finite
dimension gives
\[
\tnorm{\hat b(s,\mu,a)-b(s,\mu,a)}_{\xi}\le C\Delta t.
\]

For the diffusion matrix, let
$G_{ij}^s(x,\nu):=(x_i-s_i)(x_j-s_j)$. By
\eqref{eq:Gij-identities-mkv}, \eqref{eq:local-expansion-general-mkv},
and the definition of the remainder,
\[
\hat\Sigma_{ij}(s,\mu,a)-\Sigma_{ij}(s,\mu,a)
=
R_{G_{ij}^s}^{a,\Delta t}(s,\mu),
\qquad i,j=1,\dots,d.
\]
Therefore, by \Cref{lem:differentiated-remainder-bounds-mkv},
\[
\tnorm{\hat\Sigma_{ij}(s,\mu,a)-\Sigma_{ij}(s,\mu,a)}_{\xi}
=
\tnorm{R_{G_{ij}^s}^{a,\Delta t}(s,\mu)}_{\xi}
\le C\Delta t.
\]
Summing over $i,j=1,\dots,d$ and using equivalence of norms in finite
dimension gives
\[
\tnorm{\hat\Sigma(s,\mu,a)-\Sigma(s,\mu,a)}_{\xi}\le C\Delta t.
\]
This proves \eqref{eq:b-error-regularity-norm-mkv}.
\end{proof}

\begin{lemma}[Evaluation regularity of the optimal policies]
\label{lem:optimal-policies-satisfy-A1-mkv}
Assume \Cref{ass:MF-PhiBE-structure} and \Cref{ass:optimal-policy-regularity}.
Let $\Pi_*$ be the family appearing in Assumption~\ref{ass:optimal-policy-regularity}, and let $b^\pi$ and $\Sigma^\pi$ be
defined as in \eqref{eq:aggregated_coeffs_mkv}. Set $\sigma^\pi(s,\mu):=\bigl(\Sigma^\pi(s,\mu)\bigr)^{1/2}$. Then, for every $\pi\in\Pi_*$, the coefficient pair
$(b^\pi,\sigma^\pi)$ satisfies the coefficient part of
Assumption~\ref{ass:_mkv_regularity}, with coefficient bounds and local Lipschitz constants uniform in
$\pi\in\Pi_*$.
\end{lemma}
\begin{proof}
Fix $\pi\in\Pi_*$. By \Cref{ass:optimal-policy-regularity}, we may write
$\pi(da\mid s,\mu)=p_\pi(a\mid s,\mu)\nu(da)$. By Assumption~\ref{ass:optimal-policy-regularity}, there exists a single function $\ell_*\in L^1(\nu)$, independent of $\pi\in\Pi_*$, such that $p_\pi$ and all the derivatives of $p_\pi$ appearing below are bounded by $\ell_*(a)$, uniformly in $(s,\mu,\xi)$ and $\pi\in\Pi_*$. More precisely, for $\nu$-a.e. $a\in\mathcal A$,
\[
|p_\pi(a\mid s,\mu)|
+\|\nabla_s p_\pi(a\mid s,\mu)\|
+\|D_{ss}^2p_\pi(a\mid s,\mu)\|
+|\partial_\mu p_\pi(a\mid s,\mu)(\xi)|
+\|D_\xi\partial_\mu p_\pi(a\mid s,\mu)(\xi)\|
\le \ell_*(a)
\]
In addition, for every $R>0$, the same derivatives are locally
Lipschitz in $(s,\mu,\xi)$ on the set
$|s|+\sqrt{m_2(\mu)}+|\xi|\le R$, with a Lipschitz constant bounded by
a function $\ell_{*,R}\in L^1(\nu)$ independent of $\pi\in\Pi_*$.

We first prove an auxiliary statement. Let $h=h(s,\mu,a)$ be scalar
valued and assume that $h$, $\nabla_s h$, $D_{ss}^2h$, $\partial_\mu h$,
and $D_\xi\partial_\mu h$ are bounded and locally Lipschitz in their
variables, uniformly with respect to $a$. Define
\[
h^\pi(s,\mu):=\int_{\mathcal A}h(s,\mu,a)p_\pi(a\mid s,\mu)\nu(da).
\]
Since the integrands obtained by differentiating with respect to $s$ and
$\mu$ are bounded by an $L^1(\nu)$ function, the differentiation under
the integral sign is justified. Thus
\[
\nabla_s h^\pi
=
\int_{\mathcal A}
\bigl[
(\nabla_s h)p_\pi+h\nabla_s p_\pi
\bigr]\,\nu(da),
\]
\[
D_{ss}^2h^\pi
=
\int_{\mathcal A}
\bigl[
(D_{ss}^2h)p_\pi
+(\nabla_s h)\otimes(\nabla_s p_\pi)
+(\nabla_s p_\pi)\otimes(\nabla_s h)
+hD_{ss}^2p_\pi
\bigr]\,\nu(da),
\]
and
\[
\partial_\mu h^\pi(\xi)
=
\int_{\mathcal A}
\bigl[
\partial_\mu h(\xi)p_\pi
+h\partial_\mu p_\pi(\xi)
\bigr]\,\nu(da),
\qquad
D_\xi\partial_\mu h^\pi(\xi)
=
\int_{\mathcal A}
\bigl[
D_\xi\partial_\mu h(\xi)p_\pi
+hD_\xi\partial_\mu p_\pi(\xi)
\bigr]\,\nu(da).
\]
Here all terms are evaluated at $(s,\mu,a)$, except the Lions
derivatives, which are evaluated at $(s,\mu,a,\xi)$. The uniform bounds
on $h$ and its derivatives, together with the bound by $\ell_*$, give
\[
\|\nabla_s h^\pi(s,\mu)\|
+\|D_{ss}^2h^\pi(s,\mu)\|
+|\partial_\mu h^\pi(s,\mu)(\xi)|
+\|D_\xi\partial_\mu h^\pi(s,\mu)(\xi)\|
\le K_* .
\]
We now check local Lipschitz continuity. Let
$Z=(s,\mu,\xi)$ and $\bar Z=(\bar s,\bar\mu,\bar\xi)$ belong to
$|s|+\sqrt{m_2(\mu)}+|\xi|\le R$ and
$|\bar s|+\sqrt{m_2(\bar\mu)}+|\bar\xi|\le R$. Set
\[
d(Z,\bar Z):=
|s-\bar s|+\mathcal W_2(\mu,\bar\mu)+|\xi-\bar\xi|.
\]
For each term in the four displayed formulas for
$\nabla_s h^\pi$, $D_{ss}^2h^\pi$, $\partial_\mu h^\pi$, and
$D_\xi\partial_\mu h^\pi$, we use
\[
|AB-\bar A\bar B|
\le |A-\bar A|\,|B|+|\bar A|\,|B-\bar B|.
\]
For instance, in the term
$\int_{\mathcal A}(\nabla_s h)p_\pi\,\nu(da)$, we have
\[
\begin{aligned}
&\left|
(\nabla_s h)(s,\mu,a)p_\pi(a\mid s,\mu)
-
(\nabla_s h)(\bar s,\bar\mu,a)p_\pi(a\mid \bar s,\bar\mu)
\right|\le
C_R d(Z,\bar Z)\ell_*(a)
+
C\,\ell_{*,R}(a)d(Z,\bar Z).
\end{aligned}
\]
The right-hand side is integrable with respect to $\nu(da)$. The same
argument applies to each product appearing in the four formulas above,
using the local Lipschitz bounds for $h$ and its derivatives and the
$L^1(\nu)$ local Lipschitz bounds for $p_\pi$ and its derivatives.
Therefore,
\[
\begin{aligned}
&\|\nabla_s h^\pi(s,\mu)-\nabla_s h^\pi(\bar s,\bar\mu)\|
+\|D_{ss}^2h^\pi(s,\mu)-D_{ss}^2h^\pi(\bar s,\bar\mu)\|
\\
&\quad
+|\partial_\mu h^\pi(s,\mu)(\xi)
-\partial_\mu h^\pi(\bar s,\bar\mu)(\bar\xi)|
+\|D_\xi\partial_\mu h^\pi(s,\mu)(\xi)
-D_\xi\partial_\mu h^\pi(\bar s,\bar\mu)(\bar\xi)\|\le C_{*,R}d(Z,\bar Z).
\end{aligned}
\]
This proves the auxiliary statement.

Applying the auxiliary statement to $h=b_i$, for $i=1,\dots,d$, gives
the required boundedness and local Lipschitz regularity of $b^\pi$.
Applying the same statement to $h=\Sigma_{ij}$, for
$i,j=1,\dots,d$, gives the corresponding properties for $\Sigma^\pi$.
Indeed, $\Sigma=\sigma\sigma^\top$, and \Cref{ass:MF-PhiBE-structure}
gives the required bounds and local Lipschitz regularity for $\sigma$
and its derivatives.

It remains to pass from $\Sigma^\pi$ to
$\sigma^\pi=(\Sigma^\pi)^{1/2}$. By \Cref{ass:MF-PhiBE-structure},
\[
v^\top\Sigma^\pi(s,\mu)v
=
\int_{\mathcal A}
v^\top\Sigma(s,\mu,a)v\,p_\pi(a\mid s,\mu)\nu(da)
\ge \alpha |v|^2.
\]
Since $\sigma$ is bounded, there exists $M>0$ such that
$\|\Sigma^\pi(s,\mu)\|\le M$ for all $(s,\mu)$. Hence the image of
$\Sigma^\pi$ is contained in
\[
\mathcal S_{\alpha,M}
:=
\{A\in\mathbb S^d:\alpha I_d\preceq A,\ \|A\|\le M\}.
\]
The principal square-root map $\Psi(A)=A^{1/2}$ is smooth on the positive
definite cone. Hence $D\Psi$ and $D^2\Psi$ are bounded and locally
Lipschitz on $\mathcal S_{\alpha,M}$.

Since $\sigma^\pi=\Psi(\Sigma^\pi)$, the chain rule gives
\[
\nabla_s\sigma^\pi
=
D\Psi(\Sigma^\pi)[\nabla_s\Sigma^\pi],
\]
\[
D_{ss}^2\sigma^\pi
=
D\Psi(\Sigma^\pi)[D_{ss}^2\Sigma^\pi]
+
D^2\Psi(\Sigma^\pi)[\nabla_s\Sigma^\pi,\nabla_s\Sigma^\pi],
\]
and
\[
\partial_\mu\sigma^\pi(\xi)
=
D\Psi(\Sigma^\pi)[\partial_\mu\Sigma^\pi(\xi)],
\qquad
D_\xi\partial_\mu\sigma^\pi(\xi)
=
D\Psi(\Sigma^\pi)[D_\xi\partial_\mu\Sigma^\pi(\xi)].
\]
The bounds and local Lipschitz regularity already proved for
$\Sigma^\pi$, together with the boundedness and local Lipschitz
regularity of $D\Psi$ and $D^2\Psi$ on $\mathcal S_{\alpha,M}$, imply
the same coefficient regularity for $\sigma^\pi$.Therefore $(b^\pi,\sigma^\pi)$ satisfies the coefficient part of Assumption~\ref{ass:_mkv_regularity}, with constants independent of $\pi\in\Pi_*$.
\end{proof}

\begin{lemma}[Averaged coefficient consistency and transfer of evaluation regularity]
\label{lem:averaged-coeff-consistency-and-regularity-mkv}
Assume \Cref{ass:MF-PhiBE-structure} and
\Cref{ass:optimal-policy-regularity}. Let $\Pi_*$ be as in
Assumption~\ref{ass:optimal-policy-regularity}. For $\pi\in\Pi_*$,
consider $\hat b^\pi$ and $\hat\Sigma^\pi$ defined in
\eqref{eq:averaged_b_sigma_hat}, and set
$\hat\sigma^\pi=(\hat\Sigma^\pi)^{1/2}$. Then there exists a constant
$C_*>0$, independent of $\pi\in\Pi_*$ and $\Delta t\in(0,1]$, such that
\begin{equation}
\label{eq:averaged-b-error-mkv}
\tnorm{\hat b^\pi(s,\mu)-b^\pi(s,\mu)}_\xi
+\tnorm{\hat\Sigma^\pi(s,\mu)-\Sigma^\pi(s,\mu)}_\xi
\le C_*\Delta t.
\end{equation}
Moreover, there exists $\Delta_*\in(0,1]$, independent of $\pi\in\Pi_*$, such that for every $\Delta t\in(0,\Delta_*]$ the pair $(\hat b^\pi,\hat\sigma^\pi)$ satisfies the coefficient part of Assumption~\ref{ass:_mkv_regularity}, with derivative bounds uniform in
$\pi\in\Pi_*$ and $\Delta t\in(0,\Delta_*]$, and with local Lipschitz
constants uniform in $\pi\in\Pi_*$ for each fixed $\Delta t$.
\end{lemma}

\begin{proof}
Fix $\pi\in\Pi_*$. By \Cref{ass:optimal-policy-regularity}, we may write
$\pi(da\mid s,\mu)=p_\pi(a\mid s,\mu)\nu(da)$. There exists a common function $\ell_*\in L^1(\nu)$ such that $p_\pi$ and all derivatives of $p_\pi$ appearing below are bounded by $\ell_*(a)$, uniformly in $(s,\mu,\xi)$ and $\pi\in\Pi_*$.

Set
\[
\delta b(s,\mu,a):=\hat b(s,\mu,a)-b(s,\mu,a),
\qquad
\delta\Sigma(s,\mu,a):=\hat\Sigma(s,\mu,a)-\Sigma(s,\mu,a).
\]
By \Cref{lem:first-order-coeff-error-regularity-norm-mkv}, for every
component $i,j$,
\[
\tnorm{\delta b_i(s,\mu,a)}_\xi
+
\tnorm{\delta\Sigma_{ij}(s,\mu,a)}_\xi
\le C\Delta t,
\]
uniformly in $(s,\mu,\xi,a)$.

For each component of the averaged drift error,
\[
\hat b_i^\pi(s,\mu)-b_i^\pi(s,\mu)
=
\int_{\mathcal A}\delta b_i(s,\mu,a)p_\pi(a\mid s,\mu)\nu(da).
\]
Differentiating this identity with respect to $s$ and $\mu$ gives the
same product formulas used in the proof of
\Cref{lem:optimal-policies-satisfy-A1-mkv}, with $h$ replaced by
$\delta b_i$. Since $\delta b_i$ and the derivatives entering
$\tnorm{\cdot}_\xi$ are bounded by $C\Delta t$, and since $p_\pi$ and
its corresponding derivatives are bounded by $\ell_*\in L^1(\nu)$, we
obtain
\[
\tnorm{\hat b_i^\pi(s,\mu)-b_i^\pi(s,\mu)}_\xi
\le C_*\Delta t.
\]
Summing over $i=1,\dots,d$ gives \eqref{eq:averaged-b-error-mkv}.

The proof of \eqref{eq:averaged-b-error-mkv} is the same. Indeed,
\[
\hat\Sigma_{ij}^\pi(s,\mu)-\Sigma_{ij}^\pi(s,\mu)
=
\int_{\mathcal A}
\delta\Sigma_{ij}(s,\mu,a)p_\pi(a\mid s,\mu)\nu(da),
\]
and the same differentiation formulas, now with
$\delta\Sigma_{ij}$ in place of $\delta b_i$, give
\[
\tnorm{\hat\Sigma_{ij}^\pi(s,\mu)-\Sigma_{ij}^\pi(s,\mu)}_\xi
\le C_*\Delta t.
\]
Summing over $i,j=1,\dots,d$ proves
\eqref{eq:averaged-b-error-mkv}.

We now verify the coefficient regularity of $\hat b^\pi$. By
\Cref{lem:optimal-policies-satisfy-A1-mkv}, $b^\pi$ satisfies the
coefficient part of Assumption~\ref{ass:_mkv_regularity}. Together with
\eqref{eq:averaged-b-error-mkv}, this proves the uniform boundedness of
all derivatives of $\hat b^\pi$ required in that assumption.

It remains to verify their local Lipschitz continuity. Fix
$\Delta t\in(0,1]$ and $R>0$, and write
\[
d\bigl((s,\mu,\xi),(\bar s,\bar\mu,\bar\xi)\bigr)
:=
|s-\bar s|+\mathcal W_2(\mu,\bar\mu)+|\xi-\bar\xi|.
\]
Under Assumption~\ref{ass:MF-PhiBE-structure}, the frozen-action flow is
three times differentiable with respect to the initial state and has the
corresponding mixed state--Lions derivatives required below. These
derivatives are obtained by differentiating the first- and second-order
variational equations used in the proof of
\Cref{lem:differentiated-remainder-bounds-mkv}. Since the coefficients
and their spatial and Lions derivatives up to order four are uniformly
bounded, the resulting variational equations have bounded coefficients.
Consequently, for every differential operator
\[
D\in
\left\{
\nabla_s,\ D_{ss}^2,\ \partial_\mu(\,\cdot\,)(\xi),\
D_\xi\partial_\mu(\,\cdot\,)(\xi)
\right\},
\]
there exists $L_{R,\Delta t}<\infty$, independent of $a$, such that
\[
\begin{aligned}
&\left|
D\hat b(s,\mu,a)-D\hat b(\bar s,\bar\mu,a)
\right|
\\
&\qquad\le
L_{R,\Delta t}\,
d\bigl((s,\mu,\xi),(\bar s,\bar\mu,\bar\xi)\bigr)
\end{aligned}
\]
whenever
\[
|s|+\sqrt{m_2(\mu)}+|\xi|
+
|\bar s|+\sqrt{m_2(\bar\mu)}+|\bar\xi|
\le R.
\]
The same conclusion holds for $\hat\Sigma$. Indeed, differentiating
\[
\hat\Sigma(s,\mu,a)
=
\frac1{\Delta t}
\mathbb E\!\left[
(s_{\Delta t}-s)(s_{\Delta t}-s)^\top
\right]
\]
produces finite sums of products of increments of the frozen-action
flow and its first-, second-, and third-order variational derivatives.
The required local Lipschitz estimates follow from the corresponding
variational-flow estimates, the moment bound
\[
\mathbb E|s_t-s|^p\le C_p t^{p/2},
\qquad 0\le t\le1,
\]
and the Cauchy--Schwarz inequality.

Differentiating the policy averages under the integral sign and using
the common $L^1(\nu)$ local Lipschitz dominator from
Assumption~\ref{ass:optimal-policy-regularity} now gives
\[
\begin{aligned}
&\left|
D\hat b^\pi(s,\mu)-D\hat b^\pi(\bar s,\bar\mu)
\right|
\\
&\quad+
\left|
D\hat\Sigma^\pi(s,\mu)-D\hat\Sigma^\pi(\bar s,\bar\mu)
\right|
\\
&\qquad\le
L_{*,R,\Delta t}\,
d\bigl((s,\mu,\xi),(\bar s,\bar\mu,\bar\xi)\bigr),
\end{aligned}
\]
uniformly in $\pi\in\Pi_*$. Thus $\hat b^\pi$ and
$\hat\Sigma^\pi$ have the local Lipschitz regularity required in
Assumption~\ref{ass:_mkv_regularity}.

It remains to treat $\hat\sigma^\pi$. From
\eqref{eq:averaged-b-error-mkv},
\[
\|\hat\Sigma^\pi(s,\mu)-\Sigma^\pi(s,\mu)\|\le C_*\Delta t.
\]
By \Cref{lem:optimal-policies-satisfy-A1-mkv},
$\Sigma^\pi(s,\mu)\succeq \alpha I_d$. Choose
$\Delta_*\in(0,1]$ such that $C_*\Delta_*\le\alpha/2$. Then, for
$\Delta t\in(0,\Delta_*]$,
\[
\hat\Sigma^\pi(s,\mu)\succeq \frac{\alpha}{2}I_d.
\]
Moreover, since $\Sigma^\pi$ is uniformly bounded and
\eqref{eq:averaged-b-error-mkv} holds, there exists $M_*>0$ such
that $\|\hat\Sigma^\pi(s,\mu)\|\le M_*$ for all $(s,\mu)$. Thus the
image of $\hat\Sigma^\pi$ is contained in
\[
\mathcal S_{\alpha/2,M_*}
:=
\{A\in\mathbb S^d:\tfrac{\alpha}{2}I_d\preceq A,\ \|A\|\le M_*\}.
\]
The principal square-root map $\Psi(A)=A^{1/2}$ is smooth on the positive
definite cone. Hence $D\Psi$ and $D^2\Psi$ are bounded and locally
Lipschitz on $\mathcal S_{\alpha/2,M_*}$.

Since $\hat\sigma^\pi=\Psi(\hat\Sigma^\pi)$, the chain rule gives
\[
\nabla_s\hat\sigma^\pi
=
D\Psi(\hat\Sigma^\pi)[\nabla_s\hat\Sigma^\pi],
\]
\[
D_{ss}^2\hat\sigma^\pi
=
D\Psi(\hat\Sigma^\pi)[D_{ss}^2\hat\Sigma^\pi]
+
D^2\Psi(\hat\Sigma^\pi)[\nabla_s\hat\Sigma^\pi,\nabla_s\hat\Sigma^\pi],
\]
and
\[
\partial_\mu\hat\sigma^\pi(\xi)
=
D\Psi(\hat\Sigma^\pi)[\partial_\mu\hat\Sigma^\pi(\xi)],
\qquad
D_\xi\partial_\mu\hat\sigma^\pi(\xi)
=
D\Psi(\hat\Sigma^\pi)[D_\xi\partial_\mu\hat\Sigma^\pi(\xi)].
\]
The bounds and local Lipschitz regularity of the derivatives of
$\hat\Sigma^\pi$, together with the boundedness and local Lipschitz
regularity of $D\Psi$ and $D^2\Psi$ on $\mathcal S_{\alpha/2,M_*}$,
imply that $\hat\sigma^\pi$ satisfies the coefficient regularity required
in \Cref{ass:_mkv_regularity}. Therefore
$(\hat b^\pi,\hat\sigma^\pi)$ satisfies the coefficient part of
\Cref{ass:_mkv_regularity}.
\end{proof}

\begin{proof}[\textbf{Proof of
\Cref{thm:error_V_eval_optimal_pure_mkv}}]
Let $\Pi_*:=\{\pi^*\}\cup \{\hat\pi_{\Delta t}^*:0<\Delta t\le\Delta_0\}$. By Lemmas~\ref{lem:optimal-policies-satisfy-A1-mkv} and \ref{lem:averaged-coeff-consistency-and-regularity-mkv}, after decreasing $\Delta_0$ if necessary, the exact and estimated coefficient pairs satisfy Assumption~\ref{ass:_mkv_regularity}. Their derivative bounds are uniform in $\pi\in\Pi_*$ and $0<\Delta t\le\Delta_0$, whereas their local Lipschitz constants are uniform in $\pi$ for each fixed $\Delta t$. Since the HJB threshold and the derivative estimate depend only on the uniform derivative and growth bounds, the thresholds, the constants in \eqref{eq:value-derivative-bound-mkv}, and the second-moment exponent may be chosen as common constants $\bar\beta_*$, $C_*$ and $\beta_0(2)$. Fix $\beta>\max\{\bar\beta_*,\beta_0(2)\}$ and $\Delta t\in(0,\Delta_0]$.
\medskip

\noindent
\textbf{Step 1: fixed-policy estimate.}
Fix $\pi\in\Pi_*$. By the choice of $\bar\beta_*$, the functions $\mathcal V^\pi$ and $\hat{\mathcal V}^\pi$ are the unique classical solutions of
\[
(\mathcal L_{b,\Sigma}^{\pi}-\beta)\mathcal V^\pi+r_\lambda^\pi=0
\]
and
\[
(\mathcal L_{\hat b,\hat\Sigma}^{\pi}-\beta)\hat{\mathcal V}^\pi+r_\lambda^\pi=0,
\]
respectively. Set $w^\pi:=\mathcal V^\pi-\hat{\mathcal V}^\pi$. Subtracting the two equations gives
\begin{equation}
\label{eq:w-linear-HJB-fixed-policy-mkv}
(\mathcal L_{b,\Sigma}^{\pi}-\beta)w^\pi+
(\mathcal L_{b,\Sigma}^{\pi}-\mathcal L_{\hat b,\hat\Sigma}^{\pi})
\hat{\mathcal V}^\pi=0.
\end{equation}

Let $(\mu_t)_{t\ge0}$ be the law flow generated by \eqref{eq:MKV_only_SDE} under the policy $\pi$, with $\mu_0=\mu$. Applying the It\^o-Lions formula to $e^{-\beta t}w^\pi(\mu_t)$ and using \eqref{eq:w-linear-HJB-fixed-policy-mkv}, we obtain, for every $T>0$,
\[
w^\pi(\mu)=e^{-\beta T}w^\pi(\mu_T)+\int_0^T e^{-\beta t}\Bigl[(\mathcal L_{b,\Sigma}^{\pi}-\mathcal L_{\hat b,\hat\Sigma}^{\pi})\hat{\mathcal V}^\pi\Bigr](\mu_t)\,dt.
\]
By the polynomial growth of $\mathcal V^\pi$ and $\hat{\mathcal V}^\pi$, together with the second-moment estimate associated with $\beta_0(2)$,
\[
e^{-\beta T}|w^\pi(\mu_T)|\le C e^{-(\beta-\beta_0(2))T}(1+m_2(\mu))\longrightarrow0,
\qquad T\to\infty.
\]
Hence
\begin{equation}
\label{eq:w-representation-fixed-policy-mkv}
w^\pi(\mu)=\int_0^\infty e^{-\beta t}\Bigl[(\mathcal L_{b,\Sigma}^{\pi}-\mathcal L_{\hat b,\hat\Sigma}^{\pi})\hat{\mathcal V}^\pi\Bigr](\mu_t)\,dt.
\end{equation}

For $\nu\in\mathcal P_2(\R^d)$, by \eqref{eq:L-pi-mkv} and \eqref{eq:estimated-generator-mkv},
\[
\begin{aligned}
\Bigl[(\mathcal L_{b,\Sigma}^{\pi}-\mathcal L_{\hat b,\hat\Sigma}^{\pi})\hat{\mathcal V}^\pi\Bigr](\nu)
&=\int_{\R^d}\bigl(b^\pi(x,\nu)-\hat b^\pi(x,\nu)\bigr)\cdot\partial_\mu\hat{\mathcal V}^\pi(\nu)(x)\,\nu(dx) \\
&\quad+\frac12\int_{\R^d}\bigl(\Sigma^\pi(x,\nu)-\hat\Sigma^\pi(x,\nu)\bigr):D_x\partial_\mu\hat{\mathcal V}^\pi(\nu)(x)\,\nu(dx).
\end{aligned}
\]
By \Cref{lem:averaged-coeff-consistency-and-regularity-mkv},
\[
\|b^\pi(x,\nu)-\hat b^\pi(x,\nu)\|+\|\Sigma^\pi(x,\nu)-\hat\Sigma^\pi(x,\nu)\|\le C_*\Delta t.
\]
Moreover, the derivative estimate in \Cref{thm:MF_eval_infinite_mkv}, applied to the MF-PhiBE, gives
\[
|\partial_\mu\hat{\mathcal V}^\pi(\nu)(x)|+\|D_x\partial_\mu\hat{\mathcal V}^\pi(\nu)(x)\|\le\frac{C_*}{\beta-\bar\beta_*}(1+|x|^2+m_2(\nu)).
\]
Combining the previous two estimates yields
\[
\begin{aligned}
\left|\Bigl[(\mathcal L_{b,\Sigma}^{\pi}-\mathcal L_{\hat b,\hat\Sigma}^{\pi})\hat{\mathcal V}^\pi\Bigr](\nu)\right|
&\le\frac{C_*\Delta t}{\beta-\bar\beta_*}
\int_{\R^d}(1+|x|^2+m_2(\nu))\,\nu(dx)\le\frac{C_*}{\beta-\bar\beta_*}(1+m_2(\nu))\Delta t.
\end{aligned}
\]
Using this bound in \eqref{eq:w-representation-fixed-policy-mkv}, we get
\[
|w^\pi(\mu)|\le\frac{C_*\Delta t}{\beta-\bar\beta_*}\int_0^\infty e^{-\beta t}(1+m_2(\mu_t))\,dt.
\]
The moment estimate associated with $\beta_0(2)$ gives
\[
1+m_2(\mu_t)\le C e^{\beta_0(2)t}(1+m_2(\mu)).
\]
Since $\beta>\beta_0(2)$, we obtain
\begin{equation}
\label{eq:fixed-policy-error-final-mkv}
|\mathcal V^\pi(\mu)-\hat{\mathcal V}^\pi(\mu)|
\le
\frac{C_*}{(\beta-\bar\beta_*)(\beta-\beta_0(2))}
(1+m_2(\mu))\Delta t,
\qquad
\pi\in\Pi_*.
\end{equation}

\medskip

\noindent
\textbf{Step 2: optimality argument.}
By optimality of $\pi^*$ for the exact problem,
\[
\mathcal V^*(\mu)=\mathcal V^{\pi^*}(\mu).
\]
For the fixed value of $\Delta t$ considered above, optimality of
$\hat\pi_{\Delta t}^*$ for the MF-PhiBE gives
\begin{equation}
\label{eq:estimated-optimality-proof-mkv}
\hat{\mathcal V}^{\hat\pi_{\Delta t}^*}(\mu)
\ge
\hat{\mathcal V}^{\pi^*}(\mu).
\end{equation}
Therefore,
\[
\begin{aligned}
\mathcal V^*(\mu)
-\mathcal V^{\hat\pi_{\Delta t}^*}(\mu)
&=
\bigl(\mathcal V^{\pi^*}-\hat{\mathcal V}^{\pi^*}\bigr)(\mu)
\\
&\quad+
\bigl(
\hat{\mathcal V}^{\pi^*}
-\hat{\mathcal V}^{\hat\pi_{\Delta t}^*}
\bigr)(\mu)
\\
&\quad+
\bigl(
\hat{\mathcal V}^{\hat\pi_{\Delta t}^*}
-\mathcal V^{\hat\pi_{\Delta t}^*}
\bigr)(\mu).
\end{aligned}
\]
The middle term is nonpositive by
\eqref{eq:estimated-optimality-proof-mkv}. Applying
\eqref{eq:fixed-policy-error-final-mkv} to
$\pi=\pi^*$ and $\pi=\hat\pi_{\Delta t}^*$ gives
\[
\mathcal V^*(\mu)
-\mathcal V^{\hat\pi_{\Delta t}^*}(\mu)
\le
\frac{2C_*}{(\beta-\bar\beta_*)(\beta-\beta_0(2))}
(1+m_2(\mu))\Delta t.
\]
Absorbing the factor $2$ into $C_*$ gives the asserted upper bound.
Finally, since $\hat\pi_{\Delta t}^*\in\Pi_{\mathrm{add}}$, the
definition of $\mathcal V^*$ gives
\[
0
\le
\mathcal V^*(\mu)
-\mathcal V^{\hat\pi_{\Delta t}^*}(\mu).
\]
This proves \eqref{eq:error_V_optimal_pure_mkv}.
\end{proof}

\subsection{Proofs for Section~\ref{sec:LQR}}
\label{app:LQR-proofs}

\begin{lemma}[First-order approximation error of the estimated matrices]
\label{lem:first_order_coeff_error}
Let $c_A:=\|A\|,$ $c_{\tilde A}:=\|\tilde A\|,$ and $c_{\bar A}:=\|\bar A\|.$ Then
\[
    \|\hat A-A\|
    \le
    \frac12 c_A^2 e^{c_A\Delta t}\Delta t,\qquad 
    \|\hat{\bar A}-\bar A\|
    \le
    \frac12 c_{\bar A}(c_A+c_{\tilde A})
    e^{(c_A+c_{\tilde A})\Delta t}\Delta t,
\]
and
\[
    \|\hat B-B\|
    \le
    \frac12
    (c_A+c_{\bar A})e^{(c_A+c_{\tilde A})\Delta t}\|B\|\Delta t.
\]
\end{lemma}

\begin{proof}
For $\hat A$, we have
\begin{align}\label{eq:estimate_A_hat_A}
     \hat A-A
    =
    \frac{e^{A\Delta t}-I_d-A\Delta t}{\Delta t}
    =
    \frac1{\Delta t}\int_0^{\Delta t}(e^{A\tau}-I_d)A\,d\tau.
\end{align}
Taking norm in \eqref{eq:estimate_A_hat_A}, and using that $\|e^{A\tau}-I_d\|
\le
\tau c_A e^{c_A\tau},$ we obtain
\[
    \|\hat A-A\|
    \le
    \frac{c_A}{\Delta t}\int_0^{\Delta t}\tau c_A e^{c_A\tau}\,d\tau
    \le
    \frac{c_A^2}{\Delta t}e^{c_A\Delta t}\int_0^{\Delta t}\tau\,d\tau
    =
    \frac12 c_A^2 e^{c_A\Delta t}\Delta t.
\]

For $\hat{\bar A}$, we write
\[
    \hat{\bar A}-\bar A
    =
    \frac1{\Delta t}\int_0^{\Delta t}
    \Bigl(
        e^{A(\Delta t-\tau)}\bar A e^{\tilde A\tau}-\bar A
    \Bigr)\,d\tau.
\]
Adding and subtracting $\bar A e^{\tilde A\tau}$ and taking norm, we have
\begin{align}\label{eq:estimate_A_hat_A_tilde}
        \|\hat{\bar A}-\bar A\|
    \le
    \frac{c_{\bar A}}{\Delta t}\int_0^{\Delta t}
    \Bigl[
        \|e^{A(\Delta t-\tau)}-I_d\|\,\|e^{\tilde A\tau}\|
        +
        \|e^{\tilde A\tau}-I_d\|
    \Bigr]\,d\tau.
\end{align}
Combining
\[
\|e^{A(\Delta t-\tau)}-I_d\|
\le
(\Delta t-\tau)c_A e^{c_A(\Delta t-\tau)},
\qquad
\|e^{\tilde A\tau}\|\le e^{c_{\tilde A}\tau},\quad
\text{and}\quad\|e^{\tilde A\tau}-I_d\|
\le
\tau c_{\tilde A}e^{c_{\tilde A}\tau},
\]
with \eqref{eq:estimate_A_hat_A_tilde}, we conclude that
\[
    \|\hat{\bar A}-\bar A\|
    \le
    \frac{c_{\bar A}}{\Delta t}
    \int_0^{\Delta t}
    \Bigl[
        c_A(\Delta t-\tau)
        +
        c_{\tilde A}\tau
    \Bigr]
    e^{(c_A+c_{\tilde A})\Delta t}\,d\tau\le
    \frac12 c_{\bar A}(c_A+c_{\tilde A})
    e^{(c_A+c_{\tilde A})\Delta t}\Delta t.
\]

Finally, for $B$ we have
\begin{align}\label{eq:error_B_bhat}
        \hat B-B
    =
    \frac1{\Delta t}\int_0^{\Delta t}
    (e^{A(\Delta t-\tau)}-I_d)B\,d\tau
    +
    \frac1{\Delta t}\int_0^{\Delta t}
    e^{A(\Delta t-\tau)}\bar A M_{\tilde A,\tau}B\,d\tau.
\end{align}
For the first term on the right-hand side of \eqref{eq:error_B_bhat} we have,
\begin{align}\label{eq:estimate_b_bhat1}
     \left\|
    \frac1{\Delta t}\int_0^{\Delta t}
    (e^{A(\Delta t-\tau)}-I_d)B\,d\tau
    \right\|
    \le
    \frac{\|B\|}{\Delta t}\int_0^{\Delta t}
    (\Delta t-\tau)c_Ae^{c_A(\Delta t-\tau)}\,d\tau
    \le
    \frac12 c_A e^{c_A\Delta t}\|B\|\Delta t.
\end{align}
For the second term on the right-hand side of \eqref{eq:error_B_bhat}, since
\[
    \|M_{\tilde A,\tau}\|
    =
    \left\|
        \int_0^\tau e^{\tilde A(\tau-r)}\,dr
    \right\|
    \le
    \int_0^\tau e^{c_{\tilde A}(\tau-r)}\,dr
    \le
    \tau e^{c_{\tilde A}\tau},
\]
we have
\begin{align}\label{eq:estimate_b_bhat2}
    \frac1{\Delta t}\int_0^{\Delta t}
    \|e^{A(\Delta t-\tau)}\bar A M_{\tilde A,\tau}B\|\,d\tau
    \le
    \frac{c_{\bar A}\|B\|}{\Delta t}
    \int_0^{\Delta t}
    e^{c_A(\Delta t-\tau)}\tau e^{c_{\tilde A}\tau}\,d\tau
    \le
    \frac12 c_{\bar A}e^{(c_A+c_{\tilde A})\Delta t}\|B\|\Delta t.
\end{align}
Combining both estimates \eqref{eq:estimate_b_bhat1} and \eqref{eq:estimate_b_bhat2}, and using $e^{c_A\Delta t}\le e^{(c_A+c_{\tilde A})\Delta t}$, gives
\[
    \|\hat B-B\|
    \le
    \frac12
    \left(
        c_A e^{c_A\Delta t}
        +
        c_{\bar A}e^{(c_A+c_{\tilde A})\Delta t}
    \right)
    \|B\|\Delta t\leq \frac12
    (c_A+c_{\bar A})e^{(c_A+c_{\tilde A})\Delta t}
    \|B\|\Delta t.
\]
\end{proof}

\begin{lemma}[Assumptions on the estimated coefficients]
\label{lem:LQR-MKV-assumption-stability-hat}
Assume that
\Cref{ass:LQR-MKV} holds for the original coefficients
$(A,\bar A,B)$. Let $\hat A$, $\hat{\bar A}$, and $\hat B$ be defined by
\eqref{eq:definition_coef_HAT}. Then there exists $\Delta t_0>0$ such
that, for every $\Delta t\in(0,\Delta t_0]$, \Cref{ass:LQR-MKV} also
holds with $(A,\bar A,B)$ replaced by
$(\hat A,\hat{\bar A},\hat B)$.
\end{lemma}
\begin{proof}
Condition {\rm (H1)} is unchanged because it depends only on
$Q$, $\bar Q$, and $N$.

By {\rm (H2)}, there exist
$\Theta_1,\Theta_2\in\mathbb R^{m\times d}$ such that
\[
L_1:=A_\beta+B\Theta_1,
\qquad
L_2:=\tilde A_\beta+B\Theta_2
\]
are Hurwitz. Define
\[
\hat L_1:=
\left(\hat A-\frac{\beta}{2}I_d\right)+\hat B\Theta_1,
\qquad
\hat L_2:=
\left(\hat A+\hat{\bar A}-\frac{\beta}{2}I_d\right)
+\hat B\Theta_2.
\]
Then
\[
\hat L_1-L_1
=
(\hat A-A)+(\hat B-B)\Theta_1,
\]
and
\[
\hat L_2-L_2
=
(\hat A-A)+(\hat{\bar A}-\bar A)
+(\hat B-B)\Theta_2.
\]
By \Cref{lem:first_order_coeff_error},
\[
\|\hat L_1-L_1\|+\|\hat L_2-L_2\|
\longrightarrow0
\qquad\text{as }\Delta t\to0.
\]
Since the set of Hurwitz matrices is open, there exists
$\Delta t_0>0$ such that $\hat L_1$ and $\hat L_2$ are Hurwitz for every
$\Delta t\in(0,\Delta t_0]$. Hence the pairs
$(\hat A_\beta,\hat B)$ and
$(\hat A+\hat{\bar A}-\frac{\beta}{2}I_d,\hat B)$ are stabilizable.
Thus Assumption~\ref{ass:LQR-MKV} holds for the estimated coefficients.
\end{proof}

 Let $\widehat{\tilde A}:=\hat A+\hat{\bar A}$ and consider $\hat K,\hat\Lambda\in\mathbb S_+^d$ and $\hat R\in\mathbb R$ the solutions of the algebraic system
\begin{empheq}[left=\empheqlbrace]{align}  \beta\hat K  &=  Q+\hat K\hat A+\hat A^\top\hat K  -  \hat K\hat B N^{-1}\hat B^\top\hat K,  \label{eq:LQR-MF-PhiBE-ARE-K}\\  \beta\hat\Lambda  &=  Q+\bar Q  +  \hat\Lambda\widehat{\tilde A}  +  \widehat{\tilde A}^{\top}\hat\Lambda  -  \hat\Lambda\hat B N^{-1}\hat B^\top\hat\Lambda,  \label{eq:LQR-MF-PhiBE-ARE-Lambda}\\  \beta\hat R  &=  -  \frac{\lambda}{2}  \left(    m\log(\pi_0\lambda)-\log\det N  \right).  \label{eq:LQR-MF-PhiBE-ARE-R}
\end{empheq}

The following proposition characterizes the  optimal policy of \eqref{eq:LQR-MF-PhiBE-optimal-value} via the algebraic system \eqref{eq:LQR-MF-PhiBE-ARE-K}-\eqref{eq:LQR-MF-PhiBE-ARE-R}.
\begin{proposition}[Auxiliary control problem]
\label{prop:LQR-MF-PhiBE-auxiliary-control}
Let $\Delta t_0>0$ be given by
\Cref{lem:LQR-MKV-assumption-stability-hat}, and let
$\Delta t\in(0,\Delta t_0]$. Consider the auxiliary control problem
\begin{equation}
\label{eq:LQR-MF-PhiBE-auxiliary-value}
  \widehat{\mathcal V}^{\mathrm{aux}}(\mu)
  :=
  \sup_{\pi\in\Pi_{\mathrm{add}}}
  \mathbb E\left[
    \int_0^\infty e^{-\beta t}
    r_\lambda^\pi(\hat\mu_t)\,dt
    \,\bigg|\,
    \hat\mu_0=\mu
  \right],
  \qquad
  \mu\in\mathcal P_2(\mathbb R^d),
\end{equation}
where $r_\lambda^\pi$ is defined in \eqref{eq:reg_reward_avg_mkv}, with $r$ as in \eqref{eq:reward_LQR_simplify}, and $(\hat\mu_t)_{t\ge0}$ is generated by
\begin{equation}
\label{eq:LQR-MF-PhiBE-auxiliary-dynamics}
  d\hat s_t
  =
  \hat b^\pi(\hat s_t,\hat\mu_t)\,dt,
  \qquad
  \hat s_0\sim\mu,
  \qquad
  \hat\mu_t=\Law(\hat s_t),
\end{equation}
with
\[
  \hat b^\pi(s,\mu)
  :=
  \int_{\mathbb R^m}
  \hat b(s,\mu,a)\,\pi(da\mid s,\mu),
\]
and $\hat{b}$ defined in \eqref{eq:LQR-phibe-drift-estimator}. Then the auxiliary control problem admits an optimal Gaussian feedback
policy
\begin{equation}
\label{eq:LQR-MF-PhiBE-optimal-policy}
  \hat\pi^*(\cdot\mid s,\mu)
  =
  \mathcal N\left(
    -N^{-1}\hat B^\top
    \bigl(
      \hat K(s-\bar\mu)+\hat\Lambda\bar\mu
    \bigr),
    \frac{\lambda}{2}N^{-1}
  \right),
\end{equation}
where $\hat K$ and $\hat\Lambda$ are the stabilizing solutions of
\eqref{eq:LQR-MF-PhiBE-ARE-K} and
\eqref{eq:LQR-MF-PhiBE-ARE-Lambda}, respectively. Moreover, $\widehat{\mathcal V}^{\mathrm{aux}}(\mu)=\hat{\mathcal V}^*(\mu)$ for every $\mu\in\mathcal P_2(\mathbb R^d)$. Consequently, \eqref{eq:LQR-MF-PhiBE-optimal-policy} is an optimal LQR-MF-PhiBE policy.
\end{proposition}

\begin{proof}
By \Cref{lem:LQR-MKV-assumption-stability-hat}, for every $\Delta t\in(0,\Delta t_0]$, \Cref{ass:LQR-MKV} holds with $(A,\bar A,B)$ replaced by $(\hat A,\hat{\bar A},\hat B)$. Moreover, by \eqref{eq:LQR-phibe-drift-estimator}, the auxiliary dynamics \eqref{eq:LQR-MF-PhiBE-auxiliary-dynamics} are an entropy-regularized McKean-Vlasov LQR system with drift matrices $(\hat A,\hat{\bar A},\hat B)$ and zero diffusion. Hence \Cref{prop:LQR-MKV-optimal-policy}, applied to this auxiliary LQR problem, yields the Gaussian feedback \eqref{eq:LQR-MF-PhiBE-optimal-policy} and the optimal value $\widehat{\mathcal V}^{\mathrm{aux}}$.

It remains to identify the auxiliary performance with the LQR-MF-PhiBE
evaluation. Let $\pi$ belong to the policy class specified in
Definition~\ref{def:LQR-MF-PhiBE}, and write
\[
\widehat{\mathcal J}^{\pi}(\mu)
:=
\mathbb E\left[
\int_0^\infty e^{-\beta t}r_\lambda^\pi(\hat\mu_t)\,dt
\,\bigg|\,
\hat\mu_0=\mu
\right].
\]
The auxiliary law flow has generator
$\mathcal L_{\hat b,0}^{\pi}$, so
$\widehat{\mathcal J}^{\pi}$ solves
\[
(\mathcal L_{\hat b,0}^{\pi}-\beta)
\widehat{\mathcal J}^{\pi}+r_\lambda^\pi=0.
\]
If $U\in\mathcal C^2_{\mathrm{poly}}$ is another classical solution,
the chain rule gives, for every $T>0$,
\[
U(\mu)
=
e^{-\beta T}U(\hat\mu_T)
+
\int_0^T e^{-\beta t}r_\lambda^\pi(\hat\mu_t)\,dt.
\]
Polynomial growth and the auxiliary transversality condition imply that
the boundary term converges to zero. Hence
$U=\widehat{\mathcal J}^{\pi}$, and therefore
\[
\hat{\mathcal V}^{\pi}=\widehat{\mathcal J}^{\pi}.
\]
Taking the supremum over the common policy class in
Definition~\ref{def:LQR-MF-PhiBE} gives
\[
\hat{\mathcal V}^*=\widehat{\mathcal V}^{\mathrm{aux}}.
\]
Since \eqref{eq:LQR-MF-PhiBE-optimal-policy} attains the auxiliary
supremum, it is an optimal LQR-MF-PhiBE policy.
\end{proof}

The following lemma gives explicit bounds on the perturbation constants associated with the discounted CARE. These bounds will be used below to display the dependence on $\beta$.

\begin{lemma}[Explicit bounds for the ARE perturbation constants]
    \label{lem:sun_constants_explicit_beta}
    Let $Q_0\in\mathbb S_{++}^d$, $G\in\mathbb S_+^d$, and $A_0\in\mathbb R^{d\times d}$. Let $\beta\ge0$, set $A_\beta:=A_0-\frac{\beta}{2}I_d$, and let $X\in\mathbb S_{++}^d$ be the stabilizing solution of
    \begin{equation}
    \label{eq:generic_discounted_CARE}
        Q_0+A_\beta^\top X+XA_\beta-XGX=0.
    \end{equation}
    Define $\Phi_X:=A_\beta-GX,$ and $\mathcal L_X W:=\Phi_X^\top W+W\Phi_X.$ Let $l_X,p_X,q_X$ be the corresponding perturbation constants in the
    sense of \cite[Theorem~3.1]{Sun1998}, namely
    \[
        l_X:=\|\mathcal L_X^{-1}\|^{-1},\qquad
        p_X:=
        \left\|
        Z\mapsto \mathcal L_X^{-1}(XZ+Z^\top X)
        \right\|,
        \qquad
        q_X:=
        \left\|
        H\mapsto \mathcal L_X^{-1}(XHX)
        \right\|.
    \]
    Then
    \begin{equation}
    \label{eq:sun_pq_bound_explicit_beta}
     l_X
        \ge
        \frac{\lambda_{\min}(X)\lambda_{\min}(Q_0)}{\|X\|^2},\qquad    p_X
        \le
        \frac{2\|X\|^3}
        {\lambda_{\min}(X)\lambda_{\min}(Q_0)},
        \qquad
        q_X
        \le
        \frac{\|X\|^4}
        {\lambda_{\min}(X)\lambda_{\min}(Q_0)}.
    \end{equation}
    Moreover,
    \begin{equation}
    \label{eq:X_min_lower_beta}
        \frac{1}{\lambda_{\min}(X)}
        \le
        \frac{
          \beta+2\|A_0\|+\|G\|\|X\|
        }{
          \lambda_{\min}(Q_0)
        }.
    \end{equation}
    Consequently,
    \begin{equation}
    \label{eq:sun_l_bound_beta_visible}
        l_X
        \ge
        \frac{\lambda_{\min}(Q_0)^2}
        {\|X\|^2\bigl(\beta+2\|A_0\|+\|G\|\|X\|\bigr)},
    \end{equation}
    and
    \begin{equation}
    \label{eq:sun_p_bound_beta_visible}
        p_X
        \le
        \frac{
          2\|X\|^3
          \bigl(\beta+2\|A_0\|+\|G\|\|X\|\bigr)
        }{
          \lambda_{\min}(Q_0)^2
        },
    \end{equation}
    \begin{equation}
    \label{eq:sun_q_bound_beta_visible}
        q_X
        \le
        \frac{
          \|X\|^4
          \bigl(\beta+2\|A_0\|+\|G\|\|X\|\bigr)
        }{
          \lambda_{\min}(Q_0)^2
        }.
    \end{equation}
    \end{lemma}
    
    \begin{proof}
    Since $X$ solves \eqref{eq:generic_discounted_CARE}, we have
    \[
        A_\beta^\top X+XA_\beta-XGX=-Q_0.
    \]
    Therefore, with $\Phi_X=A_\beta-GX$,
    \begin{equation}
    \label{eq:closed_loop_lyapunov_identity_X}
        \Phi_X^\top X+X\Phi_X
        =
        -Q_0-XGX
        \preceq
        -Q_0.
    \end{equation}
    Let $z(t):=e^{\Phi_Xt}z_0$. From
    \eqref{eq:closed_loop_lyapunov_identity_X},
    \[
        \frac{d}{dt}\,z(t)^\top Xz(t)
        =
        z(t)^\top
        \bigl(\Phi_X^\top X+X\Phi_X\bigr)
        z(t)
        \le
        -\lambda_{\min}(Q_0)\|z(t)\|^2.
    \]
    Since $z^\top Xz\le \|X\|\|z\|^2$, it follows that
    \[
        \frac{d}{dt}\,z(t)^\top Xz(t)
        \le
        -
        \frac{\lambda_{\min}(Q_0)}{\|X\|}
        z(t)^\top Xz(t).
    \]
    Thus
    \[
        z(t)^\top Xz(t)
        \le
        e^{-\lambda_{\min}(Q_0)t/\|X\|}
        z_0^\top Xz_0.
    \]
    Using $z(t)^\top Xz(t)\ge\lambda_{\min}(X)\|z(t)\|^2$ and
    $z_0^\top Xz_0\le \|X\|\|z_0\|^2$, we obtain
    \begin{equation}
    \label{eq:semigroup_bound_X_metric}
        \|e^{\Phi_Xt}\|^2
        \le
        \frac{\|X\|}{\lambda_{\min}(X)}
        e^{-\lambda_{\min}(Q_0)t/\|X\|},
        \qquad t\ge0.
    \end{equation}
    
    For every symmetric matrix $Y$, the solution of
    $\mathcal L_XW=Y$ is
    \[
        W
        =
        -\int_0^\infty
        e^{\Phi_X^\top t}
        Y
        e^{\Phi_Xt}\,dt.
    \]
    Therefore, by \eqref{eq:semigroup_bound_X_metric},
    \[
        \|\mathcal L_X^{-1}Y\|
        \le
        \int_0^\infty
        \|e^{\Phi_Xt}\|^2\,dt\,\|Y\|
        \le
        \frac{\|X\|^2}
        {\lambda_{\min}(X)\lambda_{\min}(Q_0)}
        \|Y\|.
    \]
    Hence
    \[
        \|\mathcal L_X^{-1}\|
        \le
        \frac{\|X\|^2}
        {\lambda_{\min}(X)\lambda_{\min}(Q_0)},
    \]
    which gives the first term in \eqref{eq:sun_pq_bound_explicit_beta}.
    
    Next, for any $Z\in\mathbb R^{d\times d}$,
    \[
        \|\mathcal L_X^{-1}(XZ+Z^\top X)\|
        \le
        \|\mathcal L_X^{-1}\|
        \bigl(\|XZ\|+\|Z^\top X\|\bigr)
        \le
        \frac{2\|X\|^3}
        {\lambda_{\min}(X)\lambda_{\min}(Q_0)}
        \|Z\|.
    \]
    This proves the estimate for $p_X$. Similarly, for any
    $H\in\mathbb S^d$,
    \[
        \|\mathcal L_X^{-1}(XHX)\|
        \le
        \|\mathcal L_X^{-1}\|\,\|X\|^2\,\|H\|
        \le
        \frac{\|X\|^4}
        {\lambda_{\min}(X)\lambda_{\min}(Q_0)}
        \|H\|.
    \]
    This proves the estimate for $q_X$.
    
    It remains to prove \eqref{eq:X_min_lower_beta}. Let $v$ be a unit
    eigenvector of $X$ associated with $\lambda_{\min}(X)$. Evaluating
    \eqref{eq:generic_discounted_CARE} on $v$ gives
    \[
        0
        =
        v^\top Q_0v
        +
        2v^\top XA_\beta v
        -
        v^\top XGXv.
    \]
    Since $Xv=\lambda_{\min}(X)v$ and
    $A_\beta=A_0-\frac{\beta}{2}I_d$, we get
    \[
        0
        =
        v^\top Q_0v
        +
        2\lambda_{\min}(X)v^\top A_0v
        -
        \beta\lambda_{\min}(X)
        -
        \lambda_{\min}(X)^2v^\top Gv.
    \]
    Hence
    \[
        v^\top Q_0v
        =
        \beta\lambda_{\min}(X)
        +
        \lambda_{\min}(X)^2v^\top Gv
        -
        2\lambda_{\min}(X)v^\top A_0v.
    \]
    Using $v^\top Q_0v\ge\lambda_{\min}(Q_0)$,
    $v^\top Gv\le\|G\|$, and $-v^\top A_0v\le\|A_0\|$, we obtain
    \[
        \lambda_{\min}(Q_0)
        \le
        \bigl(\beta+2\|A_0\|\bigr)\lambda_{\min}(X)
        +
        \|G\|\lambda_{\min}(X)^2.
    \]
    Since $\lambda_{\min}(X)\le\|X\|$, this implies
    \[
        \lambda_{\min}(Q_0)
        \le
        \bigl(\beta+2\|A_0\|+\|G\|\|X\|\bigr)
        \lambda_{\min}(X),
    \]
    which proves \eqref{eq:X_min_lower_beta}. Finally,
    \eqref{eq:sun_l_bound_beta_visible} follows by combining the first term in 
    \eqref{eq:sun_pq_bound_explicit_beta} with
    \eqref{eq:X_min_lower_beta}, and
    \eqref{eq:sun_p_bound_beta_visible}-\eqref{eq:sun_q_bound_beta_visible}
    follow by combining \eqref{eq:sun_pq_bound_explicit_beta} with
    \eqref{eq:X_min_lower_beta}.
    \end{proof}

\begin{proof}[\textbf{Proof of \Cref{thm:policy_error_phibe_mkv}}]
We use the spectral norm. We first choose $\Delta t_*\in(0,1]$ sufficiently small so that the perturbation estimates below and the smallness conditions in \cite[Theorem~3.1]{Sun1998} and \Cref{lem:LQR-MKV-assumption-stability-hat} hold for every $\Delta t\in(0,\Delta t_*]$. Then \Cref{prop:LQR-MF-PhiBE-auxiliary-control} characterizes the
optimal policy $\hat\pi^*$ by
\eqref{eq:LQR-MF-PhiBE-optimal-policy}. By
\Cref{lem:first_order_coeff_error}, choose a constant $C\ge1$,
independent of $\Delta t\in(0,\Delta t_*]$, such that
\[
\|\hat A-A\|+\|\hat{\bar A}-\bar A\|
\le
\frac{C}{2}\bigl(\|A\|+\|A+\bar A\|\bigr)\Delta t,
\qquad
\|\hat B-B\|
\le
\frac{C}{2}\|B\|\Delta t.
\]
Define $C_K:=2\max\{1,\|K\|\}$,
$C_\Lambda:=2\max\{1,\|\Lambda\|\}$, and
\begin{align}
\label{eq:constant_C_A_b}
      C_{\tilde A}
  :=
  \frac{C}{2}(\|A\|+\|A+\bar A\|),
  \qquad
  C_B
  :=
  \frac{C}{2}\|B\|.
\end{align}
We also set
\begin{equation}
\label{eq:constant_C_K_2}
  C_{K,2}
  :=
  2C_K^3 C_{\tilde A}
  +
  C_K^4
  \Bigl[
    2C_B\|N^{-1}\|\|B\|
    +
    C_B^2\|N^{-1}\|\Delta t_*
  \Bigr],
\end{equation}
\begin{equation}
\label{eq:constant_C_lambd_2}
  C_{\Lambda,2}
  :=
  4C_\Lambda^3 C_{\tilde A}
  +
  C_\Lambda^4
  \Bigl[
    2C_B\|N^{-1}\|\|B\|
    +
    C_B^2\|N^{-1}\|\Delta t_*
  \Bigr],
\end{equation}
and
\begin{equation}
\label{eq:constant_C_K_1_C_Lambd_2}
  C_{K,1}
  :=
  \bigl(2\|A\|+\|BN^{-1}B^\top\|C_K\bigr)C_{K,2},
  \qquad
  C_{\Lambda,1}
  :=
  \bigl(2\|A+\bar A\|+\|BN^{-1}B^\top\|C_\Lambda\bigr)
  C_{\Lambda,2},
\end{equation}
together with $C_{K,3}:=C_BC_K$ and $C_{\Lambda,3}:=C_BC_\Lambda$. The constants above are fixed throughout the proof.

\medskip
\noindent
\textbf{Step 1: coefficient estimates.}
By \Cref{lem:first_order_coeff_error}, and by the definition of $C_{\tilde A}$, we have
\begin{equation}
\label{eq:coeff_errors_beta_explicit_proof}
  \|\hat A-A\|\le C_{\tilde A}\Delta t,
  \qquad
  \|\hat{\bar A}-\bar A\|\le C_{\tilde A}\Delta t,
  \qquad
  \|\hat B-B\|\le C_B\Delta t.
\end{equation}
For the proof only, set
\[
    G:=BN^{-1}B^\top,
    \qquad
    \hat G:=\hat B N^{-1}\hat B^\top.
\]
Then
\[
  \hat G-G
  =
  (\hat B-B)N^{-1}B^\top
  +
  BN^{-1}(\hat B-B)^\top
  +
  (\hat B-B)N^{-1}(\hat B-B)^\top.
\]
Using \eqref{eq:coeff_errors_beta_explicit_proof}, we obtain
\[
  \|\hat G-G\|
  \le
  2C_B\|N^{-1}\|\|B\|\Delta t
  +
  C_B^2\|N^{-1}\|\Delta t^2.
\]
Since $\Delta t\le \Delta t_*$, this yields
\begin{equation}
\label{eq:G_error_beta_explicit_proof}
  \|\hat G-G\|
  \le
  \Bigl[
    2C_B\|N^{-1}\|\|B\|
    +
    C_B^2\|N^{-1}\|\Delta t_*
  \Bigr]\Delta t.
\end{equation}

\medskip
\noindent
\textbf{Step 2: perturbation of the centered Riccati equation.}
Since $D=F=I=0$, \eqref{eq:ARE_SIMPLIFY_RICCATI_K} reduces to
\[
  \beta K
  =
  Q+KA+A^\top K-KBN^{-1}B^\top K.
\]
Equivalently,
\begin{equation}
\label{eq:K_exact_CARE_beta_explicit_proof}
  Q
  +
  \left(A-\frac{\beta}{2}I_d\right)^\top K
  +
  K\left(A-\frac{\beta}{2}I_d\right)
  -
  KGK
  =
  0.
\end{equation}
Similarly, \eqref{eq:LQR-MF-PhiBE-ARE-K} is equivalent to
\begin{equation}
\label{eq:K_hat_CARE_beta_explicit_proof}
  Q
  +
  \left(\hat A-\frac{\beta}{2}I_d\right)^\top\hat K
  +
  \hat K\left(\hat A-\frac{\beta}{2}I_d\right)
  -
  \hat K\hat G\hat K
  =
  0.
\end{equation}
We now apply \cite[Theorem~3.1]{Sun1998}. This theorem compares two algebraic Riccati equations of the form
\[
Q_0+A_0^\top X+XA_0-XG_0X=0
\]
and
\[
(Q_0+\Delta Q)+(A_0+\Delta A)^\top \hat X+\hat X(A_0+\Delta A)-\hat X(G_0+\Delta G)\hat X=0.
\]
For \eqref{eq:K_exact_CARE_beta_explicit_proof}, we take $Q_0=Q$, $A_0=A-\frac{\beta}{2}I_d$, and $G_0=G$. Then \eqref{eq:K_hat_CARE_beta_explicit_proof} corresponds to the coefficient differences
\[
\Delta Q_K:=0,
\qquad
\Delta A_K:=\hat A-A,
\qquad
\Delta G_K:=\hat G-G.
\]
Let $l_K,p_K,q_K$ be the corresponding perturbation constants in \cite[Theorem~3.1]{Sun1998}. Set
\[
\varepsilon_K:=p_K\|\Delta A_K\|+q_K\|\Delta G_K\|,
\qquad
\delta_K:=\|\Delta A_K\|+\|\Delta G_K\|\,\|K\|,
\qquad
\hat g_K:=\|G+\Delta G_K\|.
\]
By \eqref{eq:coeff_errors_beta_explicit_proof} and \eqref{eq:G_error_beta_explicit_proof}, we have $\varepsilon_K\to0$ and $\delta_K\to0$ as $\Delta t\to0$. Hence, after reducing $\Delta t_*$ if necessary, the smallness condition in \cite[Theorem~3.1]{Sun1998} holds. Moreover, the denominator in the bound of \cite[Theorem~3.1]{Sun1998} is
\[
l_K-2\delta_K+
\sqrt{(l_K-2\delta_K)^2-4l_K\hat g_K\varepsilon_K},
\]
which converges to $2l_K$ as $\Delta t\to0$. Reducing $\Delta t_*$ again if necessary, this denominator is bounded from below by $l_K$ for every $\Delta t\in(0,\Delta t_*]$. Therefore \cite[Theorem~3.1]{Sun1998} gives
\begin{equation}
\label{eq:K_error_sun_prelemma}
  \|\hat K-K\|
  \le
  2\left(p_K\|\Delta A_K\|+q_K\|\Delta G_K\|\right).
\end{equation}
By \Cref{lem:sun_constants_explicit_beta},
\[
  p_K
  \le
  \frac{
    2\|K\|^3
    \bigl(\beta+2\|A\|+\|G\|\|K\|\bigr)
  }{
    \lambda_{\min}(Q)^2
  },
\]
and
\[
  q_K
  \le
  \frac{
    \|K\|^4
    \bigl(\beta+2\|A\|+\|G\|\|K\|\bigr)
  }{
    \lambda_{\min}(Q)^2
  }.
\]
Since $\|K\|\le C_K$, $\|G\|=\|BN^{-1}B^\top\|$, and \eqref{eq:coeff_errors_beta_explicit_proof}-\eqref{eq:G_error_beta_explicit_proof} hold, \eqref{eq:K_error_sun_prelemma} gives
\begin{equation}
\label{eq:K_error_beta_explicit_proof}
  \|\hat K-K\|
  \le
  \frac{2}{\lambda_{\min}(Q)^2}
  \bigl(C_{K,1}+\beta C_{K,2}\bigr)\Delta t.
\end{equation}

\medskip
\noindent
\textbf{Step 3: perturbation of the mean-field Riccati equation.}
Since $D=\bar D=F=I=0$, \eqref{eq:ARE_SIMPLIFY_RICCATI_LAMBDA} reduces to
\[
  \beta\Lambda
  =
  Q+\bar Q
  +
  \Lambda\tilde A+\tilde A^\top\Lambda
  -
  \Lambda BN^{-1}B^\top\Lambda,
  \qquad
  \tilde A:=A+\bar A.
\]
Equivalently,
\begin{equation}
\label{eq:Lambda_exact_CARE_beta_explicit_proof}
  Q+\bar Q
  +
  \left(\tilde A-\frac{\beta}{2}I_d\right)^\top\Lambda
  +
  \Lambda\left(\tilde A-\frac{\beta}{2}I_d\right)
  -
  \Lambda G\Lambda
  =
  0.
\end{equation}
The LQR-MF-PhiBE Riccati equation \eqref{eq:LQR-MF-PhiBE-ARE-Lambda} is equivalent to
\begin{equation}
\label{eq:Lambda_hat_CARE_beta_explicit_proof}
  Q+\bar Q
  +
  \left(\hat A+\hat{\bar A}-\frac{\beta}{2}I_d\right)^\top
  \hat\Lambda
  +
  \hat\Lambda
  \left(\hat A+\hat{\bar A}-\frac{\beta}{2}I_d\right)
  -
  \hat\Lambda\hat G\hat\Lambda
  =
  0.
\end{equation}
We apply \cite[Theorem~3.1]{Sun1998} again. For \eqref{eq:Lambda_exact_CARE_beta_explicit_proof}, we take $Q_0=Q+\bar Q$, $A_0=A+\bar A-\frac{\beta}{2}I_d$, and $G_0=G$. Then \eqref{eq:Lambda_hat_CARE_beta_explicit_proof} corresponds to the coefficient differences
\[
\Delta Q_\Lambda:=0,
\qquad
\Delta A_\Lambda:=(\hat A-A)+(\hat{\bar A}-\bar A),
\qquad
\Delta G_\Lambda:=\hat G-G.
\]
By \eqref{eq:coeff_errors_beta_explicit_proof} and \eqref{eq:G_error_beta_explicit_proof}, we have $\|\Delta A_\Lambda\|\le 2C_{\tilde A}\Delta t$ and
\[
  \|\Delta G_\Lambda\|
  \le
  \Bigl[
    2C_B\|N^{-1}\|\|B\|
    +
    C_B^2\|N^{-1}\|\Delta t_*
  \Bigr]\Delta t.
\]
Let $l_\Lambda,p_\Lambda,q_\Lambda$ be the corresponding perturbation constants in \cite[Theorem~3.1]{Sun1998}. Set
\[
\varepsilon_\Lambda:=p_\Lambda\|\Delta A_\Lambda\|+q_\Lambda\|\Delta G_\Lambda\|,
\qquad
\delta_\Lambda:=\|\Delta A_\Lambda\|+\|\Delta G_\Lambda\|\,\|\Lambda\|,
\qquad
\hat g_\Lambda:=\|G+\Delta G_\Lambda\|.
\]
The previous estimates give $\varepsilon_\Lambda\to0$ and $\delta_\Lambda\to0$ as $\Delta t\to0$. Hence, after reducing $\Delta t_*$ if necessary, the smallness condition in \cite[Theorem~3.1]{Sun1998} holds. Moreover, the denominator in the bound of \cite[Theorem~3.1]{Sun1998} is
\[
l_\Lambda-2\delta_\Lambda+
\sqrt{(l_\Lambda-2\delta_\Lambda)^2-4l_\Lambda\hat g_\Lambda\varepsilon_\Lambda},
\]
which converges to $2l_\Lambda$ as $\Delta t\to0$. Reducing $\Delta t_*$ again if necessary, this denominator is bounded from below by $l_\Lambda$ for every $\Delta t\in(0,\Delta t_*]$. Therefore \cite[Theorem~3.1]{Sun1998} gives
\begin{equation}
\label{eq:Lambda_error_sun_prelemma}
  \|\hat\Lambda-\Lambda\|
  \le
  2\left(p_\Lambda\|\Delta A_\Lambda\|+q_\Lambda\|\Delta G_\Lambda\|\right).
\end{equation}
By \Cref{lem:sun_constants_explicit_beta},
\[
  p_\Lambda
  \le
  \frac{
    2\|\Lambda\|^3
    \bigl(\beta+2\|A+\bar A\|+\|G\|\|\Lambda\|\bigr)
  }{
    \lambda_{\min}(Q+\bar Q)^2
  },
\]
and
\[
  q_\Lambda
  \le
  \frac{
    \|\Lambda\|^4
    \bigl(\beta+2\|A+\bar A\|+\|G\|\|\Lambda\|\bigr)
  }{
    \lambda_{\min}(Q+\bar Q)^2
  }.
\]
Since $\|\Lambda\|\le C_\Lambda$ and $\|G\|=\|BN^{-1}B^\top\|$, using \eqref{eq:coeff_errors_beta_explicit_proof} and \eqref{eq:G_error_beta_explicit_proof} in \eqref{eq:Lambda_error_sun_prelemma} gives
\begin{equation}
\label{eq:Lambda_error_beta_explicit_proof}
  \|\hat\Lambda-\Lambda\|
  \le
  \frac{2}{\lambda_{\min}(Q+\bar Q)^2}
  \bigl(C_{\Lambda,1}+\beta C_{\Lambda,2}\bigr)\Delta t.
\end{equation}

\medskip
\noindent
\textbf{Step 4: estimate of the policy means.}
Define
\begin{equation}
\label{eq:def_policy_mean_exact_lqr_phibe}
    m^*(s,\mu)
    :=
    -N^{-1}B^\top
    \bigl(K(s-\bar\mu)+\Lambda\bar\mu\bigr),
    \qquad
    \hat m^*(s,\mu)
    :=
    -N^{-1}\hat B^\top
    \bigl(\hat K(s-\bar\mu)+\hat\Lambda\bar\mu\bigr),
\end{equation}
then we have
\begin{align}
\label{eq:policy_difference_decomposition_beta_explicit_proof}
  \hat m^*(s,\mu)-m^*(s,\mu)
  &=
  N^{-1}(B^\top K-\hat B^\top\hat K)(s-\bar\mu)
  +
  N^{-1}(B^\top\Lambda-\hat B^\top\hat\Lambda)\bar\mu .
\end{align}
For the first coefficient,
\[
  B^\top K-\hat B^\top\hat K
  =
  B^\top(K-\hat K)+(B-\hat B)^\top\hat K.
\]
Using \eqref{eq:coeff_errors_beta_explicit_proof} and \eqref{eq:K_error_beta_explicit_proof},
\begin{align*}
  \|B^\top K-\hat B^\top\hat K\|
  &\le
  \|B\|\,\|\hat K-K\|
  +
  \|\hat B-B\|\,\|\hat K\| \\
  &\le
  2\frac{\|B\|}{\lambda_{\min}(Q)^2}
  \bigl(C_{K,1}+\beta C_{K,2}\bigr)\Delta t
  +
  C_B\Delta t
  \bigl(\|K\|+\|\hat K-K\|\bigr) \\
  &\le
  \left[
    2\frac{\|B\|}{\lambda_{\min}(Q)^2}
    \bigl(C_{K,1}+\beta C_{K,2}\bigr)
    +
    C_{K,3}
  \right]\Delta t
  +
  2\frac{C_B}{\lambda_{\min}(Q)^2}
  \bigl(C_{K,1}+\beta C_{K,2}\bigr)\Delta t^2.
\end{align*}
Similarly,
\[
  B^\top\Lambda-\hat B^\top\hat\Lambda
  =
  B^\top(\Lambda-\hat\Lambda)+(B-\hat B)^\top\hat\Lambda,
\]
and \eqref{eq:Lambda_error_beta_explicit_proof} gives
\begin{align*}
  \|B^\top\Lambda-\hat B^\top\hat\Lambda\|
  &\le
  \left[
    2\frac{\|B\|}{\lambda_{\min}(Q+\bar Q)^2}
    \bigl(C_{\Lambda,1}+\beta C_{\Lambda,2}\bigr)
    +
    C_{\Lambda,3}
  \right]\Delta t
  +
  2\frac{C_B}{\lambda_{\min}(Q+\bar Q)^2}
  \bigl(C_{\Lambda,1}+\beta C_{\Lambda,2}\bigr)\Delta t^2.
\end{align*}
Substituting the last two estimates into \eqref{eq:policy_difference_decomposition_beta_explicit_proof} and taking norms gives exactly \eqref{eq:error_policy_mean_lqr_phibe}.

\medskip
\noindent

\textbf{Step 5: cancellation of the averaged feedback for $\beta=0$.}
Assume now that $\beta=0$ and
$\Lambda\widetilde A=\widetilde A^\top\Lambda$. Set
\[
F_{\Delta t}
:=
\frac1{\Delta t}\int_0^{\Delta t}e^{\widetilde A r}\,dr.
\]
The variation-of-constants formula gives
\[
\hat B=F_{\Delta t}B,
\qquad
\widehat{\widetilde A}:=\hat A+\hat{\bar A}
=\widetilde A F_{\Delta t}=F_{\Delta t}\widetilde A.
\]
Since $F_{\Delta t}\to I_d$, it is invertible for sufficiently small
$\Delta t$. The symmetry assumption and the power-series expansion of
$F_{\Delta t}$ give
\[
\Lambda F_{\Delta t}=F_{\Delta t}^\top\Lambda.
\]
Consequently
\[
P_{\Delta t}:=\Lambda F_{\Delta t}^{-1}
=F_{\Delta t}^{-\top}\Lambda
\]
is symmetric. Since $\Lambda\succ0$ and $P_{\Delta t}\to\Lambda$ as
$\Delta t\to0$, after decreasing $\Delta t_*$ if necessary one also has
$P_{\Delta t}\succ0$. Moreover,
\[
P_{\Delta t}\widehat{\widetilde A}=\Lambda\widetilde A,
\qquad
P_{\Delta t}\hat B=\Lambda B.
\]
Substituting these identities into the exact Riccati equation for
$\Lambda$ shows that
\[
Q+\bar Q
+P_{\Delta t}\widehat{\widetilde A}
+\widehat{\widetilde A}^{\top}P_{\Delta t}
-P_{\Delta t}\hat B N^{-1}\hat B^\top P_{\Delta t}=0.
\]
The corresponding closed-loop matrix is
\[
\widehat{\widetilde A}
-\hat B N^{-1}\hat B^\top P_{\Delta t}
=F_{\Delta t}
\bigl(\widetilde A-BN^{-1}B^\top\Lambda\bigr).
\]
The matrix in parentheses is Hurwitz and $F_{\Delta t}\to I_d$;
hence the displayed matrix is Hurwitz for small $\Delta t$. Thus
$P_{\Delta t}$ is the symmetric stabilizing solution of the estimated
Riccati equation. By uniqueness,
\[
\hat\Lambda=P_{\Delta t}=\Lambda F_{\Delta t}^{-1},
\qquad
\hat\Lambda\hat B=\Lambda B.
\]
Taking transposes and using
$\int(s-\bar\mu)\,\mu(ds)=0$ proves
\[
\mathbb E_{s\sim\mu}
\bigl[\widehat{\bar\pi}^*(s,\mu)-\bar\pi^*(s,\mu)\bigr]=0.
\]

\medskip
\noindent
\textbf{Step 6: Value function error.}
We now estimate the value gap between the optimal LQR policy and the LQR-MF-PhiBE policy, both evaluated in the original dynamics. Let $L:=N^{-1}B^\top K$, $\hat L:=N^{-1}\hat B^\top\hat K$, $U:=N^{-1}B^\top\Lambda$, and $\hat U:=N^{-1}\hat B^\top\hat\Lambda$. By \eqref{eq:def_policy_mean_exact_lqr_phibe}, the feedback means of $\pi^*$ and $\hat\pi^*$ are
\begin{align}
\label{eq:value-step-feedback-means}
m^*(x,\nu)=-L(x-\bar\nu)-U\bar\nu,
\qquad
\hat m^*(x,\nu)=-\hat L(x-\bar\nu)-\hat U\bar\nu.
\end{align}
The two Gaussian policies have covariance $\frac{\lambda}{2}N^{-1}$. Hence, using \eqref{eq:reward_LQR_simplify} and \eqref{eq:reg_reward_avg_mkv}, there exists a constant $c_{\lambda,N}$, independent of $\nu$, such that
\begin{align}
\label{eq:value-step-reward-pistar}
r_\lambda^{\pi^*}(\nu)
=
-\operatorname{Tr}\bigl((Q+L^\top NL)\operatorname{Cov}(\nu)\bigr)
-\bar\nu^\top(Q+\bar Q+U^\top NU)\bar\nu
+c_{\lambda,N},
\end{align}
and
\begin{align}
\label{eq:value-step-reward-pihat}
r_\lambda^{\hat\pi^*}(\nu)
=
-\operatorname{Tr}\bigl((Q+\hat L^\top N\hat L)\operatorname{Cov}(\nu)\bigr)
-\bar\nu^\top(Q+\bar Q+\hat U^\top N\hat U)\bar\nu
+c_{\lambda,N}.
\end{align}

Let $(\hat\mu_t)_{t\ge0}$ be the law flow generated by the original dynamics under $\hat\pi^*$, and set $\hat C_t:=\operatorname{Cov}(\hat\mu_t)$. From \eqref{eq:value-step-feedback-means}, the mean and covariance satisfy
\begin{align}
\label{eq:value-step-flow-pihat}
\dot{\bar{\hat\mu}}_t=(A+\bar A-B\hat U)\bar{\hat\mu}_t,
\qquad
\dot{\hat C}_t=(A-B\hat L)\hat C_t+\hat C_t(A-B\hat L)^\top+\gamma\gamma^\top .
\end{align}
Moreover, by the explicit LQR value formula,
\begin{align}
\label{eq:value-step-optimal-value}
\mathcal V^{\pi^*}(\nu)
=
-\operatorname{Tr}\bigl(K\operatorname{Cov}(\nu)\bigr)
-\bar\nu^\top\Lambda\bar\nu
-R.
\end{align}
Using \eqref{eq:value-step-optimal-value} along the flow $(\hat\mu_t)_{t\ge0}$ and integrating by parts, we obtain
\begin{align}
\label{eq:value-step-integration-by-parts}
\mathcal V^{\pi^*}(\mu)
=
\int_0^\infty e^{-\beta t}
\left[
-\frac{d}{dt}\mathcal V^{\pi^*}(\hat\mu_t)
+\beta\mathcal V^{\pi^*}(\hat\mu_t)
\right]dt .
\end{align}
Therefore, by the definition of $\mathcal V^{\hat\pi^*}$,
\begin{align}
\label{eq:value-step-gap-before-integrand}
\mathcal V^{\pi^*}(\mu)-\mathcal V^{\hat\pi^*}(\mu)
=
\int_0^\infty e^{-\beta t}
\left[
-\frac{d}{dt}\mathcal V^{\pi^*}(\hat\mu_t)
+\beta\mathcal V^{\pi^*}(\hat\mu_t)
-r_\lambda^{\hat\pi^*}(\hat\mu_t)
\right]dt .
\end{align}

We compute the integrand in \eqref{eq:value-step-gap-before-integrand}. Using \eqref{eq:value-step-flow-pihat}, \eqref{eq:value-step-reward-pihat}, and \eqref{eq:value-step-optimal-value}, we get
\begin{align}
\label{eq:value-step-integrand-expanded}
&-\frac{d}{dt}\mathcal V^{\pi^*}(\hat\mu_t)
+\beta\mathcal V^{\pi^*}(\hat\mu_t)
-r_\lambda^{\hat\pi^*}(\hat\mu_t)=
\operatorname{Tr}\Bigl(
\bigl[
(A-B\hat L)^\top K+K(A-B\hat L)-\beta K+Q+\hat L^\top N\hat L
\bigr]\hat C_t
\Bigr)
\notag\\
&\qquad\qquad
+
\bar{\hat\mu}_t^\top
\bigl[
(A+\bar A-B\hat U)^\top\Lambda+\Lambda(A+\bar A-B\hat U)-\beta\Lambda+Q+\bar Q+\hat U^\top N\hat U
\bigr]
\bar{\hat\mu}_t
\notag\\
&\qquad\qquad
-\beta R+\gamma^\top K\gamma-c_{\lambda,N}.
\end{align}
The scalar term in \eqref{eq:value-step-integrand-expanded} vanishes by \eqref{eq:ARE_SIMPLIFY_RICCATI_R}. We next simplify the two matrix residuals in \eqref{eq:value-step-integrand-expanded}. Since \eqref{eq:ARE_SIMPLIFY_RICCATI_K} is equivalent to
\begin{align}
\label{eq:value-step-K-residual-zero}
(A-BL)^\top K+K(A-BL)-\beta K+Q+L^\top NL=0,
\end{align}
inserting $K$ into the closed-loop equation associated with $\hat L$ gives
\begin{align}
\label{eq:value-step-K-residual-expansion}
(A-B\hat L)^\top K+K(A-B\hat L)-\beta K+Q&+\hat L^\top N\hat L=
(A-BL)^\top K+K(A-BL)-\beta K+Q+L^\top NL
\notag\\
&
-(\hat L-L)^\top B^\top K-KB(\hat L-L)+\hat L^\top N\hat L-L^\top NL .
\end{align}
The first line on the right-hand side of \eqref{eq:value-step-K-residual-expansion} vanishes by \eqref{eq:value-step-K-residual-zero}. Since $L=N^{-1}B^\top K$, we have $B^\top K=NL$ and $KB=L^\top N$. Hence
\begin{align}
\label{eq:value-step-K-residual-quadratic}
(A-B\hat L)^\top K+K(A-B\hat L)-\beta K+Q+\hat L^\top N\hat L
=
(\hat L-L)^\top N(\hat L-L).
\end{align}
Similarly, \eqref{eq:ARE_SIMPLIFY_RICCATI_LAMBDA} is equivalent to
\begin{align}
\label{eq:value-step-Lambda-residual-zero}
(A+\bar A-BU)^\top\Lambda+\Lambda(A+\bar A-BU)-\beta\Lambda+Q+\bar Q+U^\top NU=0.
\end{align}
Inserting $\Lambda$ into the closed-loop equation associated with $\hat U$ and using $U=N^{-1}B^\top\Lambda$, we obtain
\begin{align}
\label{eq:value-step-Lambda-residual-quadratic}
(A+\bar A-B\hat U)^\top\Lambda+\Lambda(A+\bar A-B\hat U)-\beta\Lambda+Q+\bar Q+\hat U^\top N\hat U
=
(\hat U-U)^\top N(\hat U-U).
\end{align}
Substituting \eqref{eq:value-step-K-residual-quadratic} and \eqref{eq:value-step-Lambda-residual-quadratic} into \eqref{eq:value-step-integrand-expanded}, and then into \eqref{eq:value-step-gap-before-integrand}, yields
\begin{align}
\label{eq:value-step-quadratic-representation}
\mathcal V^{\pi^*}(\mu)-\mathcal V^{\hat\pi^*}(\mu)
&=
\int_0^\infty e^{-\beta t}
\operatorname{Tr}\bigl((\hat L-L)^\top N(\hat L-L)\hat C_t\bigr)dt
+
\int_0^\infty e^{-\beta t}
\bar{\hat\mu}_t^\top(\hat U-U)^\top N(\hat U-U)\bar{\hat\mu}_t\,dt .
\end{align}
In particular, $\mathcal V^{\pi^*}(\mu)-\mathcal V^{\hat\pi^*}(\mu)\ge0.$ It remains to estimate the right-hand side of \eqref{eq:value-step-quadratic-representation}. From Step 4, we have already proved
\begin{align}
\label{eq:value-step-product-errors}
\|B^\top K-\hat B^\top\hat K\|\le C\Delta t,
\qquad
\|B^\top\Lambda-\hat B^\top\hat\Lambda\|\le C\Delta t .
\end{align}
Therefore, by the definitions of $L,\hat L,U,\hat U$,
\begin{align}
\label{eq:value-step-L-U-errors}
\|\hat L-L\|+\|\hat U-U\|
&\le
\|N^{-1}\|
\left(
\|\hat B^\top\hat K-B^\top K\|
+
\|\hat B^\top\hat\Lambda-B^\top\Lambda\|
\right)
\le C\Delta t .
\end{align}
Because $K$ and $\Lambda$ are the stabilizing solutions of the
discounted CAREs, the shifted matrices
\[
A_\beta-BL,
\qquad
\widetilde A_\beta-BU,
\qquad
A_\beta:=A-\frac\beta2I_d,
\quad
\widetilde A_\beta:=A+\bar A-\frac\beta2I_d,
\]
are Hurwitz. Since $\hat L\to L$ and $\hat U\to U$, after decreasing
$\Delta t_*$ there exist $C,c>0$, independent of $\Delta t$, such that
\[
\|e^{(A_\beta-B\hat L)t}\|
+
\|e^{(\widetilde A_\beta-B\hat U)t}\|
\le Ce^{-ct},
\qquad t\ge0.
\]
The mean equation in \eqref{eq:value-step-flow-pihat} therefore gives
\[
e^{-\beta t}|\bar{\hat\mu}_t|^2
=
\left|
e^{(\widetilde A_\beta-B\hat U)t}\bar\mu
\right|^2
\le
Ce^{-2ct}|\bar\mu|^2.
\]
For the covariance, variation of constants yields
\[
\hat C_t
=
e^{(A-B\hat L)t}\operatorname{Cov}(\mu)
e^{(A-B\hat L)^\top t}
+
\int_0^t
e^{(A-B\hat L)u}\gamma\gamma^\top
e^{(A-B\hat L)^\top u}\,du.
\]
Consequently,
\[
e^{-\beta t}\operatorname{Tr}(\hat C_t)
\le
Ce^{-2ct}\operatorname{Tr}(\operatorname{Cov}(\mu))
+
C|\gamma|^2
\int_0^t e^{-\beta(t-u)}e^{-2cu}\,du.
\]
Since $\beta>0$, integration in time and Fubini's theorem give
\begin{equation}
\label{eq:value-step-second-moment-bound}
\int_0^\infty e^{-\beta t}
\left(
\operatorname{Tr}(\hat C_t)+|\bar{\hat\mu}_t|^2
\right)dt
\le
C\bigl(1+m_2(\mu)\bigr),
\end{equation}
with $C$ independent of $\Delta t$. The same pointwise estimates imply
\[
e^{-\beta t}
\left(
\operatorname{Tr}(\hat C_t)+|\bar{\hat\mu}_t|^2
\right)
\longrightarrow0,
\qquad t\to\infty,
\]
and hence justify the boundary limit used in
\eqref{eq:value-step-integration-by-parts}.
Combining \eqref{eq:value-step-quadratic-representation}, \eqref{eq:value-step-L-U-errors}, and \eqref{eq:value-step-second-moment-bound}, we get
\begin{align*}
0
\le
\mathcal V^{\pi^*}(\mu)-\mathcal V^{\hat\pi^*}(\mu)
&\le
C\left(\|\hat L-L\|^2+\|\hat U-U\|^2\right)
\int_0^\infty e^{-\beta t}
\left(
\operatorname{Tr}(\hat C_t)+|\bar{\hat\mu}_t|^2
\right)dt\le
C\Delta t^2\bigl(1+m_2(\mu)\bigr).
\end{align*}
This proves \eqref{eq:error_value_functions_LQR}.
\end{proof}

\begin{lemma}[Second-moment bound for the MKV-LQR flow when $\beta=0$]
\label{lem:m2-bound-beta-zero-lqr-mkv}
Assume \Cref{ass:LQR-MKV} with $\beta=0$. Let $(K,\Lambda)$ be the stabilizing solution of \eqref{eq:ARE_SIMPLIFY_RICCATI_K}-\eqref{eq:ARE_SIMPLIFY_RICCATI_LAMBDA}, and let $\pi^*$ be the Gaussian feedback in \eqref{eq:LQR-MKV-optimal-policy-candidate}. Let $(s_t,\mu_t)_{t\ge0}$ be the MKV dynamics \eqref{eq:MKV_only_SDE} under $\pi^*$, with $\Law(s_t)=\mu_t$ and $\mu_0\in\mathcal P_2(\mathbb R^d)$. Then there exist constants $C\ge1$ and $\kappa>0$ such that, for every $t\ge0$,
\[
m_2(\mu_t)
\le
C e^{-2\kappa t}m_2(\mu_0)+C|\gamma|^2.
\]
\end{lemma}

\begin{proof}
Let $\bar\mu_t:=\int_{\mathbb R^d}x\,\mu_t(dx)$ and set $\xi_t:=s_t-\bar\mu_t$. Under \eqref{eq:LQR-MKV-optimal-policy-candidate}, $\int_{\mathbb R^m}a\pi^*(da\mid x,\mu)
=
-N^{-1}B^\top\bigl(K(x-\bar\mu)+\Lambda\bar\mu\bigr).$
Since $\sigma(x,\mu,a)=\gamma$, the closed-loop MKV dynamics is
\[
d s_t
=
\bigl[A_K(s_t-\bar\mu_t)+A_\Lambda\bar\mu_t\bigr]dt
+
\gamma\,dW_t,
\]
where $A_K:=A-BN^{-1}B^\top K$ and $A_\Lambda:=A+\bar A-BN^{-1}B^\top\Lambda.$ Taking expectation gives $\dot{\bar\mu}_t=A_\Lambda\bar\mu_t$. Since $\Lambda$ is the stabilizing solution of \eqref{eq:ARE_SIMPLIFY_RICCATI_LAMBDA}, $A_\Lambda$ is Hurwitz. Hence there exist $C_\Lambda\ge1$ and $\kappa_\Lambda>0$ such that
\[
|\bar\mu_t|^2
\le
C_\Lambda e^{-2\kappa_\Lambda t}|\bar\mu_0|^2.
\]

The estimate for $\xi_t$ follows from the same quadratic energy argument used in \cite[Lemma~A.8]{firstpaper}. Indeed, $\xi_t$ solves
\[
d\xi_t=A_K\xi_t\,dt+\gamma\,dW_t,
\]
and the stabilizing property of $K$ gives a Lyapunov function for $A_K$. Applying It\^o's formula to this quadratic Lyapunov function and using Gronwall's inequality yields constants $C_K\ge1$ and $\kappa_K>0$ such that
\[
\mathbb E|\xi_t|^2
\le
C_K e^{-2\kappa_K t}\mathbb E|\xi_0|^2
+
C_K|\gamma|^2.
\]
Finally, since $m_2(\mu_t)=\mathbb E|s_t|^2=\mathbb E|\xi_t|^2+|\bar\mu_t|^2$ and $m_2(\mu_0)=\mathbb E|\xi_0|^2+|\bar\mu_0|^2,$ taking $\kappa:=\min\{\kappa_K,\kappa_\Lambda\}$ and increasing $C$ if necessary gives
\[
m_2(\mu_t)
\le
C e^{-2\kappa t}m_2(\mu_0)+C|\gamma|^2.
\]
\end{proof}

    \bibliographystyle{plain} 
    \bibliography{biblio.bib} 
    
\end{document}